\documentclass[12pt,english]{article}
\usepackage{fancyhdr}
\usepackage{amssymb,amsfonts}
\usepackage{amsmath}
\usepackage{amsthm}
\usepackage{graphics}
\usepackage{graphicx}
\usepackage{subfigure,epsfig,psfrag}
\usepackage{hyperref}
\usepackage{babel}
\usepackage[latin1]{inputenc}
\usepackage[dvipsnames]{xcolor}
\newcommand{\esp}{\hspace{0.06cm}}
\newcommand{\N}{\mathbb N}
\newcommand{\cP}{\mathcal P}

\newcommand{\R}{\mathbb R}
\newcommand{\Z}{\mathbb Z}
\newcommand{\clo}{\mathrm{S}^1}

\newcommand{\dist}{\mathrm{dist}_{\infty}}
\newcommand{\var}{\mathrm{var}}
\newcommand{\id}{\mathrm{id}}
\newcommand{\Diff}{\mathrm{Diff}}
\newcommand{\tf}{\tilde f}
\newcommand{\tX}{\tilde X}
\newcommand{\tY}{\tilde Y}
\newcommand{\ie}{\emph{i.e. }}
\newcommand{\cf}{cf. }

\theoremstyle{theorem}
\newtheorem{thm}{Theorem}[section]
\newtheorem{thmintro}{Theorem}

\newtheorem{prop}[thm]{Proposition}
\newtheorem{lem}[thm]{Lemma}
\newtheorem{qs}[thm]{Question}

\newtheorem{qsintro}{Question}
\newtheorem{corintro}{Corollary}
\theoremstyle{definition}
\newtheorem{rem}[thm]{Remark}

\newtheorem{ex}[thm]{Example}

\begin{document}

\pagestyle{fancy}
\rhead{On diffeomorphisms of the interval}

%\date{}
\author{H\'el\`ene Eynard-Bontemps \, \& \, Andr\'es Navas}

\title{Mather invariant, distortion, and conjugates for diffeomorphisms of the interval}
\maketitle

%%%%%%%%%%%%%%%%%%%%%%%%%%%%%%%%%%%%%%%%%%%%%%%%%%%%%%%%%%%%%%%%%%%%%%

Although the dynamics of (orientation-preserving) homeomorphisms of the interval is quite simple, the behavior in different differentiability 
classes is much harder to understand.  Many fundamental problems in this regard were solved some decades ago ({\em e.g.} existence 
of generating vector fields for germs of diffeomorphisms \cite{sergeraert,szekeres}, triviality of the centralizer of a generic diffeomorphism 
\cite{kopell}, description of complete invariants of $C^1$ conjugacy \cite{yoccoz}, etc). However, several classical questions 
-which are highly sensitive to the differentiability assumptions- remain open. 
Besides, many perspectives have been recently pursued, raising new and challenging problems. 
In particular, over the last years, much attention has been devoted to the study of (the closure of) 
conjugacy classes of diffeomorphisms. Although this can be addressed in any dimension, 
the 1-dimensional case offers a good setting to start with. 

For circle diffeomorphisms, two different cases arise, according to the rotation number of the map.
For an irrational rotation number, the fact that the corresponding rotation is contained in the closure 
of the conjugacy class of the diffeomorphism is known to hold:

\vspace{0.1cm}

\noindent - in the continuous setting (this is an elementary remark that goes back to Herman \cite{herman});

\vspace{0.1cm}

\noindent - in the $C^1$ setting \cite{mio1};

\vspace{0.1cm}

\noindent - in the $C^{1+\mathrm{ac}}$ setting (where ``ac'' stands for {\em absolutely continuous derivative}) \cite{mio2};

\vspace{0.1cm}

\noindent - in the $C^{\infty}$ setting \cite{avila-krikorian}.

\vspace{0.1cm}

\noindent Together with a classical result of Yoccoz (according to which linearizable diffeomorphisms with irrational rotation number are dense  
in the $C^{\infty}$ topology \cite{yoccoz}), these results yield the density of the conjugacy class of a single diffeomorphism in the space of 
diffeomorphisms of the same irrational rotation number in classes $C^0$, $C^1$, $C^{1+\mathrm{ac}}$ and $C^{\infty}$. However, for 
other regularity classes, the situation is less understood. A particularly striking open case is that of $C^2$ regularity, as 
it has been stressed in \cite{mio-ICM,mio-gato}: 

%\vspace{0.5cm}

\begin{qsintro} 
\label{q:conj-cercle}
Does the conjugacy class of a $C^2$ circle 
diffeomorphism with irrational rotation number contain the corresponding rotation in its $C^2$ closure ?
\end{qsintro}

%\vspace{0.5cm}

The case of diffeomorphisms with rational rotation number reduces to that of interval diffeomorphisms (some straightforward issues 
involving periodic points need to be addressed for this reduction). 
In this setting, it is natural to first restrict the study to the space $\mathrm{Diff}_+^{r,\Delta} ([0,1])$ 
of (orientation-preserving) $C^r$ diffeomorphisms with no fixed point in the interior 
(the letter $\Delta$ meaning this latter condition). In this framework, the main question becomes: does the 
identity belong to the closure of the conjugacy class of a given element therein~? The answer is known to be affirmative in three cases:

\vspace{0.1cm}

\noindent - In the $C^0$ setting (this is an easy exercise: any two elements in $\mathrm{Homeo}_+^{\Delta}([0,1])$ are topologically conjugate, the conjugating map being orientation preserving in case the two dynamics ``point in the same direction''); 

\vspace{0.1cm}

\noindent - In the $C^1$ setting,  provided the endpoints are parabolic fixed points (this is an obvious necessary condition); see \cite{mio1}. 

\vspace{0.1cm}

\noindent - In the $C^{1+\mathrm{bv}}$ setting, where a necessary and sufficient condition is the vanishing of the {\em asymptotic distortion} defined as
$$\mathrm{dist}_{\infty} (f) := \lim_{n \to \infty} \frac{1}{n} \mathrm{var} (\log Df^n)$$
(by ``$bv$'' we mean {\em derivative with finite total variation}).  
This was proved in \cite{mio2}, where the notion of asymptotic distortion was introduced as a tool to deal with the conjugacy problem. We refer to \S \ref{s:dist} for a review of this concept.

There is, however, a more classical obstruction to the conjugacy property at least for $C^2$ diffeomorphisms, namely, the {\em Mather invariant}, to which \S \ref{s:mather} is devoted. 
Recall that the Mather invariant $M_f$ of an element $f \in \mathrm{Diff}_+^{2,\Delta} ([0,1])$ is a $C^2$ circle diffeomorphism modulo pre- and post-composition with rotations, which 
is invariant under conjugacy of $f$ by $C^1$ diffeomorphisms and depends continuously on $f$ in the $C^2$ topology in a very broad sense (the latter was proved by Yoccoz in 
\cite{yoccoz}). In particular, this invariant must be trivial (\ie coincide with the class of rotations) if there exists a sequence of conjugates 
$h_n f h_n^{-1}$ by $C^2$ diffeomorphisms $h_n$ that converges to the identity in the $C^2$ topology.

The discussion above suggests a direct relation between the asymptotic distortion and the Mather invariant of elements in $\mathrm{Diff}_+^{2,\Delta} ([0,1])$. The goal of this work is to confirm this claim, to use this relation to better understand the properties of the asymptotic distortion and related phenomena, and finally proceed to a somewhat unexpected application concerning distorted elements in diffeomorphism groups.  Our first result (proven in \S \ref{s:vanish}) is:

\begin{thmintro}
\label{t:vanish}
The asymptotic distortion of an element $f \!\in\! \mathrm{Diff}_+^{2,\Delta} ([0,1])$ 
vanishes if and only if the endpoints are parabolic fixed points of $f$ and the Mather invariant of $f$ is trivial.
\end{thmintro}

As a direct consequence we obtain:

\begin{corintro} 
\label{c:vanish}
The identity is contained in the $C^{1+\mathrm{bv}}$ closure of the conjugacy class of a diffeomorphism 
$f \in \mathrm{Diff}_+^{2,\Delta} ([0,1])$ if and only if the endpoints are parabolic fixed points of $f$ and the Mather invariant of $f$ is trivial.
\end{corintro}

In fact, for an element of $\mathrm{Diff}_+^{2,\Delta} ([0,1])$, having a trivial Mather invariant  
 is equivalent to arising as the time-$1$ map of the flow of a $C^1$ vector field on $[0,1]$ (which is generically not the case 
 in $\mathrm{Diff}_+^{2,\Delta} ([0,1])$; see \cite{kopell}), and the parabolicity of the endpoints is then equivalent to the 
 $C^1$ flatness of the vector field at these points. Stated this way, the necessary and sufficient condition of 
 Theorem \ref{t:vanish} and Corollary \ref{c:vanish} above nicely generalizes to diffeomorphisms having fixed points in the interior:

\begin{corintro} 
\label{c:several}
The asymptotic distortion of a diffeomorphism $f \in \mathrm{Diff}_+^2 ([0,1])$ vanishes if and only if $f$ arises as the time-1 map of the flow of a $C^1$ vector field that is $C^1$ flat at each vanishing point.
\end{corintro}

This follows directly from Theorem \ref{t:vanish} together with a natural formula of localization of the asymptotic distortion along the (smallest) closed intervals that are fixed by the diffeomorphism (\cf Proposition \ref{muchos} in \S \ref{s:dist}), observing in 
addition that if $f$ is the time-$1$ map of a $C^1$ vector field on each of these intervals, these vector fields match up nicely as a $C^1$ vector field on $[0,1]$ (this can be deduced from \cite[chap. IV]{yoccoz}).

\medskip

Another consequence of Theorem \ref{t:vanish} is the answer below to Question 6 from \cite{mio2} (see \S \ref{s:mather} for the proof and a generalization to diffeomorphisms with interior fixed points). 

\begin{corintro} 
\label{c:simult}
If $f,g$ are commuting diffeomorphisms in $\mathrm{Diff}_+^{2,\Delta} ([0,1])$, 
then $\dist (f)$ and $\dist (g)$ either both vanish or are both strictly positive.
\end{corintro}

The relation between the asymptotic distortion and the Mather invariant above can be made even more explicit, as we show in \S \ref{s:fund-ineq}.

\begin{thmintro}
\label{t:fund-ineq}
For every $f  \!\in\! \mathrm{Diff}_+^{2,\Delta} ([0,1])$, one has 
\begin{equation}\label{general-eq}
\big| \mathrm{dist}_{\infty} (f) -  \mathrm{var} ( \log  DM_f ) \big| \leq \big| \log Df(0) \big| + \big| \log Df(1) \big|.
\end{equation}
\end{thmintro}

In the statement above observe that, though the Mather invariant 
$M_f$ is defined only up to pre- and post-composition with rotations, the total variation of the logarithm of its derivative is well defined.

A direct consequence of Theorem \ref{t:fund-ineq} is that if $f  \!\in\! \mathrm{Diff}_+^{2,\Delta} ([0,1])$ is tangent to the identity at the endpoints, then
\begin{equation} \label{parabolic}
\mathrm{dist}_{\infty} (f) =  \mathrm{var} ( \log  DM_f ).
\end{equation}
Notice that this equality implies Theorem \ref{t:vanish}. More interestingly, it implies that $\dist$ is surjective as a function into the non-negative real numbers. Indeed, Yoccoz has shown in \cite{yoccoz} that the Mather invariant is surjective as a function 
from $\mathrm{Diff}_+^{2,\Delta} ([0,1])$ into the space of $C^2$ circle diffeomorphisms modulo rotations. We give a proof ``by hand'' of this fact in Remark \ref{r:mather-local-2}.

Another consequence of Theorem \ref{t:fund-ineq} is the continuity of the asymptotic distortion as a function defined from $\mathrm{Diff}_+^{2,\Delta} ([0,1])$ into the real numbers at each $f$ that is tangent to the identity at the endpoints. Indeed, as we already mentioned, it was also proved by Yoccoz in \cite{yoccoz} that the map sending $f \!\in\! \mathrm{Diff}_+^{2,\Delta} ([0,1])$ to its Mather invariant is continuous for the $C^2$ topology. Moreover, both $Df(0)$ and $Df(1)$ depend continuously on $f$. Therefore, if $f_n \in \mathrm{Diff}_+^{2,\Delta} ([0,1])$, with Mather invariant $M_{f_n}$, converges to a certain $f \in \mathrm{Diff}_+^{2,\Delta} ([0,1])$ that is tangent to the identity at the endpoints, then inequality (\ref{general-eq}) yields
$$\lim_{n \to \infty} \dist (f_n) = \lim_{n\to \infty} \var( \log DM_{f_n} ) = \var ( \log DM_f ) = \dist (f),$$
which shows the continuity of $\dist$ at $f$.\medskip

The argument above requires $C^2$ differentiability and does not apply to diffeomorphisms with hyperbolic fixed points. In \S \ref{s:continu}, we develop a more direct argument that yields the first part of Theorem \ref{t:continu} below, which answers a question from \cite{mio2}. This argument is based on Proposition \ref{cor-dist-loc}, to which \S \ref{s:local} is devoted, concerning the approximation of $\dist$ by the localization of the distortion of a finite iterate of the given diffeomorphism. This new description of the asymptotic distortion is also used at the end of \S \ref{s:local} to give an alternative proof of Theorem \ref{t:fund-ineq}.

\begin{thmintro}
\label{t:continu}
The asymptotic distortion is continuous as a function from $\mathrm{Diff}_+^{1+\mathrm{bv},\Delta} ([0,1])$ into the nonnegative real numbers. 
However, it is not continuous on the space $\mathrm{Diff}_+^{1+\mathrm{bv}} ([0,1])$. Actually, it is not continuous on $\mathrm{Diff}^{\infty}_+([0,1])$.
\end{thmintro}

It will arise along the proof that $\dist$ is discontinuous in a very strong sense. Indeed, there exists a $C^{\infty}$ diffeomorphism $f$ of $[0,1]$ with a single fixed point in the interior as well as a sequence of pairwise $C^{\infty}$ conjugate maps $f_n \!\in\! \mathrm{Diff}_+^{\infty,\Delta} ([0,1])$ that converge in the $C^{\infty}$ topology to $f$ and such that $\, \dist (f_n) = 0 \,$ for all $n$, but $\, \dist (f) > 0$. It is worth stressing that, though $\,\dist\,$ is not continuous on the whole space $\mathrm{Diff}_+^{1+\mathrm{bv}} ([0,1])$, elementary arguments show that it is {\em upper semicontinuous}. In \S \ref{s:continu}, we provide two proofs of this fact: a direct one that only uses the definition of $\dist$, and an alternative one based on considerations related to conjugate diffeomorphisms. 

Since the asymptotic distortion measures the variation of the first derivative, it may be thought of as a kind of Lyapunov exponent at the level of the second derivative. In this view, Theorem \ref{t:continu} should be compared with several results concerning (dis)continuity of Lyapunov exponents in many different settings. (A nice panorama for this important topic is \cite{viana}.) Notice that,  as it is the case of $\dist$, Lyapunov exponents always vary upper semi-continuously. In the task of still pursuing the analogy above, it would be nice to look for examples of discontinuity of Lyapunov exponents along the closure of single conjugacy classes of linear cocycles. Here is a concrete question in this regard.

\begin{qsintro} 
Let $R$ denote a circle rotation by an irrational angle. Does there exists a continuous cocycle 
$A: \clo \to \mathrm{PSL}(2,\mathbb{R})$ with vanishing Lyapunov exponents and a sequence 
of continuous maps $B_n: \clo \to \mathrm{PSL}(2,\mathbb{R})$ such that the conjugate cocycles
$$A_n (\theta) := B_n (R(\theta)) \, A (\theta) \, B_n (\theta)^{-1}$$
converge in $C^0$ topology to a cocycle $A_{\infty}$ having nonzero Lyapunov exponents?
\end{qsintro} 

In \S \ref{s:invariant}, we use Proposition \ref{cor-dist-loc} again to yield a somewhat 
surprising result, namely, the invariance of $\dist$ under $C^1$ conjugacy.

\begin{thmintro}
\label{t:invariant}
 If $f,g$ are $C^2$ diffeomorphisms of a compact 1-manifold for which there exists a $C^1$ diffeomorphism $h$ such that $f = h g h^{-1}$, then $\dist (f) = \dist (g)$.
\end{thmintro}
  
In this work, most of the results concern diffeomorphisms that are $C^2$, and not only $C^{1+\mathrm{bv}}$, yet the asymptotic distortion is naturally defined for the latter class of differentiability. We will deal with this technical issue in the separate work \cite{article2}, where -- among other things -- we will extend the Mather invariant to $C^{1+\mathrm{bv}}$ diffeomorphisms and prove generalized versions of Theorems \ref{t:vanish}, \ref{t:fund-ineq} and \ref{t:invariant}. However, we emphasize that a conjugacy problem 
analog to that of Question \ref{q:conj-cercle} remains widely open for $C^2$ diffeomorphisms of the interval. We explicitly record the precise question in this regard.

\begin{qsintro} 
\label{q:rapproch-C2}
Does there exist $f \in \mathrm{Diff}_+^{2,\Delta} ([0,1])$ with parabolic fixed points 
and trivial Mather invariant whose conjugacy class does not contain the identity in its $C^2$ closure?
\end{qsintro}

In \S \ref{s:reste}, we provide partial answers to this question. In particular, we show:

\begin{thmintro}
\label{t:rapproch-C2}
Let $f \in \Diff^{2,\Delta}_{+}([0,1])$ have parabolic fixed points. If the $C^2$ centralizer of $f$ is not infinite cyclic, then the $C^2$ conjugacy class of $f$ contains the identity in its $C^2$ closure.
\end{thmintro}

As we will recall in \S \ref{s:reste}, the centralizer hypothesis easily implies the triviality of the Mather invariant, but is strictly stronger. More precisely, as we already mentioned, 
having a trivial Mather invariant is equivalent to belonging to the flow of a $C^1$ vector field (a {\em $C^1$ flow}, for short), whose elements then constitute the $C^1$ centralizer of $f$ (\cf \S \ref{s:mather}). However, this flow may be not $C^2$ even if $f$ 
is, and Sergeraert's work \cite{sergeraert} provides an explicit example where the only $C^2$ elements of the flow/centralizer are the (forward and backward) iterates of $f$, and thus for which the \emph{hypothesis} of Theorem \ref{t:rapproch-C2} is not satisfied.  
Despite this, it is worth stressing that, for this particular example, it can be shown by hand that the \emph{conclusion} of  Theorem \ref{t:rapproch-C2} holds. Thus, one must still look elsewhere for a candidate to a negative answer to Question \ref{q:rapproch-C2}.

Under the hypothesis on the centralizer of Theorem \ref{t:rapproch-C2}, another fact holds: one can construct $C^2$ conjugators which bring the $C^1$ generating vector field closer and closer (in the $C^1$ topology) to a $C^2$ vector field, while the convergence for the $C^2$ elements of the initial flow happens in the $C^2$ topology (\cf Proposition \ref{p:C1toC2}). The proof of Theorem \ref{t:rapproch-C2} hence requires explicitly solving the ``$C^2$ conjugacy problem'' for the diffeomorphisms that arise along a $C^2$ flow, and this is done by hand in Proposition \ref{t:champ-C2}. The proof of Proposition \ref{p:C1toC2} is partly inspired from those of previous results on the $C^1$ conjugacy problem for centralizers and flows  which appear at the beginning of \S \ref{s:reste}.

\medskip

Let us also stress that our results only concern the approximation of the identity by conjugates of a given diffeomorphism, and not the approximation of a general diffeomorphism. (A similar remark applies to vector fields.) In concrete terms, we may state the following question:

\begin{qsintro}
\label{q:closure}
 Do there exist elements $f,g$ in $\mathrm{Diff}_+^{1+\mathrm{bv},\Delta} ([0,1])$, or even $\mathrm{Diff}_+^{\infty,\Delta} ([0,1])$, with parabolic fixed points and trivial Mather invariant, such that the $C^{1+\mathrm{bv}}$ closures of their conjugacy classes are different~?
\end{qsintro}

We close this work with a somewhat surprising group-theoretical application. 
Recall that a non-torsion element $f$ of a finitely generated group is said to be {\em distorted} if the word-length of $f^n$ grows sublinearly in $n$. 
(This does not depend on the choice of the finite generating system with respect to which word-lengths are computed.) More generally, an element 
of a general group is said to be a {\em distorted element} if it is distorted inside some finitely generated subgroup of the given group. This notion was 
introduced by Gromov \cite{gromov} and has been recently studied in many different contexts; see for instance \cite{cantat,kra}.

If a $C^1$ diffeomorphism $f$ of $[0,1]$ is a distorted element of $\mathrm{Diff}_+^1 ([0,1])$, then both endpoints of $[0,1]$ must be parabolic fixed 
points of $f$. Indeed, if one of these points were hyperbolic, then the growth of the logarithm of the derivative of the iterates of $f$ at this point would 
be linear; therefore, the word-length of $f^n$ could not grow sublinearly for any finite generating system of a group of $C^1$ diffeomorphisms containing 
$f$. Now, for a $C^2$ diffeomorphism, being a distorted element in $\mathrm{Diff}^{1+\mathrm{bv}}_+ ([0,1])$ requires having a vanishing asymptotic 
distortion (which again implies the parabolicity of the endpoints). This immediately follows from the {\em subadditivity property} of $\var (\log D (\cdot))$, 
namely
\begin{equation}\label{sub-prop}
\var (\log D (gh) ) \leq \var( \log Dg) + \var (\log Dh).
\end{equation}
Thus, Theorem \ref{t:vanish} and Corollary \ref{c:several} yield the following unexpected consequence:

\vspace{0.2cm}

\begin{corintro} 
\label{c:distort}
If $f \!\in\! \mathrm{Diff}^{2,\Delta}_+ ([0,1])$ has a nontrivial Mather invariant, then it cannot be a distorted element of 
$\mathrm{Diff}^{1+\mathrm{bv}}_+ ([0,1])$. More generally, the same holds if $f \!\in\! \mathrm{Diff}^{2}_+ ([0,1])$ 
is not the time-$1$ map of the flow of a $C^1$ vector field.
\end{corintro}

Relevant examples of distorted circle diffeomorphisms are parabolic M\"obius transformations~$f$. Indeed, these are conjugated (via appropriate hyperbolic M\"obius maps $h$) to integer powers of themselves, which allows writing $f^n$ as a product of \, $\sim \! \log (n)$ \, factors among $h^{\pm 1}$ and~$f$. Notice that both $f$ and $h$ may be thought of as diffeomorphisms of the interval by ``opening the circle'' at a common fixed point. This gives examples of distorted elements in $\mathrm{Diff}^2_+ ([0,1])$ (actually, in $\mathrm{Diff}^{\infty}_+([0,1])$). We do not know, however, to what extent the converse of Corollary \ref{c:distort} holds. This seems to be a rather technical 
problem.\footnote{Afther this work, relevant progress in this direction has been achieved first in \cite{Na-dist} and latter in \cite{DE}.}
As a matter of example, remind Sergeraert's example of a $C^{\infty}$ diffeomorphism with trivial Mather invariant and generating vector field no more regular than $C^1$ quoted above. The answer to the following question is unclear to us:

\begin{qsintro}
\label{q:dist}
Is Sergeraert's diffeomorphism undistorted in $\mathrm{Diff}^2_+ ([0,1])$ ? 
\end{qsintro}

It is worth stressing that no analog of Corollary \ref{c:distort} can hold for circle diffeomorphisms. Indeed, there exist $C^{\infty}$ circle diffeomorphisms of irrational rotation number that are distorted yet they do not arise as the time-1 map of a $C^1$ vector field. To see this, notice that a standard Baire's category argument shows that a generic circle diffeomorphism $f$ of irrational rotation number is {\em recurrent}, that is, there exists an increasing sequence of integer numbers $n_k$ such that $f^{n_k}$ converges to the identity as $k \to \infty$ (see \cite[Exercise 5.2.26]{libro}). Moreover, a result of Avila \cite{avila} (see also \cite{CF}) establishes that every recurrent circle diffeomorphism is distorted (see \cite{militon} for a higher dimensional version of this theorem). Nevertheless, a generic circle diffeomorphism of irrational rotation number does not arise as the time-1 map of a $C^1$ vector field (this directly follows from the well-known fact that, generically, the unique invariant probability measure is totally singular with respect to the Lebesgue measure).

Notice that the existence of circle diffeomorphisms that are distorted and do not arise from a vector field implies the existence of diffeomorphisms with the same properties on any higher dimensional manifold of the form $\mathrm{S}^1\times M$ (and, more generally, on any manifold admitting a nontrivial circle action, as for instance $\mathrm{S}^3$; compare \cite{militon}). Do such examples exist on other closed manifolds? What about manifolds with boundary? Does a distorted diffeomorphism fixing the (nonempty) boundary necessarily arise from a vector field, like in the case of the interval?

\vspace{0.1cm}

We close this Introduction with the announcement of another result for which the relation between the Mather invariant and the asymptotic distortion will play a fundamental role. It deals with the question of the connectedness of the space of commuting diffeomorphisms of a 1-dimensional manifold, which was raised by Rosenberg in the 70's (he was particularly interested in the {\em local path-connectedness} of the space of pairs of commuting circle diffeomorphisms). Although the nowadays classical theorems concerning linearization of  circle diffeomorphisms somewhat point in this direction (\cite{yoccoz}; see also \cite{MPHA, KF}), there are only a couple of definitive results regarding Rosenberg's question, and these are very recent. Indeed, the list reduces to: 

\begin{itemize}

\item The space of $\mathbb{Z}^d$ actions by $C^{\infty}$ diffeomorphisms of the interval is connected \cite{bonatti-eynard};

\item The space of $\mathbb{Z}^d$ actions by $C^1$ diffeomorphisms of either the circle or the interval is path-connected \cite{mio1};

\item Any two $\mathbb{Z}^d$ actions by $C^{1+\mathrm{ac}}$ circle diffeomorphisms may be connected by a path 
of $\Z^d$ actions provided one of the generators acts with an irrational rotation number \cite{mio2}. 

\end{itemize}

In \cite{article2}, we prove a general theorem of path-connectedness for the space of $\mathbb{Z}^d$ actions by $C^{1+\mathrm{ac}}$ diffeomorphisms of either the circle or the interval. The proof of this result is a tricky combination of the ideas and techniques from \cite{bonatti-eynard,mio2,mio1} and those of this work, together with further developments that are interesting by themselves. In particular, this requires extending the Mather invariant and its fundamental properties to $C^{1+\mathrm{ac}}$ (and even $C^{1+\mathrm{bv}}$) diffeomorphisms.

\vspace{0.1cm}

It is worth mentioning that several (somewhat technical) questions are spread throughout this text. We hope these will attract  
the interest of the specialists.

%%%%%%%%%%%%%%%%%%%%%%%%%%%%%%%%%%%%%%%%%%%%%%%%%%%%%%%%%%%%%%%%%%%%%%

\section{On the asymptotic distortion}
\label{s:dist}

Let $f$ be a $C^{1+\mathrm{bv}}$ diffeomorphism of $[0,1]$ or the circle (possibly with fixed points in the interior). 
Recall from \cite{mio2} that the {\em asymptotic distortion} of $f$, denoted $\dist (f)$, is defined as 
$$ \lim_{n \to \infty} \frac{1}{n} \var (\log Df^n).$$ 
Notice that by the subadditivity property (\ref{sub-prop}) of $\mathrm{var} (\log D(\cdot))$, we have  
$$\dist (f) = \inf_{n} \frac{1}{n} \var (\log Df^n).$$
Moreover,
\begin{equation}\label{uno}
\dist (f) \leq \var (\log Df).
\end{equation}
Compared to $\mathrm{var} (\log D (\cdot))$, the asymptotic distortion has two new and useful properties that are easy to check and that we record for future use:

\begin{itemize}

\item {\bf {\em Stability:}} $\, \dist (f^n) = |n| \cdot \dist (f) \,$ for all $n \in \mathbb{Z}$ (this immediately follows from the definition and from that $\mathrm{var} (\log Df^{-1} ) = \mathrm{var} (\log Df)$).

\item {\bf {\em Invariance under conjugacy:}} $\, \dist (hfh^{-1}) = \dist (f) \,$ for every $C^{1+\mathrm{bv}}$ diffeomorphism $h$ 
(again, this follows from the definition and the subadditivity inequality $\var(\log D(hf^nh^{-1})) \leq 2 \, \var (\log Dh) + \var (\log Df^n)$; 
see the proof of Proposition \ref{doble-criterio} below for more details).

\end{itemize}

\noindent Nevertheless, $\dist$ is not subadditive. Actually, one may have $\, \dist (f_1) = \dist (f_2) = 0 \,$ yet $\, \dist (f_1 f_2) > 0. \,$ As a matter of example, for the case of the circle, it is easy to produce elliptic M\"obius maps $f_1,f_2$ whose composition $f_1 f_2$ is hyperbolic. The former have vanishing asymptotic distortion since they are conjugate to Euclidean rotations, while the latter has positive asymptotic distortion because it admits hyperbolic fixed points (see the first line of the proof of Proposition \ref{primera-prop}). For the case of the interval, explicit examples of this phenomenon will be presented in Remark \ref{no-sub}. Notice however that, if $f_1$ and $f_2$ commute, then it readily follows from the definition that 
$$\dist (f_1 f_2) \leq \dist (f_1) + \dist (f_2).$$

Most of the discussion in this work focuses on diffeomorphisms of $[0,1]$ with no interior fixed point. In order to treat the case where these points arise, given $f \in \mathrm{Diff}_+^{1+\mathrm{bv}} ([0,1])$ and a subinterval $I \subset [0,1]$, we will denote by $\mathrm{var} (\log Df; I)$ the variation of the logarithm of $Df$ restricted to $I$. With this notation, the subadditivity property of $\mathrm{var} (\log D (\cdot))$ extends to the useful estimate
$$\mathrm{var} \big( \log D (f_1 f_2) ; I \big) 
\leq 
\mathrm{var} \big( \log D f_2 ; I \big) + \mathrm{var} \big( \log D f_1 ; f_2 (I) \big).$$
If $I$ is fixed by $f$, it also makes sense to consider the asymptotic distortion of the restriction of $f$ to $I$, denoted $\dist (f |_I)$. Obviously, $\dist (f |_I) \leq \mathrm{var} (\log Df; I)$. Moreover, we have:

\vspace{0.2cm}

\begin{lem}\label{muchos}
If $f \in \mathrm{Diff}_+^{1+\mathrm{bv}} ([0,1])$ has interior fixed points, then
$$\dist (f) = \sum_{I \in \mathcal{I}} \dist (f|_I),$$
where $\mathcal{I}$ denotes the family of intervals contained in $[0,1]$ that are fixed by $f$ but contain no fixed point in the interior.
\end{lem}

\vspace{0.1cm}

\begin{proof} 
Given $\varepsilon > 0$, we may choose a finite subfamily $\mathcal{I}_{\varepsilon}$ of $\mathcal{I}$ such that 
$$\sum_{I \in \mathcal{I} \setminus \mathcal{I}_{\varepsilon}} \mathrm{var} (\log Df ; I) \leq \varepsilon.$$ 
Since $\var (\log Dg^n; I) \leq n \cdot \mathrm{var} (\log Dg; I)$ holds for every diffeomorphism $g$ fixing the interval $I$ and all $n \geq 1$, this implies that, for all $n \in \mathbb{N}$, the value of 
$$\mathrm{var} (\log Df^n) = \sum_{I \in \mathcal{I}} \mathrm{var} (\log D f^n; I)$$
is $n \varepsilon$-close to the finite sum 
$$\sum_{\mathcal{I}_{\varepsilon}} \mathrm{var} (\log D f^n; I).$$
If we divide by $n$, this becomes 
$$\left| \frac{\mathrm{var} (\log Df^n)}{n} 
- \sum_{I \in \mathcal{I}_{\varepsilon}} \frac{\mathrm{var} (\log D f^n; I)}{n} \right| \leq \varepsilon.$$
Passing to the limit in $n$, this yields
$$\Big| \dist(f) - \sum_{I \in \mathcal{I}_{\varepsilon}} \dist (f|_{I}) \Big| \leq \varepsilon.$$
Finally, letting $\varepsilon$ go to zero, this yields the desired equality. 
\end{proof}

\vspace{0.25cm}

The following crucial result is somewhat contained in \cite{mio2}, but not stated therein in a concise way. For the reader's convenience, we provide a short and direct proof.

\begin{prop} \label{doble-criterio}
One has the equality
\begin{equation}\label{primera}
\mathrm{dist}_{\infty} (f) = \inf_{h \in \mathrm{Diff}_+^{1+\mathrm{bv}}([0,1])} \mathrm{var} \big( \log D (hfh^{-1}) \big).
\end{equation}
Moreover, if $Df$ is absolutely continuous, then
\begin{equation}\label{segunda}
\mathrm{dist}_{\infty} (f) = \inf_{u \in L^1([0,1])} \left\| \frac{D^2 f}{D f} - \big( (u \circ f) \cdot Df - u \big) \right\|_{L^1}.
\end{equation}
\end{prop}

\vspace{0.1cm}

\begin{proof} 
Since $\, (hfh^{-1})^n = hf^nh^{-1} \,$ holds for all $h \in \mathrm{Diff}_+^{1+\mathrm{bv}}([0,1])$, 
the subadditivity property of $\var (\log D (\cdot))$ and the fact that $\var (\log Dh) = \var (\log Dh^{-1})$ yield
$$\var (\log Df^n) - 2 \, \var(\log Dh) \leq \var (\log D (hfh^{-1})^n ) \leq \var (\log Df^n) + 2  \, \var(\log Dh).$$
If we divide by $n$ and pass to the limit, we obtain
\begin{equation}\label{dos}
\dist (f) = \dist (hfh^{-1}).
\end{equation}
In other words, $\dist$ is invariant under conjugacy by  $C^{1+\mathrm{bv}}$ diffeomorphisms. 

Putting (\ref{uno}) and (\ref{dos}) together, we conclude that 
$$\dist (f) = \dist (hfh^{-1}) \leq \var (\log D (hfh^{-1})).$$
Since $h$ above was arbitrary, we obtain
$$\mathrm{dist}_{\infty} (f) \leq \inf_{h \in \mathrm{Diff}_+^{1+\mathrm{bv}}([0,1])} \mathrm{var} (\log D (hfh^{-1})).$$

To prove the reverse inequality, we consider the sequence of diffeomorphisms $h_n$ defined by $h_n (0) = 0$ and 
\begin{equation}\label{def-h_n}
D h_n (x) = \frac{\big[ Df(x) \cdot Df^2 (x) \cdots Df^{n-1}(x) \big]^{1/n}}{\int_0^1 \big[ Df(t) \cdot Df^2(t) \cdots Df^{n-1}(t) \big]^{1/n} dt}.
\end{equation}
Notice that $h_n$ belongs to $\mathrm{Diff}_+^{1+\mathrm{bv}} ([0,1])$. A straightforward computation (that we leave to the reader; alternatively, see the proof of Proposition \ref{calculo-general}) shows that
$$\log \big( D (h_n f h_n^{-1}) (x) \big) = \frac{1}{n} \log \big( Df^n (h_n^{-1} (x)) \big).$$
Therefore, 
$$\var \big( \log D (h_n f h_n^{-1}) \big) = \frac{1}{n} \var (\log Df^n) \stackrel{n \to \infty}{\longrightarrow} \dist (f),$$
which closes the proof of (\ref{primera}).

For the proof of (\ref{segunda}) first remind that, if $g \!\in\! \mathrm{Diff}^{1+\mathrm{bv}}_+ ([0,1])$ has absolutely continuous derivative, then 
$$\var (\log Dg) = \int_0^1 \big| D (\log (Dg)) \big| \, dx = \int_0^1 \left| \frac{D^2 g}{D g} (x) \right| \, dx = \left\| \frac{D^2 g}{D g} \right\|_{L^1}.$$
Next, given $u \in L^1 ([0,1])$, let 
$$I :=  \left\| \frac{D^2 f}{D f} - \big( (u \circ f) \cdot Df - u \big) \right\|_{L^1}.$$
By performing a change of variable we for each $i \geq 1$, we obtain 
$$I = \left\| \frac{D^2 f}{D f} \circ f^i \cdot D f^i - \big( (u \circ f^{i+1}) \cdot Df^{i+1} - u \circ f^i \cdot Df^i \big) \right\|_{L^1}.$$
The triangular inequality and a telescopic sum trick then yield, for each $n \geq 1$, 
$$\left\| \sum_{i=0}^{n-1} \frac{D^2 f}{D f} \circ f^i \cdot Df^i - \big( (u \circ f^{n}) \cdot Df^n - u \big) \right\|_{L^1} \leq nI.$$
Now recall the cocycle identity
\begin{equation}
\label{relacion-cociclo}
\frac{D^2 (g_2 g_1)}{D (g_2 g_1)} = \frac{D^2 g_1}{D g_1} + \frac{D^2 g_2}{D g_2} \circ g_1 \cdot Dg_1,
\end{equation}
which easily yields 
$$\sum_{i=0}^{n-1} \frac{D^2 f}{D f} \circ f^i \cdot Df^i = \frac{D^2 f^n}{D f^n}.$$
Introducing this equality in the previous inequality, dividing by $n$ and using the triangle inequality, we obtain 
$$\frac{1}{n} \left\| \frac{D^2 f^n}{D f^n} \right\|_{L^1} \leq I + \frac{\| u \circ f^n \cdot Df^n - u\|_{L^1}}{n},$$
hence
$$\frac{1}{n} \var (\log Df^n) \leq I + \frac{2 \, \| u \|_{L^1}}{n}.$$
Passing to the limit in $n$, this yields $\, \dist (f) \leq I. \,$ Since this holds for all $u \in L^1([0,1])$, we finally obtain
$$\mathrm{dist}_{\infty} (f) \leq \inf_{u \in L^1([0,1])} \left\| \frac{D^2 f}{D f} - \big( (u \circ f) \cdot Df - u \big) \right\|_{L^1}.$$

To prove the reverse inequality, fix $\varepsilon > 0$. By the first part of the proposition,
 there exists $h \in \mathrm{Diff}^{1+\mathrm{bv}}_+ ([0,1])$ such that 
$$\var \big( \log D (hfh^{-1}) \big) < \dist (f) + \varepsilon.$$
Actually, since $f$ has absolutely continuous derivative, the diffeomorphism $h$ can be chosen satisfying this property as well: this directly follows from the explicit formula (\ref{def-h_n}) for the conjugators. Thus, $hfh^{-1}$ has absolutely continuous derivative, and
$$\var \big( \log D (hfh^{-1}) \big) = \left\| \frac{D^2 (hfh^{-1})}{D (hfh^{-1})} \right\|_{L^1}.$$
Therefore,
\begin{eqnarray*}
\dist (f) + \varepsilon 
&>& \left\| \frac{D^2 h}{D h} \circ (fh^{-1}) \cdot D (fh^{-1}) + \frac{D^2 f}{D f} \circ h^{-1} \cdot D h^{-1}+ \frac{D^2 h^{-1}}{D h^{-1}} \right \|_{L^1} \\
&=&  \left\| \frac{D^2 f}{D f} \circ h^{-1} \cdot D h^{-1} + 
\frac{D^2 h}{D h} \circ (fh^{-1}) \cdot D (fh^{-1}) - \frac{D^2 h}{D h} \circ h^{-1} \cdot Dh^{-1} \right \|_{L^1} \\
&=&  \left\| \frac{D^2 f}{D f} + 
\frac{D^2 h}{D h} \circ f \cdot D f - \frac{D^2 h}{D h} \right \|_{L^1},
\end{eqnarray*}
where we have used the identity 
$$\frac{D^2 h^{-1}}{D h^{-1}} = - \frac{D^2 h}{D h} \circ h^{-1} \cdot D h^{-1}$$
(which easily follows from the cocycle identity (\ref{relacion-cociclo})) and a change of variable equality for the $L^1$-norm. Letting $\, u_h := - D^2 h / D h \,$ (which is an $L^1$-function since $h$ has absolutely continuous derivative), this reads as 
$$\left\| \frac{D^2 f}{D f} - \big( (u_h \circ f) \cdot Df - u_h \big) \right\|_{L^1} < \dist (f) + \varepsilon,$$
thus showing that 
$$\inf_{u \in L^1([0,1])} \left\| \frac{D^2 f}{D f} - \big( (u \circ f) \cdot Df - u \big) \right\|_{L^1} \leq \dist (f) + \varepsilon.$$
Finally, since this holds for all $\varepsilon > 0$, this yields 
$$\inf_{u \in L^1([0,1])} \left\| \frac{D^2 f}{D f} - \big( (u \circ f) \cdot Df - u \big) \right\|_{L^1} \leq \dist (f),$$
which closes the proof. 
\end{proof}

\vspace{0.1cm}

\begin{rem} 
The previous proposition applies verbatim to circle diffeomorphisms. Recall, however, that in case of irrational rotation number, the asymptotic distortion vanishes provided the derivative is absolutely continuous. This is essentially a consequence of the ergodicity with respect to the Lebesgue measure; see \cite{mio2}. It is interesting to compare this result with our Theorem \ref{t:vanish} here in the setting of parabolic fixed points: although maps of the interval are never ergodic, when these arise from $C^1$ vector fields of the closed interval $[0,1]$ (so that the Mather invariant is trivial), we may think that the ``richness'' of the centralizer is somewhat related to the vanishing of the asymptotic distortion. This is, however, just a heuristic view of the phenomenon. 
\end{rem}

\begin{rem} 
\label{conjugators-h_n} 
In the equality (\ref{segunda}) above, we may restrict to functions $u$ that belong to the subspace $L^1_0$ of $L^1$ functions with zero mean in two relevant cases, namely for circle diffeomorphisms, and for diffeomorphisms of the interval for which the endpoints are parabolic fixed points. One way to prove this is just by repeating the argument above inside this space $L^1_0$ starting from the fact that, in both cases, $D^2 f / Df $ belongs to $L^1_0$. The relevance of this arises when noticing that given a sequence of functions $u_n$ realizing the infimum in (\ref{segunda}), a sequence of conjugating diffeomorphisms $h_n$ realizing the infimum in (\ref{primera}) can be obtained by solving the equation $D^2 h_n / D h_n = -u_n$, as one may easily check. (Just reverse some of the previous arguments.) Now, on the one hand, a circle diffeomorphism satisfying this equality only exists in case of zero mean for $u_n$. On the other hand, 
for a diffeomorphism of the interval that is tangent to the identity at the endpoints, zero mean for $u_n$ translates into that the diffeomorphism $h_n$ solving the preceding equation satisfies  $Dh_n (0) = Dh_n (1)$, as it follows from the following explicit expression for $h$ associated to $u \in L^1$:
$$h (x) = \frac{\int_{0}^{x} \exp (\int_0^t -u(s) \, ds) \, dt}{\int_0^1 \exp(\int_{0}^{t} -u(s) \, ds) \, dt}.$$
However, we cannot ensure {\em a priori} that the value of $Dh_n (0) = Dh_n (1)$ can be taken equal to $1$. We formulate this as an explicit question, for which we suspect a negative answer. (See however Remark \ref{raro}.)

\begin{qsintro} 
\label{q:conj-tg}
In the right-hand side expression of equality (\ref{primera}) above, is it possible to restrict to conjugating diffeomorphisms $h$ that are $C^1$ tangent to the identity at the endpoints whenever the starting diffeomorphism $f$ is in $\mathrm{Diff}^{1+\mathrm{bv},\Delta}_+([0,1])$ and has only parabolic fixed points ?
\end{qsintro}

\noindent{\bf {\em On the conjugators $h_n$}.} 
As a kind of illustration, notice that if $f \in \mathrm{Diff}^{1,\Delta}_+ ([0,1])$ has only parabolic fixed points, then the conjugating diffeomorphisms $h_n$ given by (\ref{def-h_n}) have derivative strictly greater than 1 at both endpoints. Indeed, this follows as a direct application of H\"older's inequality just noticing that, for all $k \geq 1$, 
$$\int_0^1 Df^k (t) \, dt = 1.$$
Actually, a stronger fact holds, as it is shown below. 
\end{rem}

\vspace{0.02cm}

\begin{prop} \label{derivada-extremos}
Let $f \in \mathrm{Diff}^{1,\Delta}_+ ([0,1])$ have only parabolic fixed points, and let $h_n$ be the sequence of diffeomorphisms defined by (\ref{def-h_n}), namely,
$$D h_n (x) = \frac{\big[ Df(x) \cdot Df^2 (x) \cdots Df^{n-1}(x) \big]^{1/n}}{\int_0^1 \big[ Df(t) \cdot Df^2(t) \cdots Df^{n-1}(t) \big]^{1/n} dt}.$$
Then the derivative of $h_n$ at each endpoint diverges as $n$ goes to infinity.
\end{prop}

\begin{proof} First notice that 
$$Dh_n (0) = D h_n (1) = \frac{1}{\int_0^1 \big[ Df(t) \cdot Df^2(t) \cdots Df^{n-1}(t) \big]^{1/n} dt}.$$
Fix $p \in (0,1)$, and denote by $I$ the (closed) {\em fundamental interval} of endpoints $p$ and $f(p)$. Then
$$\int_0^1 \big[ Df(t) \cdot Df^2(t) \cdots Df^{n-1}(t) \big]^{1/n} dt 
= 
\sum_{k \in \mathbb{Z}} \int_{f^k (I)} \big[ Df(t) \cdot Df^2(t) \cdots Df^{n-1}(t) \big]^{1/n} dt.$$
By H\"older's inequality, this implies 
\begin{small}\begin{eqnarray*}
\int_0^1 \! \big[ Df(t) \cdots Df^{n-1}(t) \big]^{1/n} dt 
&\leq& \sum_{k \in \mathbb{Z}}
\left[ \int_{f^k(I)} 1 \, dt \right]^{1/n} \left[ \int_{f^k(I)} \! Df(t) \, dt \right]^{1/n} \!\! \cdots \left[ \int_{f^k(I)} \! Df^{n-1}(t) \, dt \right]^{1/n}\\
&=&
\sum_{k \in \mathbb{Z}} \big| f^k(I) \big|^{1/n} \big| f^{k+1}(I) \big|^{1/n} \cdots \big| f^{k+n-1} (I) \big|^{1/n}.
\end{eqnarray*}
\end{small}
The claim of the proposition then follows from the elementary lemma below.
\end{proof}

\begin{lem} 
Let $a_k$ be positive numbers indexed by $k \!\in\! \mathbb{Z}$ whose total sum is finite. If we let 
$$S_n := \sum_{k \in \mathbb{Z}} \big[ a_k \cdot a_{k+1} \cdots a_{k+n-1}\big]^{1/n},$$
then $S_n$ converges to zero as $n$ goes to infinity.
\end{lem}

\begin{proof} 
We will perform the computations along the integers $n$ that are multiples of $3$ to better manipulate the indices,  
leaving the easy modifications to the two other cases as a task for the reader. 
(Alternatively, see Remark \ref{rem-subad} below.) Denote $S:= S_1$, and given $\varepsilon> 0$, fix $N \geq 1$ such that 
\begin{equation}\label{prescribing-epsilon}
\sum_{k=-\infty}^{-N-1} a_k  \leq \frac{\varepsilon^3}{8 \, S^2}, \qquad \qquad 
\sum_{k=N+1}^{\infty} a_k \leq \frac{\varepsilon^3}{8 \, S^2}.
\end{equation}
Starting from 
$$S_n 
<
\sum_{k= - \infty}^{-N - \frac{n}{3}} [a_k \cdots a_{k+n-1}]^{1/n} \,\, + \sum^{\infty}_{k = -N - \frac{n}{3}} [a_k \cdots a_{k+n-1}]^{1/n}$$
and using H\"older's inequality, we obtain that $S_n$ is smaller than the sum of 
$$\left[ \sum_{k= - \infty}^{-N - \frac{n}{3}} \big[ a_k^{1/n} \cdots a_{k+\frac{n}{3} - 1}^{1/n} \big]^{3} \right]^{1/3} 
\left[ \sum_{k= - \infty}^{-N - \frac{n}{3}} \big[ a_{k+\frac{n}{3}}^{1/n} \cdots a_{k + n - 1}^{1/n} \big]^{3/2} \right]^{2/3}$$
and
$$\left[ \sum_{k=-N-\frac{n}{3}}^{\infty} \big[ a_k^{1/n} \cdots a_{k+\frac{2n}{3} - 1}^{1/n} \big]^{3/2} \right]^{2/3} 
\left[ \sum_{k=-N-\frac{n}{3}}^{\infty} \big[ a_{k+\frac{2n}{3}}^{1/n} \cdots a_{k+ n  - 1}^{1/n} \big]^{3}  \right]^{1/3}.$$
Just by changing indices, these expressions may be respectively rewritten as 
$$\left[ \sum_{\ell = - \infty}^{-N - 1} \big[ a_{\ell - \frac{n}{3} + 1}^{3/n} \cdots a_{\ell}^{3/n} \big] \right]^{1/3} 
\left[ \sum_{k= - \infty}^{-N - \frac{n}{3}} \big[ a_{k+\frac{n}{3}}^{3/2n} \cdots a_{k + n - 1}^{3/2n} \big] \right]^{2/3}$$
and
$$\left[ \sum_{k=-N-\frac{n}{3}}^{\infty} \big[ a_k^{3/2n} \cdots a_{k+\frac{2n}{3} - 1}^{3/2n} \big] \right]^{2/3} 
\left[ \sum_{\ell=-N+\frac{n}{3}}^{\infty} \big[ a_{\ell}^{3/n} \cdots a_{\ell + \frac{n}{3} - 1}^{3/n} \big]  \right]^{1/3}.$$
Again, by H\"older's inequality, the first expression is bounded from above by 
$$\prod_{i = 0}^{\frac{n}{3} - 1} \left[ \sum_{\ell = -\infty}^{-N-1} a_{\ell - i} \right]^{1/n} \cdot 
\prod_{j=0}^{\frac{2n}{3} - 1} \left[ \sum_{k = -\infty}^{\infty} a_k \right]^{1/n},$$
which, by (\ref{prescribing-epsilon}), is smaller than or equal to 
$$\left[ \frac{\varepsilon^3}{8 \, S^2} \right]^{1/3} \! \cdot \, S^{2/3} = \frac{\varepsilon}{2}.$$ 
Now, if  we choose $n > 6N$, then the second expression is bounded from above by 
$$\left[ \sum_{k=-N-\frac{n}{3}}^{\infty} \big[ a_k^{3/2n} \cdots a_{k+\frac{2n}{3} - 1}^{3/2n} \big] \right]^{2/3} 
\left[ \sum_{\ell=N+1}^{\infty} \big[ a_{\ell}^{3/n} \cdots a_{\ell + \frac{n}{3} - 1}^{3/n} \big]  \right]^{1/3},$$
and H\"older's inequality shows again that this is smaller than or equal to 
$$S^{2/3} \cdot \left[ \frac{\varepsilon^3}{8 \, S^2} \right]^{1/3}= \frac{\varepsilon}{2}.$$ 
We thus conclude that $S_n < \varepsilon$ for $n$ larger than $6N$ (and multiple of $3$). Since $\varepsilon > 0$ was arbitrary, this closes the proof of the convergence of $S_n$ towards $0$.
\end{proof}

\begin{rem} 
\label{rem-subad} 
In the lemma above, one may easily check that the sequence $\, n \, S_n \, $ is subadditive:
$$(m+n) \, S_{m+n} \leq m \, S_m + n \, S_n.$$
Indeed, for all $k \in \mathbb{Z}$, by the arithmetic-geometric inequality 
$$\frac{b_1+\dots +b_{2n+2m}}{2n+2m}\ge (b_1 \cdot b_2 \cdots  b_{2n+2m})^{1/(2n+2m)}$$
applied to $b_i$'s of the form $[ a_k \cdots a_{k+n-1}]^{1/n}$ ($n$ times), $[ a_{k+m} \cdots a_{k+n+m-1}]^{1/n}$ ($n$ times), $ [a_k \cdots a_{k+m-1} ]^{1/m} $ ($m$ times) and $[ a_{k+n} \cdots a_{k+m+n-1} \big]^{1/m}$ ($m$ times), the value of 
\begin{small}
$$n \, \big( [ a_k \cdots a_{k+n-1} \big]^{1/n} + \big[ a_{k+m} \cdots a_{k+n+m-1} \big]^{1/n} \big) 
\, + \, m \, \big( \big[ a_k \cdots a_{k+m-1} \big]^{1/m} + \big[ a_{k+n} \cdots a_{k+m+n-1} \big]^{1/m} \big)$$  
\end{small}is greater than or equal to 
\begin{align*}
(2n+2m) & \left( [ a_k \cdots a_{k+n-1}]\cdot[ a_{k+m} \cdots a_{k+n+m-1}]\cdot [a_k \cdots a_{k+m-1} ]\cdot[ a_{k+n} \cdots a_{k+m+n-1} \big]\right)^{1/(2n+2m)} \end{align*}
which, after interchanging the second and last terms of the product in the parenthesis, becomes
$$
\big[ (a_k \cdots a_{k+m+n-1})^2 \big]^{1/(2n+2m)} = 
(2n+2m) \, \big[ a_k \cdots a_{k+m+n-1} \big]^{1/(n+m)}.$$
Summing over all $k \in \mathbb{Z}$, this yields
$$2n \, S_n + 2m \, S_m \geq (2n+2m) \, S_{n+m},$$
as announced. Notice that this implies the convergence of $S_n$. However, it seems hard to show that the limit of $S_n$ vanishes just by analyzing the differences between arithmetic and geometric means at each step.
\end{rem}

\vspace{0.35cm}

Another particular feature of the conjugators $h_n$ is given below.

\vspace{0.35cm}

\begin{prop} 
\label{derivada-interior}
Let $f \in \mathrm{Diff}^{1+\mathrm{bv},\Delta}_+ ([0,1])$ have only parabolic fixed points, and let again $h_n$ be the sequence of diffeomorphisms defined by (\ref{def-h_n}). Then the derivative of $h_n$ at each interior point converges to $0$ as $n$ goes to infinity. Actually, the convergence is uniform on compact subsets of $(0,1)$.
\end{prop}

\begin{proof} 
We consider a fundamental interval $I$ as in the proof of Proposition \ref{derivada-extremos}. For each $k \in \mathbb{Z}$ and all $0 \leq i \leq n-1$, there is a point $s = s_{k,i} \in f^k (I)$ such that
$$\frac{|f^{k+i} (I)|}{|f^k(I)|} = Df^i (s).$$
Moreover, for each $t \in f^k (I)$, 
\begin{eqnarray*}
\left| \log \left( \frac{Df^i (t)}{Df^i (s)} \right) \right| 
&=& 
\left| \log \left( \frac{Df (t) \cdot Df (f(t)) \cdots Df (f^{i-1}(t))}{Df (s) \cdot Df (f(s)) \cdots Df (f^{i-1}(s)) } \right) \right| \\
&\leq&
\sum_{j=0}^{i-1} \big| \log Df (f^j(t)) - \log Df (f^j(s)) \big| 
\,\,\,\, \leq \,\,\,\,
\mathrm{var} (\log Df ).
\end{eqnarray*}
Letting $V:= \var (\log Df )$, this yields 
$$ \frac{e^{V} |f^{k+i} (I)|}{|f^k(I)|} = e^{V} Df^i (s) \geq Df^i (t) \geq e^{-V} Df^i (s) = \frac{e^{-V} |f^{k+i} (I)|}{|f^k(I)|}.$$
Therefore, 
\begin{small}\begin{eqnarray*}
\int_0^1 \big[ Df(t) \cdot Df^2(t) \cdots Df^{n-1}(t) \big]^{1/n} dt 
&=& 
\sum_{k \in \mathbb{Z}} \int_{f^k (I)} \big[ Df(t) \cdot Df^2(t) \cdots Df^{n-1}(t) \big]^{1/n} dt \\
&\geq& \sum_{k \in \mathbb{Z}} \! \int_{\! f^k (I)} \frac{e^{-V}}{| f^k (I) |} \big[ |f^k(I)| \cdot |f^{k+1}(I)| \cdots |f^{k+n-1}(I)| \big]^{1/n}\\
&=& e^{-V}  \sum_{k \in \mathbb{Z}} \big[ |f^k(I)| \cdot |f^{k+1}(I)| \cdots |f^{k+n-1}(I)| \big]^{1/n} \!\! .
\end{eqnarray*}\end{small}If we fix $k_0 \in \mathbb{Z}$ and take any point $x \in f^{k_0}(I)$, then this yields, for $a_k := | f^k (I) |$, 
$$D h_n (x) = \frac{\big[ Df(x) \cdot Df^2 (x) \cdots Df^{n-1}(x) \big]^{1/n}}{\int_0^1 \big[ Df(t) \cdot Df^2(t) \cdots Df^{n-1}(t) \big]^{1/n} dt} 
\leq \frac{e^{2V}}{a_{k_0}} \cdot \frac{[a_{k_0} \cdots a_{k_0 + n-1}]^{1/n}}{\sum_{k \in \mathbb{Z}} [a_k \cdots a_{k+n-1}]^{1/n}}.$$
Notice that, since $Df (0) = Df (1) = 1$, the value of $a_{k+1}/ a_{k}$ converges to $1$ as $k$ goes to $\pm \infty$. 
The claim of the proposition then follows from the elementary lemma below.
\end{proof}

\vspace{0.35cm}

\begin{lem} Let $a_k$ be positive numbers indexed by $k \!\in\! \mathbb{Z}$ whose total sum is finite. 
If $a_{k+1} / a_k$ converges to $1$ as $k$ goes to $\pm \infty$, then, for each $k_0 \in \mathbb{Z}$, the value of
$$\frac{[a_{k_0} \cdots a_{k_0 + n-1}]^{1/n}}{\sum_{k \in \mathbb{Z}} [a_k \cdots a_{k+n-1}]^{1/n}}$$
converges to zero as $n$ goes to infinity.
\end{lem}

\begin{proof} 
Denote $c_n := \inf_k \big( a_{k+n} / a_k \big) < 1$. Then 
\begin{eqnarray*}
\frac{[a_{k_0} \cdots a_{k_0 + n-1}]^{1/n}}{\sum_{k \in \mathbb{Z}} [a_k \cdots a_{k+n-1}]^{1/n}} 
&\leq& 
\frac{[a_{k_0} \cdots a_{k_0 + n-1}]^{1/n}}{\sum_{k \geq k_0} [a_k \cdots a_{k+n-1}]^{1/n}}  \\
&=&
\frac{1}{1 + \left[ \frac{a_{k_0 + n}}{a_{k_0}} \right]^{1/n} + \ldots + \left[ \frac{a_{k_0 + n} \cdots a_{k_0 + n + i}}{a_{k_0} \cdots a_{k_0 + i}} \right]^{1/n} + \ldots } \\
&\leq&
\frac{1}{1 + c_n^{1/n} + (c_n^{1/n})^2 + (c_n^{1/n})^3 + \ldots} 
\,\,\,\,\, = \,\,\,\,\,
1-c_n^{1/n}.
\end{eqnarray*}
We are hence left to show that $c_n^{1/n}$ converges to $1$. To do this, fix an arbitrary positive constant $\delta < 1$. Then there exists a positive integer $N = N_{\delta}$ such that \, $a_{i+1} / a_i \geq \sqrt{\delta}$ \, for $i \notin [-N,N]$. Letting $C := \inf_i \big( a_{i+1} / a_i \big)$ this easily yields, for all $n \geq 2N + 1$ and all $k \in \mathbb{Z}$,
$$\frac{a_{k+n}}{a_k} \geq C^{2N+1} \, (\sqrt{\delta})^{n - (2N+1) }.$$
Hence,
$$\sqrt[n]{\frac{a_{k+n}}{a_k}} \geq \sqrt{\delta} \left( \frac{C}{\sqrt{\delta}} \right)^{(2N+1)/n}.$$
Therefore, for a large-enough $n$ (and all $k \in \mathbb{Z}$), we have  
$$\sqrt[n]{\frac{a_{k+n}}{a_k}} \geq \delta.$$
Thus, $c_n^{1/n} \geq \delta$ for all large-enough $n$. Since $\delta < 1$ was arbitrary, this shows the announced convergence \, $c_n^{1/n} \to 1$. 
\end{proof}

\vspace{0.5cm}

\begin{qsintro} Does Proposition \ref{derivada-interior} still hold for $C^1$ diffeomorphisms $f \in \mathrm{Diff}^{1,\Delta}_+([0,1])$~?
\end{qsintro}

\vspace{0.2cm}

\begin{rem} \label{only-one}
Proposition \ref{derivada-extremos} easily implies that the maps $h_n$ above have at least one fixed point in $(0,1)$. Whether or not they may have more fixed points inside (say, for large $n$) is unclear. It would be enlightening to do explicit computations starting with ``degenerate'' maps $f$, 
as those built by Sergeraert in \cite{sergeraert} or, more recently, by Polterovich and Sodin in \cite{Polt-Sod} and by the first-named author in \cite{eynard}. (This seems, however, rather difficult to implement.) A similar remark applies with respect to Proposition \ref{final-prop-extremos} further on.
\end{rem}

\vspace{0.1cm}

In a certain sense, Propositions \ref{derivada-extremos} and \ref{derivada-interior} reflect that the behaviour of the diffeomorphisms $h_n$ 
with respecto to $f$ somewhat mimics the behaviour inside $\mathrm{PSL} (2,\mathbb{R})$ of hyperbolic maps 
with respect to a parabolic map whose fixed point is the repelling fixed point of the hyperbolic ones. Indeed, using 
coordinates on the real line, after conjugacy these become  
$\tilde{h}: x \mapsto x / \lambda $ and $\tilde{f}: x \mapsto x+1$, for which one has 
$$\tilde{h} \circ \tilde{f} \circ \tilde{h}^{-1} (x) = x + \frac{1}{\lambda}.$$
Therefore, $\tilde{h} \circ \tilde{f} \circ \tilde{h}^{-1}$ converges to the identity if $\lambda \to \infty$. (The convergence in this case holds even in the real-analytic topology.) 

This is, however, just a heuristic viewpoint. Indeed, there is no reason to expect that, in the general case, conjugation by the map $h_n$ above will send $f$ to a root of itself, as it may happen that $f$ has no root that is more regular than $C^1$ (this occurs for instance for Sergeraert's examples already quoted). Besides, one can easily check that the conjugators $h_n$ for parabolic elements $f$ of $\mathrm{PSL} (2,\mathbb{R})$ are not contained in $\mathrm{PSL} (2,\mathbb{R})$, and hence differ from the maps $\tilde{h}$ above.

%%%%%%%%%%%%%%%%%%%%%%%%%%%%%%%%%%%%%%%%%%%%%%%%%%%%%%%%%%%%%%%%%%%%%%

\section{On the Mather invariant}
\label{s:mather}

It was dealing with the question of simplicity of diffeomorphisms groups that Mather introduced his invariant as a tool to produce a first (positive) solution in the 1-dimensional case. However, he later turned to a more general construction (which is actually weaker in the 1-dimensional setting), namely the {\em transfer operator}, that allowed him to answer the simplicity question in full generality for compact boundaryless manifolds (with the only exception of the case where the differentiability class equals the dimension of the manifold plus 1, which is still open). This led him to the famous series of papers \cite{mather-1,mather-2,mather-3,mather-4}, yet the original construction from \cite{mather-original} remained unpublished. Yoccoz devoted part of his thesis to deeply study the Mather invariant. Our discussion below is much inspired by Yoccoz' work \cite{yoccoz}.

\vspace{0.3cm}

\noindent{\bf {\em Generating vector fields.}}  Let $f$ be a $C^r$ diffeomorphism of $[0,1)$, $r\ge2$, with no fixed point in the interior. According to Szekeres \cite{szekeres} and Kopell \cite{kopell}, there exists a unique vector field $X = X_f$ on $[0,1)$ such that: 

\vspace{0.1cm}

\noindent - $X$ is $C^1$ on $[0,1)$ (and extends continuously to $[0,1]$);

\vspace{0.1cm}

\noindent - $f$ is the time-1 map of the flow of $X$.

\vspace{0.1cm}

\noindent Moreover, the flow of $X$ coincides with the centralizer of $f$ in $\mathrm{Diff}^1_+ ([0,1))$. Note that if $f_t$, for $t\in\R$, denotes the time-$t$ map of $X$, then $Df_t(0)= \exp(t DX(0))$. In particular, if $f$ is $C^1$ tangent to the identity at $0$, then every element of its $C^1$ centralizer is. 

\vspace{0.1cm}

Furthermore, according to Sergeraert \cite{sergeraert}, $X$ is $C^{r-1}$ on $(0,1)$, but may be no more than $C^1$ at $0$. Indeed, he gives an example of a smooth $f$ for which only the integer times of the flow of $X$ are $C^2$; see also \cite{eynard-these} in this regard.

In the present article, \emph{the existence} of $X$ would be sufficient to carry out most of the proofs. However, having more information, as for instance the estimate \eqref{e:logX} below, helps shorten some proofs and, more importantly, allows generalizations to diffeomorphisms of lower regularity (these will be treated in \cite{article2}). \medskip

\noindent\textbf{Warning.} In order to reduce the amount of notations, in this article we will often identify a vector field $X$ on $[0,1]$ with the \emph{function} $dx(X)$, where $x$ denotes the coordinate on $[0,1]$. With this abuse, given a diffeomorphism $h$, the push-forward $h_*X$ will become the function $(Dh \cdot X)\circ h^{-1}$, and the pull-back $h^*X$ the function $\frac{X\circ h}{Dh}$.\medskip

The expression for $X$ can be made explicit as follows (see \cite{yoccoz} for the details):  assuming that $f (x) > x$ for all $x \in (0,1)$ and letting $\, \Delta (x) := f(x) - x \,$ be the {\em displacement function}, one has
$$X(x) = c_{0}(f) \lim_{n \to \infty} \Delta (f^{-n}(x)) \cdot Df^{n} (f^{-n} (x)) = c_0(f) \lim_{n \to \infty} \frac{\Delta (f^{-n}(x))}{Df^{-n} (x)},$$
where
$$c_{0} (f) =  \left \{
      \begin{array}{rcl}
          \frac{\log (Df(0))}{Df(0)-1}  & \mbox{if} & Df (0) \neq 1, \\
         1 & \mbox{if} & Df (0) = 1. 
      \end{array}
   \right .$$
The role of the constant $c_0$ is to ensure the condition
\begin{equation}\label{eq:int=1}
\int_x^{f(x)} \frac{ds}{ X(s)} = 1,
\end{equation}
which is necessary for $f$ to be the time-1 map of the flow of $X$. The convergence occurs in the $C^1$ topology on every compact subset of $[0,1)$. Notice that $X_n$ and $X$ do not vanish on $(0,1)$, so $1/X_n$ also converges uniformly towards $1/X$ on every compact subset of $(0,1)$. Assuming this, we claim that, for every $a \in (0,1)$,
\begin{equation}\label{e:logX}
\int_a^{f(a)}\left|\frac{DX-\log Df(0)}X\right|\le \var(\log Df ; [0,a]).
\end{equation}
Indeed, letting \, $X_n:= c_0(f) \frac{\Delta\circ f^{-n}}{Df^{-n}}$ \, so that 
$\log X_n = \log(\Delta\circ f^{-n})-\log(Df^{-n}),$ \, we have 
\begin{small}
\begin{align*}
\frac{DX_n}{X_n} = D\log X_n= \frac{(D\Delta)\circ f^{-n}\cdot Df^{-n}}{\Delta\circ f^{-n}}-D\log(Df^{-n})=\frac{c_0(f)(D\Delta)\circ f^{-n}}{X_n}-D\log(Df^{-n}).
\end{align*}
\end{small}Therefore, 
\begin{align*}
\int_a^{f(a)}\left|\frac{DX_n-c_0(f)(D\Delta)\circ f^{-n}}{X_n}\right|&=\int_a^{f(a)}\left|D\log Df^{-n}\right|\\
&=\var \big(\log Df^{-n} ; [a,f(a)] \big)\\
&\le \sum_{i=0}^{n-1}\var \big( \log Df^{-1}\circ f^{-i};[a,f(a)] \big)\\
&= \sum_{i=0}^{n-1}\var \big( \log Df^{-1};[f^{-i}(a),f^{-i+1}(a)] \big)\\
&\le\var \big( \log Df^{-1};[0,f(a)] \big) = \var \big( \log Df;[0,a] \big),
\end{align*}
and \eqref{e:logX} follows by passing to the limit just noticing that $D\Delta\circ f^{-n}$ 
converges uniformly to $D\Delta(0)=Df(0)-1$ on $[a,f(a)]$. 

A useful consequence of \eqref{e:logX} and the equality 
$$\big| \var(\log X ; [a,f(a)]) - \log Df(0) \big|
= \left| \int_a^{f(a)}\left|\frac{DX}X\right| - \int_a^{f(a)} \left| \frac{\log Df(0)}{X}\right| \right| $$
is the estimate
\begin{equation}\label{e:logX2}
\big| \var(\log X ; [a,f(a)]) - \log Df(0) \big| \leq \var (\log Df;[0,a]).
\end{equation}

\medskip

\noindent{\bf {\em Mather invariant as an obstruction to ``flowability''.}} Let now $f$ be a $C^r$ diffeomorphism of the interval $[0,1]$, $r\ge2$, with no fixed point in the interior. The results discussed above, applied to the restrictions of $f$ to $[0,1)$ and $(0,1]$ respectively, provide two ``generating vector fields'' $X$ and $Y$ for $f$ (meaning that $f$ is the time-$1$ map of both), of class $C^1$ on $[0,1)$ and $(0,1]$, respectively. Following Yoccoz \cite{yoccoz}, we will denote by $f_t$ and $f^t$ the time-$t$ maps of the respective flows (so that $f_1=f^1=f$). These are homeomorphisms of $[0,1]$ which restrict to $C^1$ diffeomorphisms of $[0,1)$ and $(0,1]$, respectively, and to $C^{r-1}$-diffeomorphisms of $(0,1)$.

The vector fields $X$ and $Y$ do not necessarily coincide, and the Mather  invariant measures this defect of coincidence. Given points $a,b$ in $(0,1)$, consider the $C^{r-1}$ diffeomorphisms $\psi_X : t\mapsto f_t(a)$ and $\psi_Y : t\mapsto f^t(b)$ from $\mathbb{R}$ into $(0,1)$, and the change of coordinates 
$$M^{a,b}_f := (\psi_Y)^{-1} \circ \psi_X : \mathbb{R} \to \mathbb{R}.$$ 
Notice that 
$$ D M^{a,b}_f (t) 
= \frac{\frac{d \psi_X}{dt}(t)}{\frac{d \psi_Y}{dt} (\psi_Y^{-1} \psi_X (t))} 
= \frac{X (\psi_X (t))}{Y (\psi_Y ( \psi_Y^{-1} \psi_X (t)))},$$
hence 
\begin{equation}\label{derivada-de-M}
D M^{a,b}_f (t) = \frac{X}{Y}( \psi_X (t)).
\end{equation}

The fact that $f$ is the time-1 map of the flows of $X$ and $Y$ implies that $M^{a,b}_f$ commutes with the translation by 1. Hence, it induces a diffeomorphism of the circle $\R/\Z$, which is $C^r$ by (\ref{derivada-de-M}) above, and which we still denote by $M^{a,b}_f$. 

Changing $a$ and $b$ translates into pre/post composition of $M^{a,b}_f$ by rotations. The class of $M^{a,b}_f$ modulo these $\mathrm{SO}(2,\mathbb{R})$-actions (on the left and right) is the Mather invariant of $f$, that we just denote by $M_f$. One says that the Mather invariant is trivial if $M_f$ coincides with the class of rotations. In view of the discussion above, this amounts to saying that $X$ and $Y$ coincide, which is equivalent to that $f$ arises as the time-1 map of the flow of a $C^1$ vector field of $[0,1]$. 

\vspace{0.3cm}

\begin{proof}[Proof of Corollary \ref{c:simult} (assuming Theorem \ref{t:vanish})] Let $f,g$ be commuting elements in $\mathrm{Diff}_+^{2,\Delta} ([0,1])$. Denote by $X_f,Y_f$ (resp. $X_g,Y_g$) the Szekeres vector fields of $f$ (resp. $g$). Since $g$ belongs to the centralizer of $f$ in $\mathrm{Diff}_+^{1} ([0,1])$, it is a flow map of both $X_f$ and $Y_f$ (as a consequence of Kopell's lemma \cite{kopell}, as recalled at the beginning of the section). In other words, $g$ is the time-$1$ map of $\lambda X_f$ and $\lambda Y_f$, with the same  $\lambda \neq 0$, which is nothing but the ``relative translation number'' of $g$ with respect to $f$. By uniqueness of the generating vector fields, we thus have $X_g = \lambda X_f$ and $Y_g = \lambda Y_f$. 

Now according to Proposition \ref{primera-prop}, if $f$ has vanishing asymptotic distortion, its Mather invariant is trivial and both endpoints are parabolic for $f$. This is equivalent to saying that $X_f=Y_f$ and $DX_f(0)=DY_f(1)=0$. Then, according to the above discussion, the same holds for $X_g$ and $Y_g$, and Proposition \ref{segunda-prop} shows that the asymptotic distortion of $g$ also vanishes.
\end{proof}
 
\vspace{0.2cm}

\begin{rem} It is not hard to extend Corollary \ref{c:simult} to commuting diffeomorphisms with interior fixed points provided neither of them acts nontrivially 
on an interval (of positive length) on which the other one acts trivially. (In particular, it holds for any pair of real-analytic commuting diffeomorphisms.) Indeed, in this situation, one can show using Kopell's lemma \cite{kopell} that the two diffeomorphisms have exactly the same fixed points, and the proof then follows by applying the result above on each interval that is fixed by the two diffeomorphisms with no fixed point inside. We leave the details to the reader.
\end{rem}

\vspace{0.1cm}

The next lemma describes a prototypical case where a nontrivial Mather invariant arises. 
Roughly speaking, any $C^2$ perturbation supported on a {\em fundamental interval} does the job.

\vspace{0.1cm}

\begin{lem}
\label{break}
Let $f\in \Diff^{2,\Delta}_+([0,1])$ and $p \in(0,1)$. If the Mather invariant of $f$ is trivial, then any $g\in \Diff^{2,\Delta}_+([0,1])$ that is equal to $f$ outside the interval of endpoints $p,f(p)$ and different from $f$ on this interval has a nontrivial Mather invariant.
\end{lem}

\begin{proof} To fix ideas, assume $f(x)>x$ on $(0,1)$. Since $g$ (resp. $g^{-1}$) coincides with $f$ on $[0,p]$ (resp. with $f^{-1}$ on $[f^2(p),1]$), $X_g$ (resp. $Y_g$) coincides with $X_f$ (resp. with $Y_f=X_f$) on $[0,f(p)]$ (resp. $[f(p),1]$). If $g$ had a trivial Mather invariant, one would have $X_g=Y_g=X_f$ on the whole segment $[0,1]$, so $f$ would be equal to $g$ everywhere, which is not the case.
\end{proof}

\begin{rem}\label{r:mather-local}
In the above setting, there is actually a very simple expression for $M_g$, or rather its representative $M^{p,p}_g$, in terms of the perturbation $h$ (supported in $[p,f(p)]$) applied to $f$ to obtain $g$ ({\em i.e.} such that $g=fh$). Namely, letting (as before) $\psi_{X_f} (t) = f_t(p)$, one has 
$$M^{p,p}_g = \psi_{X_f}^{-1}\circ h \circ \psi_{X_f}.$$ 
In other words, $M^{p,p}_g$ is nothing but $h$ seen in the natural coordinates in which $f$ becomes the unit translation. More generally, if one removes the assumption of vanishing Mather invariant for $f$, one has
$$M^{p,p}_g = M^{p,p}_f\circ \psi_{X_f}^{-1}\circ h \circ \psi_{X_f}=\psi_{Y_f}^{-1}\circ h \circ \psi_{X_f}.$$  
Indeed, first observe that, since $X_f=X_g$ on $[p,f_1(p)]$, for $t \in [0,1]$ one has 
$$\psi_{X_f}(t)=f_t(p)=g_t(p)=\psi_{X_g}(t).$$ 
Similarly, $\psi_{Y_f} (t) = \psi_{Y_g} (t)$ for $t \in [1,2]$. Notice that these maps send $[1,2]$ into $[f(p),f^2(p)]$.
Recall furthermore that $\psi_{X_f}$ and $\psi_{Y_f}$ (resp. $\psi_{X_g}$ and $\psi_{Y_g}$) conjugate the unit translation $T_1$ to $f$ (resp. $g$). Hence, on $[0,1]$,
\begin{align}\label{e:comput}
\psi_{Y_f}^{-1}\circ h \circ \psi_{X_f} &= \psi_{Y_f}^{-1}\circ (f^{-1}\circ g) \circ \psi_{X_g}\notag \\
&= T_{-1} \circ \psi_{Y_f}^{-1}\circ g\circ \psi_{X_g}= T_{-1} \circ \psi_{Y_g}^{-1}\circ \psi_{X_g}\circ  T_1=T_{-1}\circ  M_g^{p,p} \circ T_1=M_g^{p,p}.
\end{align}

Conversely, every $C^2$ diffeomorphism $\Phi$ of the real line commuting with the unit translation and $C^2$-tangent to the identity at $0$ is (a representative of) the Mather invariant of some $g\in \Diff^{2,\Delta}_+([0,1])$. Indeed, just take $f$ and $p$ as above, define $h := \psi_{X_f}\circ \Phi \circ \psi_{X_f}^{-1}$ on $[p,f(p)]$, and finally let $g$ be equal to $fh$ on $[p,f(p)]$ and to $f$ elsewhere. (The ``boundary'' hypothesis on $\Phi$ ensures the regularity of $g$.) \medskip
\end{rem}

\begin{rem}\label{r:mather-local-2}
Following \cite{mather-original}, we now adapt the construction above in order to realize \emph{any} $C^2$ circle diffeomorphism  as (a member of the class of) the Mather invariant of a certain $g\in \Diff^{2,\Delta}_+([0,1])$. In particular, we thus recover the surjectivity of the Mather invariant proved by Yoccoz in \cite{yoccoz} (which together with Theorem \ref{t:fund-ineq} implies the surjectivity of $\dist$, as mentioned in the Introduction). However, our motivation for presenting this explicit construction is rather the content of Remark \ref{r:correction} below. 

Consider a $C^2$ circle diffeomorphism $\phi$ which, up to composing with a rotation, may be assumed to fix $0\in\R/\Z$. Let $\phi_0$ be a circle diffeomorphism coinciding with $\phi$ near $0$ and with the identity near $1 / 2$. Let now $\phi_{1/2} := \phi_0^{-1}\phi$, which thus coincides with $\phi$ near $1/2$ and with the identity near $0$. Now lift $\phi_0$ (resp. $\phi_{1/2}$) to a diffeomorphism of $[\frac12,\frac32]$ 
(resp. $[-1,0]$) by ``opening the circle'' at $1/2$ (resp. $0$), and extend it \emph{by the identity} to a $C^2$ diffeomorphism $\Phi_0$ (resp. $\Phi_{1/2}$) of the real line. Finally, take $f\in\Diff^{2,\Delta}_+([0,1])$ with a trivial Mather invariant and $p\in(0,1)$, consider the associated maps $\psi_{X_f}$ and $\psi_{Y_f}$, and let \, $h_i := \psi_{X_f}\circ \Phi_i \circ \psi_{X_f}^{-1}$ \, for $i=0$ and $1/2$. We claim that the Mather invariant of \, $g := h_0fh_{1/2}$ \, is the class of $\phi$ modulo rotations. 

More precisely, we claim that the $1$-periodic extension of the homeomorphism 
\begin{equation}\label{long}
T_{-3}\circ (\Phi_0\circ T_1\circ \Phi_{1/2})^3
\end{equation}
of $[-1,0]$ is a lift of $\phi$ to $\R$ and a representative of $M_g$. The first claim follows from the definitions of $\Phi_0$ and $\Phi_{1/2}$ and the decomposition of $\phi$ as $\phi_0\phi_{1/2}$, noticing that $\Phi_i$ is often equal to the identity in the long composition (\ref{long}) above. The proof of the second claim is similar to the above computation \eqref{e:comput}. Indeed, observing that $\psi_{X_g}=\psi_{X_f}$ on $[-1,0]$ and $\psi_{Y_g}=\psi_{Y_f}=\psi_{X_f}$ on $[2,3]=(g^3\circ \psi_{X_f})([-1,0])$, one obtains
\begin{align*}
T_{-3}\circ (\Phi_0\circ T_1\circ \Phi_{1/2})^3&=
T_{-3}\circ \psi_{X_f}^{-1}\circ (h_0\circ f\circ h_{1/2})^3\circ \psi_{X_f}\\
&=T_{-3}\circ \psi_{Y_g}^{-1}\circ g^3\circ \psi_{X_g}\\
&=\psi_{Y_g}^{-1}\circ g^{-3}\circ g^3\circ \psi_{X_g} = M_g^{p,p}.
\end{align*}
\end{rem}

%\marginpar{\textcolor{red}{Le rapporteur demande de combiner les deux dernieres remarqes mais on avait discute cela et on avait conclu que cela faisait une remarque trop longue...}}

\begin{rem}
\label{r:correction} In the remark above, we started with an $f$ with trivial Mather invariant and perturbed it locally (on the union of three consecutive fundamental intervals) to obtain a $g$ with a prescribed (nontrivial) Mather invariant. Conversely, one can start with \emph{any} $f\in\Diff^{2,\Delta}_+([0,1])$ and perform an explicit local perturbation to adjust the Mather invariant as desired, in particular to cancel it. Indeed, given $p\in(0,1)$, define the circle diffeomorphism $\phi$ so that $\phi^{-1}$ is a representative of $M_f^{p,p}$ on the circle, and construct $g$ just as above. In the last computation, $\psi_{X_f}^{-1}$ must now be replaced by $(M_f^{p,p})^{-1}\psi_{Y_f}^{-1}=(M_f^{p,p})^{-1}\psi_{Y_g}^{-1}$ rather than just $\psi_{Y_g}^{-1}$, and one thus gets 
$$T_{-3}\circ (\Phi_0\circ T_1\circ \Phi_{1/2})^3=(M_f^{p,p})^{-1} M_g^{p,p}.$$
Again, the expression on the left is a lift of $\phi$ to the real line fixing $0$, while $M_f^{p,p}$ is a lift of $\phi^{-1}$. Thus, in the end, $M_g^{p,p}=\id$.
 \end{rem}

\noindent{\bf {\em Mather invariant and conjugacy.}} The Mather \emph{invariant} is thus named because it is invariant under $C^1$ conjugacy (among $C^2$ diffeomorphisms), as it readily follows from the definition. More importantly, together with the $C^1$ conjugacy classes of the germs at the endpoints, it constitutes a \emph{complete} $C^1$ conjugacy invariant: for $r \geq 2$, two elements of $\Diff^{r,\Delta}_+([0,1])$ are $C^1$ conjugated if and only if their germs at the endpoints are $C^1$ conjugated and they have the same Mather invariant \cite{yoccoz}. Unfortunately, $C^1$ cannot be replaced by $C^r$ in the previous statement, as will be shown in Remark \ref{r:Cr-conj} below. Nevertheless, the following \emph{sufficient condition} for $C^r$ conjugacy will be useful later in this article. 

\vspace{0.1cm}

\begin{lem} \label{l:Cr-conj} 
Let $f$ and $g$ in $\Diff^{r,\Delta}_{+}([0,1])$ coincide near $0$ and $1$ and have a trivial Mather invariant, where $r \geq 2$. If, for some $a\in(0,1)$, the orbits of $a$ under $f$ and $g$ coincide near $0$ and $1$, {\em or} if the generating vector fields of $f$ and $g$ are $C^r$ near $0$ {\em or} $1$, then $f$ and $g$ are $C^r$ conjugated. 
\end{lem}

\begin{proof} Without loss of generality (replacing $f$ and $g$ by $f^{-1}$ and $g^{-1}$ if necessary), we can assume that $f (x) > x$ for all $x \in (0,1)$. If the orbit condition is satisfied, let $a$ be as in the statement. Otherwise, fix any $a\in(0,1)$, and assume, to fix ideas, that the Szekeres vector fields of $f$ and $g$ are $C^r$ near $1$. Picking another point of its orbit if necessary, we assume that $a$ and $f(a)$ belong to an interval containing $0$ on which $f$ and $g$ coincide. Let $X_f$ and $X_g$ be the Szekeres vector fields of $f$ and $g$, respectively (since the Mather invariants are trivial, both $f$ and $g$ have a single Szekeres vector field, which is $C^1$ on $[0,1]$ and $C^{r-1}$ on $(0,1)$). Denote by $(f^t)_{t\in\R}$ and $(g^t)_{t\in\R}$ their flows, and consider the two $C^r$-diffeomorphisms from $\R$ to $(0,1)$ defined by $\psi_f:t\mapsto f^t(a)$ and $\psi_g:t\mapsto g^t(a)$. (The $C^r$ regularity of $\psi_f$ comes from the fact that $\psi_f$ is $C^{r-1}$ since $X_f$ is $C^{r-1}$ on $(0,1)$, and thus $d\psi_f / dt = X_f \circ \psi_f$ is $C^{r-1}$; similar arguments apply to $\psi_g$.) These respectively conjugate $f$ and $g$ to the unit translation on $\R$. 

Hence, the map $\psi=\psi_g\circ \psi_f^{-1}$ is a $C^r$ conjugacy from $f$ to $g$ (\emph{i.e.} $\psi \circ f \circ \psi^{-1}=g$) on $(0,1)$. We still denote by $\psi$ its extension as a homeomorphism of $[0,1]$. Since $f$ and $g$ coincide on $[0,f(a))$, so do their Szekeres vector fields. Thus, $\psi_g=\psi_f$ on $(-\infty,0]$ (onto $(0,a]$), and hence $\psi=\id$ on $[0,a]$. Therefore, $\psi$ is $C^r$ on $[0,1)$. 

Now from $\psi \circ f \circ \psi^{-1} = g$ and the uniqueness of Szekeres vector fields, one actually gets $\psi_*(X_f) = X_g$. Since $X_f=X_g$ near $1$ (recall these are also the right Szekeres vector fields of $f$ and $g$, which coincide near $1$!), this gives 
$\psi_*(X_f) = X_f$ near $1$. This is equivalent to saying that $\psi$ belongs to the flow of $X_f$, that is, $\psi = f^{t_0}$ near $1$ for some $t_0 \in \R$.  

If $X_f$ is $C^r$ near $1$, we are done. Otherwise, equality $\psi = f^{t_0}$ (near 1) already shows that $\psi$ is $C^1$ near $1$. To conclude the proof, it suffices to show that, under the orbit condition, $t_0$ must be an integer, and thus $\psi$ is $C^r$ near $1$ because it is an iterate of $f$ or its inverse. Indeed, $\psi$ sends the orbit of $a$ under $g$ to the orbit of $\psi(a)=a$ under $f$. In particular, if the two orbits coincide in a neighborhood of $1$, the map $\psi$ sends each $f^n(a)$ therein to some point $f^m(a)$. This implies that $t_0 = m-n\in\Z$, as we wanted to check. 
\end{proof}

\vspace{0.1cm}

\begin{rem}\label{r:Cr-conj}
Unfortunately, the statement of Lemma \ref{l:Cr-conj} is no longer true if one replaces ``have a trivial Mather invariant'' by ``have the same Mather invariant''. As already mentioned, Mather's theory says that, in this case, $f$ and $g$ are $C^1$-conjugated, and the conjugacy, which can be described explicitly, is at least as smooth as the Szekeres vector fields of $f$ and $g$. However, these can be just $C^1$ at the endpoints, and in this case the conjugacy may be no more regular than $C^1$.  

Here is an example where this fact actually happens: Start with the time-$1$ map $f$ of a $C^1$ vector field $X$ that is smooth and positive on $(0,1)$ and coincides near $0$ and $1$ with a ``Sergeraert counterexample''. The latter means that the only $C^2$ times of its local flows are the integers; see \cite{eynard-these, sergeraert}. Fix $a\in(0,1)$, and perturb $f$ locally on $I:= [a,f(a)]$ to obtain a new diffeomorphism $\tilde f$ equal to $f$ outside $I$ and to $ f\circ \phi$ on $I$, for some smooth diffeomorphism $\phi$ of $[0,1]$ supported on $I$. According to Lemma \ref{break}, 
the Mather invariant of $\tf$ is nontrivial. We will denote by $\tf_t$ and $\tf^t$ the time-$t$ maps associated to the left and right Szekeres vector fields of $\tf$, denoted $\tilde{X}$ and $\tilde{Y}$, respectively. We will need an additional assumption on $\phi$, namely, that $\phi$ fixes $f^\alpha(a)$ for some irrational $\alpha\in(0,1)$.  

Now let $g := f^{-\alpha}\circ \tf\circ f^\alpha$. One has $\tf=f=g$ near the endpoints (in particular, though $f^\alpha$ is only $C^1$, $g$ is smooth on $[0,1]$), and $M_{\tf} = M_g$ since $\tf$ and $g$ are $C^1$-conjugated. Besides, the equality $\tf=f$ on $[0,1]\setminus I$ implies the equality of $X$ and $\tX$ on $[0,f(a)]$ and of $X$ and $\tY$ on $[f(a),1]$. As a consequence, $f^{\alpha} = \tf_{\alpha}$ on $[0,a]$, and $f^{-\alpha} = \tf^{-\alpha}$ on $[f^{1+\alpha}(a),1]$. Let us now check that the orbits of $a$ under $\tf$ and $g$ respectively coincide near $0$ and $1$. For this, it is sufficient to have $\tf_{\alpha}(a)=\tf^{\alpha}(a)$. Indeed, if so, for all $n \geq 1$,
$$g^n(a)=(f^{-\alpha}\tf^n f^\alpha)(a)= (\tf^{-\alpha}\tf^n \tf_\alpha)(a)=(\tf^n\tf^{-\alpha}\tf_\alpha)(a)=\tf^n(a).$$ 

Now equality $\tf_{\alpha}(a)=\tf^{\alpha}(a)$ comes from the additional hypothesis on $\phi$. Indeed, on $(0,1]$, there is a unique $C^\infty$ diffeomorphism $\Phi$ equal to the identity on $[f(a),1]$ such that $\tf=\Phi^{-1}\circ f\circ \Phi$. Namely, $\Phi$ is defined as $\Phi = f^{-n}\phi f^n$ on $[f^{-n}(a),f^{-n+1}(a)]$ for all $n \geq 0$ (in particular, $\Phi$ fixes the whole orbit of $a$). Then $\Phi$ necessarily conjugates the right Szekeres vector fields of $f$ and $\tf$ and their flows. In particular, 
$$\tf^\alpha(a)=(\Phi^{-1}\circ f^\alpha\circ \Phi)(a)=(\Phi^{-1}\circ f^\alpha)(a) = \phi^{-1}(f^\alpha(a))=f^\alpha(a)=\tf_\alpha(a).$$

Let us finally show that $g$ is not $C^2$-conjugated to $\tf$. If there existed such a conjugacy $h$, then $f^\alpha h^{-1}$ would commute with $\tf$, and since this composition is $C^1$, it would be an element of both Szekeres flows of $\tf$, say $f^{\alpha} h^{-1} = \tf^t = \tf_{t'}$. From $\tf^1 = \tf_1 = \tf$, it is not hard to see that $t=t'$, hence $f^\alpha=\tf^t\circ h=\tf_t\circ h$. In particular, near $0$, one would have $h=f_{\alpha - t}$, and near $1$, one would have $h=f^{\alpha - t}$. Since we are assuming that for both germs of flow of $f$ the only smooth times are the integers, $\alpha - t$ must be an integer, and thus $t$ must be irrational. But then $\tf^t=\tf_t$ and $\tf^1=\tf_1$ would imply that the left and right Szekeres flows coincide, 
which contradicts the non-triviality of the Mather invariant of $\tf$. 
\end{rem}

%%%%%%%%%%%%%%%%%%%%%%%%%%%%%%%%%%%%%%%%%%%%%%%%%%%%%%%%%%%%%%

\section{Trivial Mather invariant: a proof of Theorem \ref{t:vanish}}
\label{s:vanish}

We next proceed with a direct proof of Theorem \ref{t:vanish} that will give us some insight to later proceed with the (more general) Theorem \ref{t:fund-ineq}. We start with the following implication:
\vspace{0.2cm}

\begin{prop} \label{primera-prop} 
Let $f$ be an element of $\mathrm{Diff}^{2,\Delta}_+ ([0,1])$. If the asymptotic distortion of $f$ vanishes, then $0$ and $1$ are parabolic fixed points for $f$, and its Mather invariant is trivial.
\end{prop}

\begin{proof}
The first part holds more generally for $f \in \mathrm{Diff}_+^{1+\mathrm{bv},\Delta} ([0,1])$. Indeed, since 
$$\mathrm{\var} (\log Df^n) \geq \big| \log Df^n (1) - \log Df^n (0) \big| = n \, \big| \log Df (1) - \log Df (0) \big|$$
and since $\log Df(0)$ and $\log Df(1)$ have ``different signs'' (in a broad sense), we have
\begin{equation}\label{siempre}
\dist (f) \geq \big| \log Df (0) \big| + \big| \log Df (1) \big|.
\end{equation}
As for the second part, we will actually prove a stronger statement, namely
\begin{equation}\label{e:ineg1}
\var(\log DM_f)\le |\log Df(0)|+|\log Df(1)|+\dist(f),
\end{equation}
which will also imply half of Theorem \ref{t:fund-ineq}. Again, notice that $D M_f$ is not properly defined, but \, $\var(\log DM_f)$ \, is. Indeed, 
$\var(\log DM^{a,b}_f)$ does not depend on the base points $a, b$ chosen to parametrize $(0,1)$ by $\R$ using the flows $f_t$ and $f^t$ 
associated to $f$ (here, the total variation is understood on the circle or on $[0,1]$, depending whether $M^{a,b}_f$ is seen as a 
diffeomorphism of the circle or of $\R$). From now on we fix such a representative $M^{a,b}_f$ of $M_f$ with $a=b$, 
we still denote this representative by $M_f$, and we proceed to do explicit computations. 

In view of the discussion above, if the asymptotic distorsion of $f\in \mathrm{Diff}^{2,\Delta}_+ ([0,1])$ vanishes, the right hand side of \eqref{e:ineg1} 
vanishes. Thus, assuming this inequality holds, $DM_f$ must be constant, which is necessarily $1$ since $M_f$ is a circle diffeomorphism. 
Therefore, $f$ has a trivial Mather invariant.  \medskip

We now proceed to prove \eqref{e:ineg1}. We may assume that $f(x) > x$  for all $x \in (0,1)$. Indeed, changing $f$ by $f^{-1}$ 
transforms $M_f$ into its conjugate by the involution $\, y \mapsto -y \,$ of the circle, so that $\var (\log DM_f)$ 
remains the same, while $\, \dist (f) = \dist (f^{-1})$. 

Recall that, letting $\psi_X:t\mapsto  f_t(a)$ (which sends $[0,1]$ to $[a,f(a)]$), one has $DM_f=\frac X Y\circ \psi_X$, hence
\begin{equation}\label{e:varlogDM}
\var(\log DM_f) = \var \big( \log \tfrac X Y ; [a,f(a)] \big).
\end{equation}
As a consequence,
\begin{align*}\label{e:varlogDM}
\var(\log DM_f) &\le  \var(\log X ; [a,f(a)])+ \var(\log Y ; [a,f(a)])\\
&\le  |\log Df(0)| + \var (\log Df ; [0,a]) + |\log Df(1)| + \var(\log Df; [a,1]) \\
&=  |\log Df(0)|+|\log Df(1)|+\var(\log Df),
\end{align*}
where we have used \eqref{e:logX2} and its analogue for $Y$. Now, since $M_f$, as well as $\log Df(0)$ and $\log Df(1)$, are invariant under $C^{1+\mathrm{bv}}$ conjugacies, we get
$$\mathrm{var} (\log DM_f) \leq  |\log Df(0)| + |\log  Df(1)| +  \inf_{h \in \mathrm{Diff}_+^{1+\mathrm{bv}}([0,1])} \mathrm{var} \big( \log D (hfh^{-1}) \big),$$
which is precisely the announced inequality, according to \eqref{primera}.
\end{proof}

\begin{proof}[Alternative proof]
The essence of the proof of \eqref{e:ineg1} given above is ``hidden'' in the use of \eqref{e:logX2}, which relies on the explicit expressions of $X$ and $Y$. We next provide a more self-contained proof of \eqref{e:ineg1}, which will also be a starting point for the proof of Proposition~\ref{prop-unica}.

Given a point $a$ in the interior of the unit interval and $k \in \mathbb{Z}$, replacing $a$ by $f^k(a)$ in \eqref{e:varlogDM} we get
$$\var ( \log DM_f ) = \var \big( \log \tfrac X Y, [f^k(a),f^{k+1}(a)] \big).$$
Hence, for all $n \geq 1$,
\begin{equation}\label{e:mean}
\var ( \log D M_f ) = \frac{1}{2n} \var(\log \tfrac X Y, [f^{-n}(a),f^{n}(a)])
=\frac{1}{2n}  \int_{f^{-n}(a)}^{f^n(a)} \left| \frac{DX}{X} - \frac{DY}{Y} \right| (x) \, dx.
\end{equation}
Now, from the relations $\, X \circ f = X \cdot Df \,$ and $\, Y \circ f = Y \cdot Df, \,$ one obtains
$$X \circ f^{-2n} = X \cdot Df^{-2n}, \qquad Y \circ f^{2n} = Y \cdot Df^{2n}.$$
Taking derivatives of logarithms, these equalities yield
$$\frac{DX}{X} \circ f^{-2n} \cdot Df^{-2n} = \frac{DX}{X} + \frac{D^2 f^{-2n}}{D f^{-2n}}, \qquad 
\frac{DY}{Y} \circ f^{2n} \cdot Df^{2n} = \frac{DY}{Y} + \frac{D^2 f^{2n}}{D f^{2n}}.$$
Replacing the thus obtained values of $DX/X$ and $DY/Y$ in (\ref{e:mean}), we deduce that 
$\, \var ( \log DM_f)  \,$ is equal to
\begin{equation*}
\frac{1}{2n} \int_{f^{-n}(a)}^{f^n(a)} 
\left| \frac{DX}{X} \circ f^{-2n} \cdot Df^{-2n} - \frac{D^2 f^{-2n}}{D f^{-2n}} -
 \frac{DY}{Y} \circ f^{2n} \cdot Df^{2n} + \frac{D^2 f^{2n}}{D f^{2n}} \right| (x) \, dx,
\end{equation*}
which may be rewritten as  
\begin{equation}\label{estimate-4-terms}\frac{1}{2n} \int_{f^{-n}(a)}^{f^n(a)} 
\left| \frac{DX}{X} \circ f^{-2n} \cdot Df^{-2n} + \frac{D^2 f^{2n}}{D f^{2n}} \circ f^{-2n} \cdot Df^{-2n}-
 \frac{DY}{Y} \circ f^{2n} \cdot Df^{2n} + \frac{D^2 f^{2n}}{D f^{2n}} \right| \! (x) \, dx.
 \end{equation}
In particular, $\var ( \log DM_f )$ is bounded from above by the sum of the absolute values of the four terms involved in the integral (\ref{estimate-4-terms}) above.

In order to bound these terms, first notice that 
$$\frac{1}{2n} \int_{f^{-n}(a)}^{f^n(a)} \left| \frac{DX}{X} \circ f^{-2n} \cdot Df^{-2n} \right| \! (x) \, dx 
= \frac{1}{2n} \int_{f^{-3n}(a)}^{f^{-n}(a)} \left| \frac{DX}{X}  \right| \! (x) \, dx$$
converges to $|DX(0)| = |\log Df(0)|$ as $n$ goes to infinity. This is due to the fact that $DX$ is continuous, $f^{-n}(a)$ converges to $0$, and the integral of $1/X$ equals 1 on each interval $[f^{k}(a),f^{k+1}(a)]$. 

Similarly,
$$\frac{1}{2n} \int_{f^{-n}(a)}^{f^n(a)} \left| \frac{DY}{Y} \circ f^{2n} \cdot Df^{2n} \right| \! (x) \, dx 
= \frac{1}{2n} \int_{f^{n}(a)}^{f^{3n}(a)} \left| \frac{DY}{Y} \!  \right|  \! (x) \, dx$$
converges to $|DY (1)| = |\log Df (1)|$ as $n$ goes to infinity, since $f^n(a)$ tends to $1$.

Concerning the integrals of the absolute values of the two other terms in (\ref{estimate-4-terms}), their sum equals
\begin{small}\begin{eqnarray*}
\frac{1}{2n} \int_{f^{-n}(a)}^{f^n(a)} \left| \frac{D^2 f^{2n}}{D f^{2n}} \right| \! (x) \, dx &+ &
\frac{1}{2n} \int_{f^{-n}(a)}^{f^n(a)} \left| \frac{D^2 f^{2n}}{D f^{2n}} \circ f^{-2n} \cdot Df^{-2n} \right| \! (x) \, dx = \\
&=& 
\frac{1}{2n} \int_{f^{-n}(a)}^{f^n(a)} \left| \frac{D^2 f^{2n}}{D f^{2n}} \right| \! (x) \, dx + 
\frac{1}{2n} \int_{f^{-3n}(a)}^{f^{-n}(a)} \left| \frac{D^2 f^{2n}}{D f^{2n}} \right| \! (x) \, dx \\
&=& 
\frac{1}{2n} \int_{f^{-3n}(a)}^{f^n(a)} \left| \frac{D^2 f^{2n}}{D f^{2n}} \right| \! (x) \, dx 
\,\,\,\, \leq\,\,\,\,
\frac{1}{2n} \int_{0}^{1} \left| \frac{D^2 f^{2n}}{D f^{2n}} \right| \! (x) \, dx.
\end{eqnarray*}\end{small}Since the last expression converges to $\, \dist (f) \,$ as $n$ goes to infinity, the desired inequality \eqref{e:ineg1}
is obtained just by passing to the limit.
\end{proof}

\vspace{0.1cm}

We next give a proof of the reverse statement, thus closing the proof of Theorem \ref{t:vanish}.

\vspace{0.2cm}

\begin{prop} \label{segunda-prop} 
If $f \in \mathrm{Diff}_+^{2,\Delta} ([0,1])$ has a trivial Mather invariant and the endpoints are parabolic fixed points, then its asymptotic distortion vanishes.
\end{prop}

\begin{proof} Let $X$ be the $C^1$ vector field on $[0,1]$ associated to $f$, according to the hypothesis of triviality of the Mather invariant. We assume that $f(x) > x$ in the interior, the other case being analogous. For each $x$ we have 
$$X(f(x)) = Df(x) \cdot X(x).$$
Thus, for $x \in (0,1)$,
$$\log Df(x) = \log (X (f(x))) - \log (X(x)),$$
hence
\begin{equation}\label{cocycle}
\frac{D^2 f}{D f} (x) = D \log (X) (f(x)) \cdot Df (x) - D \log (X) (x).
\end{equation}

We let $y_n,z_n$ be sequences of points in $(0,1)$ converging to $0$ and $1$, respectively. We then let $u : [0,1] \to \mathbb{R}$ be the function defined as the truncation of $D \log (X)$ to $[y_n,z_n]$, that is, 
$$u_n := D \log (X) \cdot 1_{[y_n,z_n]}.$$
Since $X$ is $C^1$ on $(0,1)$, this function belongs to $L^1([0,1])$. We will show that the $L^1$-norm of 
\begin{equation}\label{a-estimar}
\frac{D^2f}{Df} - \big( (u_n \circ f) \cdot Df - u_n \big)
\end{equation}
converges to $0$ as $n$ goes to infinity. According to Proposition \ref{doble-criterio}, 
this implies that the asymptotic distortion of $f$ vanishes. (These $u_n$ and their convergence will be 
studied as the derivatives of a sequence of conjugators later on.)

To show the desired convergence, notice that, by (\ref{cocycle}), 
$$\frac{D^2f}{Df} (x) - \big( (u_n \circ f) \cdot Df - u_n \big) (x) 
= \left \{
      \begin{array}{rcl}
          0 & \mbox{if} & x \in [y_n , f^{-1}(z_n)], \\         
         -D \log (X) (x)  & \mbox{if} & x \in [f^{-1}(y_n), y_n], \\ 
         D \log (X) (f(x)) \cdot Df (x) & \mbox{if} & x \in [f^{-1} (z_n), z_n], \\ 
         \frac{D^2f}{Df}(x) & \mbox{if} & x < f^{-1}(y_n) \mbox{ or }  x > z_n.
      \end{array}
   \right .$$
Thus, the $L^1$-norm of the coboundary defect (\ref{a-estimar}) equals
\begin{footnotesize}
$$\int_{0}^{f^{-1}(y_n)} \left| \frac{D^2 f}{D f} \right|  (x) \, dx 
+ \int^{y_n}_{f^{-1}(y_n)} \big| D \log (X) \big| (x)\, dx 
+ \int^{z_n}_{f^{-1}(z_n)} \big| D \log (X) (f(x)) \cdot Df(x) \big| \, dx 
+ \int_{z_n}^{1} \left| \frac{D^2 f}{D f} \right| (x)\, dx,$$
\end{footnotesize}where the third term, by a change of variable, is equal to $\int_{z_n}^{f(z_n)}|D\log(X)|$. 
According to \eqref{e:logX2} (and its analogue on $(0,1]$), and since $Df (0) = Df (1) = 1$, the second and third terms are bounded from above by $\var (\log Df ; [0,f^{-1}(y_n)])$ and $\var (\log Df ; [z_n,1])$, respectively. These are equal to the first and last terms, and  converge to zero as $n$ goes to infinity since $D^2 f / D f$ is integrable on $[0,1]$. (One can alternatively deal with the second and third terms by observing that $DX$ is continuous and vanishes at $0$  and $1$ by the parabolicity hypothesis, and that, according to (\ref{eq:int=1}), for every $a\in(0,1)$ one has $\int_a^{f(a)}\frac1X=1$.)
\end{proof}

\vspace{0.1cm}

\begin{rem} \label{no-sub} 
One can use Theorem \ref{t:vanish} to easily show that $\dist$ is not subadditive, as announced at the beginning of \S \ref{s:dist}. Indeed, let $X$ be a $C^{\infty}$ vector field on $[0,1]$ that does not vanish in $(0,1)$ and is $C^{\infty}$ flat at the endpoints. Let $f$ be the time-1 map of the flow of $X$. Fix $p \in (0,1)$, and denote by $I$ the closed interval with endpoints $p,f(p)$. Let $Y$ be a $C^{\infty}$ vector field on $I$ that does not vanish in the interior of $I$ and is $C^{\infty}$ flat at the endpoints. Extend $Y$ as being identically zero on $[0,1] \setminus I$, and let $h$ be the time-1 map of $Y$. 
It follows from Theorem \ref{t:vanish} that both $f$ and $h$ have zero asymptotic distortion. However, by Lemma \ref{break}, the diffeomorphism $fh$ has a nontrivial Mather invariant. Therefore, by Theorem \ref{t:vanish} again, $\dist (fh) > 0 = \dist (f) + \dist(h)$.

It is worth noticing that the diffeomorphism $g : = f h$ above follows the same line of construction as the example given in \cite{mio2} of an interval diffeomorphism with parabolic fixed points and positive asymptotic distortion. The proof of the positivity therein is done by hand (that is, without any reference to the Mather invariant), and applies more generally to $C^{1+\mathrm{bv}}$ diffeomorphisms. However, it only works under an extra dynamical hypothesis, namely, there is a point $q \in (p,f(p))$ fixed by $h$ (hence having the same orbits under $f$ and $g$), yet $Dh (q) \neq 1$ (hence $Df (q) \neq Dg (q)$). See Proposition \ref{extension-bump} for more on this.
\end{rem}

\vspace{0.1cm}

\begin{rem} 
According to Theorem \ref{t:vanish}, the stabilization equality $\, \dist(f^n) = n \cdot \dist (f) \,$ for $f \in \mathrm{Diff}^{2,\Delta}_+ ([0,1])$ and $n \geq 1$ translates into the equality
$$\var (\log M_{f^n}) = n \cdot \var (\log M_f),$$
which follows more directly from the fact that $M_{f^n}$ is nothing but the lift of $M_f$ to the $n$-fold-covering of the circle, as can be easily checked.
\end{rem}

\vspace{0.2cm}

\noindent{\bf {\em Another sequence of conjugators.}} 
Recall that, by the proof of Proposition \ref{doble-criterio}, a sequence $u_n\in L^1([0,1])$ such that the coboundary defect $D^2f/Df-((u_n\circ f)Df-u_n)$ goes to $0$ when $n$ goes to infinity gives rise to\footnote{Or arises from, depending on the point of view...} a sequence of conjugators $\bar h_n\in \Diff^{1+\mathrm{ac}}([0,1])$ such that $\bar h_n\circ f\circ \bar h_n^{-1}$ goes to $\id$ in $\Diff^{1+\mathrm{ac}}_+([0,1])$. Namely, according to Remark \ref{conjugators-h_n}, these conjugators arise as solutions of $D^2 \bar h_n / D \bar h_n = - u_n$. Given $f \in \mathrm{Diff}^{2,\Delta}([0,1])$ with trivial Mather invariant and parabolic fixed points, the $\bar h_n$ associated to the specific $u_n$ from the proof of Proposition \ref{segunda-prop} differs from the conjugators $h_n$ defined by (\ref{def-h_n}). However, they share similar properties, as we will next see. More precisely, Propositions \ref{p:Dhn-centre} and \ref{final-prop-extremos} below are analogues of of Propositions \ref{derivada-interior} and \ref{derivada-extremos}, respectively. \medskip

Let us start by ``interpreting'' these $u_n$ and $\bar h_n$. On $(0,1)$, there are many \emph{continuous} solutions $u$ to 
\begin{equation}\label{eq:exacta}
\frac{D^2f}{Df}=(u\circ f)\cdot Df-u.
\end{equation}
However, none of these solutions can be integrable. Otherwise, if we let $h$ be the $C^{1+\mathrm{ac}}$ diffeomorphism of $[0,1]$ solving $D^2 h / Dh = -u$, one would have $D^2 (hfh^{-1}) / D (hfh^{-1}) = 0.$ This would mean that $hfh^{-1}$ is an affine map, which cannot be the case since the only affine homeomorphism of the interval is the identity and $f$ is nontrivial.

Due to (\ref{cocycle}), a particular (and relevant) solution to (\ref{eq:exacta}) on $(0,1)$ is $u = D\log X$, where $X$ is the unique generating $C^1$ vector field for $f$ on $[0,1]$.  For this choice of $u$, the equation $\frac{D^2h}{Dh}=-u$, that is, $D\log Dh=-D\log X$, defines $Dh$ up to a positive multiplicative constant, namely $Dh=\frac{c}{X}$. For any $c>0$, this defines (up to a choice for $h(\frac12)$, say) a diffeomorphism $h \!: (0,1)\to\R$, which, by property (\ref{eq:int=1}), conjugates $f|_{(0,1)}$ to the translation by $c$ (which is an affine map!). Indeed, for all $x \in (0,1)$, 
\begin{equation}
\label{e:conj-trans}
h(f(x))-h(x) = \int_x^{f(x)}Dh(y) \, dy = c  \int_x^{f(x)}\frac{dy}{X(y)} =c.
\end{equation}
Up to now, this was a rather convoluted way of conjugating a fixed point free diffeomorphism of the open interval to a translation.  Now, when we truncate $u$ to obtain the integrable function $u_n$ in the proof of Proposition \ref{segunda-prop}, we give rise to diffeomorphisms $\bar h_n$ \emph{from $[0,1]$ to itself} which conjugate $f$ to smaller and smaller translations on bigger and bigger compact intervals inside 
$(0,1)$ (see Proposition \ref{p:Dhn-centre} below). The issue here is that near the endpoints, the convergence to the identity of the conjugates is guaranteed by the regularity of $X$ \emph{at both endpoints}. This is where the triviality of the Mather invariant crucially comes into play. 

Let us now be more specific. For concreteness, we will still assume that $f (x) > x$ for all $x \in (0,1)$. We fix a point $p \in (0,1)$, and for simplicity we consider choices of the form $y_n := f^{-k_n} (p)$ and $z_n := f^{\ell_n} (p)$ for increasing sequences of positive integers $k_n,\ell_n$. (This is not  a major restriction, as one can easily adapt the estimates below by looking at the points in the orbit of $p$ that are the 
closest to $y_n$ and $z_n$, respectively.) 

\vspace{0.3cm}

\begin{prop}\label{p:Dhn-centre}
For every $x \in (0,1)$, the value of $D \bar{h}_n$ converges to $0$ as $n$ goes to infinity. Actually, the convergence is uniform on compact subsets of $(0,1)$.
\end{prop}

\begin{proof} From 
$$D \log (D \bar{h}_n) = \frac{D^2 \bar{h}_n}{D \bar{h}_n} = - u_n = -D \log (X) \cdot 1_{[y_n,z_n]}$$
one obtains by integration
$$\log (D \bar{h}_n) (x) = \log (D \bar{h}_n (0)) - \int_0^x D \log (X) (s) \cdot 1_{[y_n,z_n]}(s) \, ds.$$
This easily yields
\begin{equation}\label{derivadas-en-3}
D \bar{h}_n (x) 
= \left \{
      \begin{array}{rcl}
         D \bar{h}_n (0) & \,\, \mbox{if} \,\, & x \in [0, y_n], \\ 
         D \bar{h}_n (0) \cdot \frac{X(y_n)}{X(x)} & \mbox{if} & x \in [y_n,z_n], \\ 
         D \bar{h}_n (0) \cdot \frac{X(y_n)}{X(z_n)} & \mbox{if} & x \in [z_n, 1]. \\ 
      \end{array}
   \right .
\end{equation}
In other words, $D \bar{h}_n$ is constant on $[0,y_n]$ and $[z_n,1]$, and equal to \, 
$c_n / X$ \, on $[y_n,z_n]$, where the constant $c_n$ is uniquely determined by the requirement that $ \int_0^1 D \bar{h}_n (x) \, dx =1$. Since $1/X$ is bounded from above on any compact $J\subset (0,1)$, in order to prove the proposition we only need to show that \, $\lim_n c_n=0$. \, Now, as in \eqref{e:conj-trans}, for every $x\in [y_n,f^{-1}(z_n)]$ we have 
$$h(f(x))-h(x) = c_n.$$ 
As a consequence, 
$$1\ge h(z_n)-h(y_n)=h(f^{\ell_n}(p))-h(f^{-k_n}(p))=(k_n+\ell_n) \, c_n,$$ 
which implies the desired convergence.
\end{proof}

For Proposition \ref{final-prop-extremos} below (which is an analog of Proposition \ref{derivada-extremos}), we will crucially need the following condition for the sequences $k_n, \ell_n$:
\begin{equation}\label{convergences-to-zero}
\lim_{n \to \infty} k_n \, X (f^{-k_n} (p)) = \lim_{n \to \infty} \ell_n \, X (f^{\ell_n}(p)) = 0.
\end{equation}
The existence of sequences satisfying this property is ensured by the next lemma.

\vspace{0.3cm}

\begin{lem} \label{density-1}
With the notations above, there exist increasing and decreasing infinite sequences of integers $n_k$ such that 
\, $n_k \, X (f^{n_k} (p))$ \, converges to zero as $k \to \infty$.
\end{lem}

\vspace{0.1cm}

\begin{proof} Since $X (f^n(p)) = X(p) \cdot Df^n (p)$, what we need to show is that 
$n \, Df^n (p)$ converges to zero along infinite sequences of positive and negative integers $n$. Let $I$ be the interval of endpoints $p,f(p)$. For each $n \geq 0$, let $p_n \in I$ be such that
$$\frac{| f^n (I) |}{|I|} = Df^n (p_n).$$
Since 
\begin{eqnarray*}
\left| \log \left( \frac{Df^n (p)}{Df^n (p_n)} \right) \right| 
&=& 
\left| \log \left( \frac{Df(p) \cdot Df (f(p)) \cdots Df(f^{n-1}(p)) }{Df (p_n) \cdot Df (f(p_n)) \cdots Df (f^{n-1}(p_n))} \right) \right| \\
&\leq&
\sum_{i=0}^{n - 1}  \big| \log (Df (f^i (p))) - \log (Df (f^i (p_n))) \big| \,\,\, \leq \,\,\, \mathrm{var} (\log Df) \,\,\, =: \,\,\, V,
\end{eqnarray*}
we have 
$$Df^n (p) \leq e^V \, Df^n (p_n) = e^V \, \frac{|f^n (I)|}{|I|}.$$
Thus,
$$e^{-V} |I| \sum_{n \geq 0} Df^n (p) \leq \sum_{n \geq 0} | f^n (I) | < 1.$$
Therefore, given any $\varepsilon > 0$, we must have $Df^n (p) \leq \varepsilon / n$ for 
a density-1 set of positive indices $n$, otherwise the sum of $Df^n (p)$ wouldn't converge. This shows the existence of a sequence of positive integers 
$n$ along which $n \, D(f^n(p))$ converges to $0$. The existence of a sequence of negative integers sharing this property is established in a similar way.
\end{proof}

\vspace{0.3cm}

We will need to impose another extra condition on the sequences $k_n, \ell_n$. Namely, 
we will ask for the existence of a positive constant $C$ such that, for all $n$, 
\begin{equation}\label{extra-condition}
\frac{1}{C} \leq \frac{X (f^{-k_n} (p))}{X (f^{\ell_n} (p))} \leq C.
\end{equation}
The existence of sequences with this property (and also satisfying (\ref{convergences-to-zero})) can be easily established from the 
relation $X (f^m (p)) = X(p) \cdot Df^m (p)$ by using the fact that, as it was pointed out in the proof of Lemma \ref{density-1}, the 
sequences $k_n$ and $\ell_n$ arise along subsets of density 1 of the set of positive integers. (We leave the details to the reader.)

\vspace{0.3cm}

\begin{prop} \label{final-prop-extremos}
If we impose the extra conditions (\ref{convergences-to-zero}) and (\ref{extra-condition}) for the sequences $k_n, \ell_n$, then the derivative of $\bar{h}_n$ at each endpoint diverges as $n$ goes to infinity.
\end{prop}

\begin{proof} Given \eqref{derivadas-en-3}, we have 
\begin{small}
$$1 
= \int_0^1 D \bar{h}_n (x) \, dx 
= D \bar{h}_n (0) 
    \left[ y_n + X (y_n) \underbrace{\int_{y_n}^{z_n} \frac{dx}{X(x)} }_{k_n+\ell_n}
    + (1-z_n) \cdot \frac{X (y_n)}{X (z_n)} \right],$$
\end{small}which gives
\begin{equation}\label{derivada--en--zero}
D \bar{h}_n (0) = \frac{1}{y_n + X (f^{-k_n}(p)) \cdot (k_n + \ell_n)
    + (1-z_n) \cdot \frac{X (f^{-k_n}(p))}{X (f^{\ell_n} (p) )}}.
\end{equation}     
We claim that each term in the denominator of the right-side expression in (\ref{derivada--en--zero}) converges to zero. Indeed, this is obvious for $y_n$ and it follows from (\ref{extra-condition}) for $(1-z_n) \cdot \frac{X (f^{-k_n}(p))}{X (f^{\ell_n} (p))}$. Finally, by (\ref{convergences-to-zero}), the value of 
$$(k_n + \ell_n) \cdot X (f^{-k_n}(p)) \leq k_n \, X (f^{-k_n}(p)) + C \, \ell_n \, X (f^{\ell_n}(p))$$
converges to zero as well. This shows that $D \bar{h}_n (0)$ diverges to infinity, and the 
same holds for $D \bar{h}_n (1) = D \bar{h}_n (0) \cdot \frac{X (f^{-k_n}(p))}{X (f^{\ell_n} (p))}$ thanks to \eqref{extra-condition}.
\end{proof}

\vspace{0.05cm}

\begin{rem}\label{raro}
In order to produce examples for which the answer to Question \ref{q:conj-tg} is affirmative, the computations above suggest to consider oscillating diffeomorphisms so that the derivatives of iterates have a low growth along increasing subsequences (as it happens,  
for instance, for the diffeomorphisms constructed in Theorem 1.7 of \cite{Polt-Sod}) and/or looking at times $k_n, \ell_n$ along which the behaviors of $X (f^{-k_n} (p))$ and $X (f^{\ell_n} (p))$ are decoupled.
\end{rem}

%%%%%%%%%%%%%%%%%%%%%%%%%%%%%%%%%%%%%%%%%%%%%%%%%%%%%%%%%%%%%%%%%%

\section{Nontrivial Mather invariant: a proof of Theorem \ref{t:fund-ineq}}
\label{s:fund-ineq}

Given \eqref{e:ineg1}, in order to prove Theorem \ref{t:fund-ineq}, we are left with proving that 
$$\mathrm{dist}_{\infty} (f) \leq  \mathrm{var} ( \log D M_f ) + \big| \log Df(0) \big| + \big| \log Df(1) \big|.$$
After this, we will discuss some examples and results in relation to the equality case in (\ref{general-eq}). 

\vspace{0.5cm}

\noindent{\bf {\em Controlling the asymptotic distortion.}} We follow the method of proof of Proposition \ref{segunda-prop}, but instead of working with a single vector field $X$ as in the case of a trivial Mather invariant, we consider the left and right vector fields 
($X$ and $Y$, respectively), and use them to approximate $D^2 f / Df$ by coboundaries in $L^1$. As we will see, there are two reasons this can 
fail: the hyperbolicity at the endpoints and the failure of compatibility between these vector fields. The last issue can be detected in a single fundamental domain by the Mather invariant $M_f$, more precisely, by the total variation of the logarithm of its derivative.

As before, we may assume that $f(x) > x$  for all $x \in (0,1)$. We fix a point $a \in (0,1)$, and we do explicit computations for $M_f = M^{a,a}_f$. Consider the function $u_n$ defined as 
$$u_n (x) := \left\{
      \begin{array}{rcl}
          0 & \mbox{if} & x \notin [y_n,z_n], \\
         D \log (X) (x) & \mbox{if} & x \in [y_n, f(a)], \\ 
         D \log (Y) (x)  & \mbox{if} & x \in [f(a),z_n], 
        \end{array} \right .$$
where $y_n$ (resp. $z_n$) is a sequence of points in $(0,1)$ converging to $0$ (resp. $1$). We next compute the $L^1$-norm of the coboundary defect
\begin{equation}\label{cob-defect}
\frac{D^2 f}{D f} - \big( u_n (f(x)) \cdot Df (x) - u_n (x) \big) 
\end{equation}
which, as can be easily checked, coincides with
$$\left \{
      \begin{array}{rcl}
          0 & \mbox{if} & x \in [ y_n, a] \cup [f(a), f^{-1}(z_n)], \\
         -D \log (X) (x)  & \mbox{if} & x \in [f^{-1}(y_n), y_n], \\ 
         D \log (Y) (f(x)) \cdot Df (x) & \mbox{if} & x \in [f^{-1} (z_n), z_n], \\ 
         \frac{D^2 f}{D f} - \big( D \log (Y) (f(x)) \cdot Df (x) - D \log X (x) \big) & \mbox{if} & x \in [a, f(a)], \\
         \frac{D^2f}{Df}(x) & \mbox{if} & x < f^{-1}(y_n) \mbox{ or }  x >  z_n.
      \end{array}
   \right .$$
In comparison to the computation of the proof of Proposition \ref{segunda-prop}, the main new ingredient is the estimate on the interval $[a,f(a)]$, namely 
\begin{equation}\label{cob}
\int^{f(a)}_a \left| \frac{D^2 f}{D f} - \big( D \log (Y) (f(x)) \cdot Df (x) - D \log X (x) \big) \right| \, dx.
\end{equation}
In order to proceed notice that, just as for $X$ in \eqref{cocycle}, 
$$D \log(Y) \circ f \cdot Df - D \log (Y) = \frac{D^2 f}{Df}.$$
Therefore, the expression \eqref{cob} equals  
\begin{small}
\begin{equation}\label{var-mather}
\int^{f(a)}_a \big| D \log (X)(x) - D \log Y (x) \big| \, dx = \var \big( \log(\tfrac X Y ) ; [a,f(a)] \big) = \var (\log DM_f),
\end{equation}
\end{small}according to (\ref{derivada-de-M}).

Concerning the other intervals,  the integral of the absolute value of (\ref{cob-defect}) is obviously $0$ on $[y_n,a]$ and $[f(a), f^{-1}(z_n)]$. Moreover, it converges to zero on $[0,f^{-1}(y_n)]$ and $[z_n, 1]$, since $D^2 f / Df$ is an integrable function. Finally, on $[f^{-1}(y_n), y_n]$ and $[f^{-1}(z_n),z_n]$, the integral converges to $| \log  Df (0)|$ and $| \log Df (1)|$, respectively, 
due to \eqref{e:logX2}. 

We thus conclude that 
$$\liminf_{n \to \infty} \left\| \frac{D^2 f}{D f} - \big( u_n \circ f \cdot Df - u_n \big) \right\|_{L^1} \leq \big| \log Df(0) \big| + \big| \log Df (1) \big| + \var ( \log DM_f ).$$
In view of (\ref{segunda}), this yields the desired upper bound: 
$$\dist (f) \leq \var (\log  D M_f ) + \big| \log Df(0) \big| + \big| \log Df (1) \big|.$$

\vspace{0.35cm}

\noindent{\bf {\em Some examples.}} 
Remind from (\ref{siempre}) that  \,
$\dist (f) \geq \big| \log Df(0) \big| + \big| \log Df(1) \big|.$

\vspace{0.2cm}

\begin{ex} If $f \in \mathrm{Diff}^{2,\Delta}_+ ([0,1])$ has a trivial Mather invariant, 
then \eqref{general-eq} and \eqref{siempre} give 
\begin{equation}\label{iguala}
\dist (f) = \big| \log Df(0) \big| + \big| \log Df(1) \big|.
\end{equation}
It is worth noticing, however, that equality (\ref{iguala}) may also hold for diffeomorphisms with a nontrivial Mather invariant, as it is shown in the next example.
\end{ex}

\vspace{0.2cm}

\begin{ex} \label{monotone}
If $f \in \mathrm{Diff}^{2,\Delta}_+ ([0,1])$ has hyperbolic fixed points and its derivative is monotone, an easy argument (that we leave to the reader) shows that 
$$\dist (f) = \big| \log Df(0) \big| + \big| \log Df(1) \big|.$$ 
If, besides, $f$ has a nontrivial Mather invariant, this obviously yields a strict inequality
$$\dist (f) < \var ( \log DM_f ) + \big| \log Df(0) \big| + \big| \log Df(1) \big|.$$
If $f$ has a trivial Mather invariant, $f$ can be perturbed (on an interval with endpoints of the form $p,f(p)$) into a diffeomorphism with nontrivial Mather invariant (see Lemma \ref{break}) so that the monotonicity of $Df$ is preserved, thus providing a diffeomorphism satisfying the strict inequality above. As we will see in Proposition \ref{prop-unica} 
below, this implies that, for such an $f$, one has 
$$\big| \var (  \log DM_f ) - \dist (f) \big| \neq \big| \log Df(0) \big| + \big| \log Df(1) \big|.$$
\end{ex}

\vspace{0.2cm}

\begin{ex} \label{ex-limit} We next show that the inequality (\ref{e:ineg1}), namely
\begin{equation*}
\var(\log DM_f)\le |\log Df(0)|+|\log Df(1)|+\dist(f),
\end{equation*}
cannot be improved even in the case of non-parabolic fixed points. More precisely, 
we will build a sequence of diffeomorphisms $f_n \!\in\! \mathrm{Diff}^{\infty, \Delta}_+ ([0,1])$ for which both sequences $Df_n (0)$, $Df_n (1)$ are constant and such that the difference between the left and right hand-side expressions above converges to $0$ as $n$ goes to infinity.

Given $\lambda > 0 > \mu$, let $\, b \,$ be chosen so that $\, e^{\lambda} b + e^{\mu} (1-b) = 1, \,$ that is, 
$$\, b := (1-e^{\mu}) / (e^{\lambda}-e^{\mu}). \,$$
Consider the piecewise-affine homeomorphism $f$ of $[0,1]$ having derivative $e^{\lambda}$ on $[0,b]$ and $e^{\mu}$ on $[b,1]$. Let $a < b$ be a point close enough to $b$ so that $e^{\lambda} a > b$. Let $f_n$ be a sequence of $C^{\infty}$ diffeomorphisms with non-increasing derivative that coincide with $f$ on a neighborhood 
of the complement of $[b-1/n,b+1/n]$. By the monotonicity of the derivative, for each $n$ we have 
$$\dist (f_n) =  \lambda - \mu = \big| \log Df(0) \big| + \big| \log Df(1) \big|.$$
We claim that $\var ( \log D M_{f_n} )$ converges to 
$$ 2 \, \big[ | \log Df(0) | + | \log Df(1) | \big].$$ 
Indeed, the vector fields $X_n, Y_n$ associated to $f_n$ satisfy $X_n (x) = \lambda \, x$ for all $x \in [0,f(a)]$ and $Y_n (x) = \mu \, (1-x)$ for all $x \in [f(a),1]$. Using the equality 
$$\frac{DY_n}{Y_n} \circ f_n \cdot Df_n = \frac{DY_n}{Y_n} +  \frac{D^2 f_n}{Df_n},$$
we conclude
\begin{small}
$$\var ( \log D M_{f_n} ) 
= \int_{a}^{f_n(a)} \left| \frac{DX_n}{X_n} - \frac{DY_n}{Y_n} \right| (x) \, dx 
= \int_{a}^{f_n(a)} \left| \frac{DX_n}{X_n} - \frac{DY_n}{Y_n} \circ f_n \cdot Df_n + \frac{D^2f_n}{Df_n} \right| (x) \, dx.$$
\end{small}
Since $D^2 f_n/ Df_n = 0$ outside $[b-1/n,b+1/n]$, we conclude that $\var ( \log D M_{f_n} )$ equals 
\begin{footnotesize}
$$\int_a^{b-1/n} \left| \frac{1}{x} + \frac{Df_n(x)}{1-f_n(x)} \right| \, dx + \int_{b+1/n}^{f_n(a)} \left| \frac{1}{x} + \frac{Df_n (x)}{1-f_n(x)} \right| \, dx 
+ \int_{b-1/n}^{b+1/n} \left| \frac{DX_n}{X_n} - \frac{DY_n}{Y_n} \circ f_n \cdot Df_n + \frac{D^2f_n}{Df_n} \right| (x) \, dx.$$
\end{footnotesize}The sum of the first two integrals equals
$$\log \left( \frac{b-1/n}{a} \cdot \frac{f (a)}{b+1/n} \right) - \log \left( \frac{1-f (b-1/n)}{1-f(a)} \cdot \frac{1-f^2 (a)}{1-f (b+1/n)} \right),$$
which converges to $\log (e^{\lambda}) + \log (e^{-\mu}) = |\log (Df(0))| + |\log Df(1)|$ as $n$ goes to infinity. Concerning the third integral, the contribution of the terms $DX_n / X_n$ and $(DY_n / Y_n ) \circ f_n \cdot Df_n$ becomes negligible in the limit, and 
$$\int_{b-1/n}^{b+1/n} \left| \frac{D^2 f_n}{Df_n} \right| (x) \, dx = \int_{0}^{1} \left| \frac{D^2 f_n}{Df_n} \right| (x) \, dx 
= \var(\log Df_n) = \big| \log Df(0) \big| + \big| \log Df(1) \big|,$$
which yields the announced convergence.
\end{ex}

\vspace{0.3cm}

In the example above, it can be directly checked that the value of 
$$2 \, \big[ | \log Df_n (0) | + | \log Df_n (1) | \big]  = \dist (f_n) + | \log Df_n (0) | + | \log Df_n (1) |$$
is strictly larger than $\var ( \log  D M_{f_n} )$. Actually, this is a quite general phenomenon, as it is next shown.

\vspace{0.3cm}

\begin{prop}\label{prop-unica} 
If $f \in \mathrm{Diff}^{2,\Delta}_+ ([0,1])$ is hyperbolic at both endpoints, then the strict inequality below is satisfied:
$$\var ( \log DM_f )  < \dist (f) + \big| \log Df(0) \big| + \big| \log Df(1) \big|.$$
\end{prop}

\begin{proof} We give a proof based on the second proof of inequality \eqref{e:ineg1} in the previous section, where we showed that, for every $n\ge 1$, the value of $\, \var ( \log DM_f)  \,$ equals
\begin{equation*}\frac{1}{2n} \int_{f^{-n}(a)}^{f^n(a)} 
\left| \frac{DX}{X} \circ f^{-2n} \cdot Df^{-2n} + \frac{D^2 f^{2n}}{D f^{2n}} \circ f^{-2n} \cdot Df^{-2n}-
 \frac{DY}{Y} \circ f^{2n} \cdot Df^{2n} + \frac{D^2 f^{2n}}{D f^{2n}} \right| \! (x) \, dx.
 \end{equation*} 
An alternative proof based on the ``localization result'' for the asymptotic 
distortion obtained in \S \ref{s:local} will be given in Remark \ref{preuve-alt}.

To begin this proof notice that, since every term in the inequality we wish to prove is invariant under $C^2$ conjugacy, by the Sternberg-Yoccoz' linearization theorem \cite{sternberg,yoccoz}, we may assume that $f$ is linear on neighborhoods of both endpoints, say $f(x) = e^{\lambda} x$ close to $0$ and $f(x) = 1 - e^{\mu} (1-x)$ close to $1$. For the sake of concreteness, we assume $\lambda > 0 > \mu$. 
We choose a large enough $n$ so that both $[0, f^{-n} (a)]$ and $[f^n(a), 1]$ lie in the domains of linearity of $f$. Over there, the vector fields $X$ and $Y$ coincide with $\lambda \, x$ and $\mu \, (1-x)$, respectively, where $\lambda := \log Df(0)$ and $\mu := \log Df (1)$. Therefore, the expression for $\var ( \log DM_f )$ above transforms into 
\begin{equation}\label{int-trans}
\frac{1}{2n} \int_{f^{-n}(a)}^{f^n(a)} 
\left| \frac{Df^{-2n}}{f^{-2n}}  + \frac{D^2 f^{2n}}{D f^{2n}} \circ f^{-2n} \cdot Df^{-2n} + 
 \frac{D f^{2n}}{1- f^{2n}} + \frac{D^2 f^{2n}}{D f^{2n}} \right| (x) \, dx.
\end{equation} 
Notice that, on the one hand, the first and third terms inside the absolute value in the integral are strictly positive. On the other hand, since $D^2 f / Df$ equals zero on $[0,f^{-n}(a)] \cup [f^n(a),1]$, we have that $D^2 f^{2n} / Df^{2n}$ equals zero on $[0,f^{-3n}(a)] \cup [f^n(a),1]$, hence
\begin{eqnarray*}
\frac{1}{2n} \int_{f^{-n}(a)}^{f^n (a)} \left[ \frac{D^2 f^{2n}}{D f^{2n}} \circ f^{-2n} \cdot Df^{-2n} + \frac{D^2 f^{2n}}{D f^{2n}} \right] \! (x) \, dx
&=& \frac{1}{2n} \int_{f^{-3n}(a)}^{f^n(a)} \frac{D^2 f^{2n}}{D f^{2n}}(x) \, dx \\
&=& \frac{1}{2n} \int_0^1 \frac{D^2 f^{2n}}{D f^{2n}}(x) \, dx \\
&=& \log(Df (1)) - \log (Df (0)) \\
&=& \mu - \lambda 
\,\,\,  < \,\,\, 0.
\end{eqnarray*}
Therefore, the expression 
$$\frac{D^2 f^{2n}}{D f^{2n}} \circ f^{-2n} \cdot Df^{-2n} + \frac{D^2 f^{2n}}{D f^{2n}}$$
must be negative on some subset of positive measure of $[f^{-n}(a), f^n (a)]$. This implies that some cancellation must occur in the integral (\ref{int-trans}) above, hence $\var ( \log DM_f )$ is strictly smaller than the sum 
\begin{small}
$$\int_{f^{-n}(a)}^{f^n(a)} 
\left| \frac{Df^{-2n}}{f^{-2n}} \right| (x) \, dx 
 + 
 \int_{f^{-n}(a)}^{f^n(a)} 
\left| \frac{D f^{2n}}{1- f^{2n}} \right| (x) \, dx 
 +
 \int_{f^{-n}(a)}^{f^n(a)} 
\left| \frac{D^2 f^{2n}}{D f^{2n}} \circ f^{-2n} \cdot Df^{-2n} + \frac{D^2 f^{2n}}{D f^{2n}} \right| (x) \, dx.
 $$
\end{small}The proof will be finished by showing that this sum equals 
the second member of the inequality of Proposition \ref{prop-unica}. 

To do this, first notice that 
\begin{eqnarray*}
\int_{f^{-n}(a)}^{f^n(a)} \left| \frac{Df^{-2n}}{f^{-2n}} \right| (x) \, dx 
&=& \frac{1}{2n} \int_{f^{-n}(a)}^{f^n(a)} \left| \frac{DX}{X} \circ f^{-2n} \cdot Df^{-2n} \right| \! (x) \, dx \\
&=& \frac{1}{2n} \int_{f^{-3n}(a)}^{f^{-n}(a)} \left| \frac{DX}{X} \right| (x) \, dx \,\, = \,\, |DX(0)| \,\, = \,\, |\log Df(0)|,
\end{eqnarray*}
since $DX$ is constant on $[0,f^{-n}(a)]$ and the integral of $1/X$ equals 1 on each interval $[f^{k}(a),f^{k+1}(a)]$. Similarly,
\begin{eqnarray*}
\int_{f^{-n}(a)}^{f^n(a)} \left| \frac{D f^{2n}}{1- f^{2n}} \right| (x) \, dx 
&=& \frac{1}{2n} \int_{f^{-n}(a)}^{f^n(a)} \left| \frac{DY}{Y} \circ f^{2n} \cdot Df^{2n} \right| (x) \, dx \\
&=& \frac{1}{2n} \int_{f^{n}(a)}^{f^{3n}(a)} \left| \frac{DY}{Y}  \right| (x) \, dx \,\, = \,\, |DY (1)| \,\, = \,\, |\log Df (1)|.
\end{eqnarray*}
To conclude, we claim that 
\begin{equation}\label{eq:rara} 
\frac{1}{2n} 
\int_{f^{-n}(a)}^{f^n(a)} \left| \frac{D^2 f^{2n}}{D f^{2n}} \circ f^{-2n} \cdot Df^{-2n} + \frac{D^2 f^{2n}}{D f^{2n}} \right| (x) \, dx 
= \dist (f).
\end{equation}
Indeed, since $\frac{D^2f}{Df} = 0$ on $[0,f^{-n}(a)] \cup [f^n(a),1]$, for each $k \geq 1$ we have that 
$$\int_0^1 \left| \frac{D^2 f^{2nk}}{D f^{2nk}} \right| (x) \, dx
= \int_0^1 \left| \sum_{i=0}^{k-1} \frac{D^2 f^{2n}}{D f^{2n}} \circ f^{2ni} \cdot Df^{2ni} \right| (x) \, dx$$
equals 
\begin{small}
\begin{eqnarray*}
\int_{f^{-n}(a)}^{f^n(a)} \left| \frac{D^2 f^{2n}}{D f^{2n}} \right| 
&& \!\!\!\!\!\!\!\!\!\!\!\!\!\!(x) \, dx
+  \int_{f^{(-1-2k)n}(a)}^{f^{(1-2k)n}(a)} \left| \frac{D^2 f^{2n}}{D f^{2n}} \circ f^{2(k-1)n} \cdot Df^{2(k-1)n}\right|  (x) \, dx \, + \\
&+& \sum_{j=1}^{k-1} \int_{f^{(-1-2j)n}(a)}^{f^{(1-2j)n}(a)} \left| \frac{D^2 f^{2n}}{D f^{2n}} \circ f^{2(j-1)n} \cdot Df^{2(j-1)n} + 
\frac{D^2 f^{2n}}{D f^{2n}} \circ f^{2jn} \cdot D f^{2jn} \right|  (x) \, dx,
\end{eqnarray*}
\end{small}which by change of variables transforms into
\begin{eqnarray*}
\int_{f^{-n}(a)}^{f^n(a)} \left| \frac{D^2 f^{2n}}{D f^{2n}} \right|  (x) \, dx
&+&  \int_{f^{-3n}(a)}^{f^{-n}(a)} \left| \frac{D^2 f^{2n}}{D f^{2n}} \right|  (x) \, dx \,\, + \\
&+& (k-1) \int_{f^{-n}(a)}^{f^{n}(a)} \left| \frac{D^2 f^{2n}}{D f^{2n}} \circ f^{-2n} \cdot Df^{-2n} + \frac{D^2 f^{2n}}{D f^{2n}} \right|  (x) \, dx.
\end{eqnarray*}
If we divide by $2nk$, this yields
$$\frac{\var (\log Df^{2nk})}{2nk} = \frac{1}{k} \cdot \frac{\var (\log Df^{2n})}{2n} + 
\frac{k-1}{k} \cdot \frac{1}{2n} \int_{f^{-n}(a)}^{f^{n}(a)} \left| \frac{D^2 f^{2n}}{D f^{2n}} \circ f^{-2n} \cdot Df^{-2n} + \frac{D^2 f^{2n}}{D f^{2n}} \right|  (x) \, dx.$$
Finally, letting $k$ go to infinity, we obtain (\ref{eq:rara}) in the limit.
\end{proof}

%%%%%%%%%%%%%%%%%%%%%%%%%%%%%%%%%%%%%%%%%%%%%%%%%%%%%%%%%%%%%%%%%%%%%%

\section{Localizing the asymptotic distortion}
\label{s:local}

In the previous proofs of Theorems \ref{t:vanish} and \ref{t:fund-ineq}, we used the characterizations of the asymptotic distortion in terms of conjugacy and approximation by $L^1$ coboundaries given by Proposition \ref{doble-criterio}. Here, we take yet another viewpoint and show that, for diffeomorphisms without interior fixed points, the asymptotic distortion can be approximated by the localization of the distortion of a finite iterate. At the end of the section, this will yield another proof of Theorem \ref{t:fund-ineq}, and will be used to prove Theorems \ref{t:continu} and \ref{t:invariant} in the next sections.

\begin{prop}\label{cor-dist-loc}
Let $f \in \mathrm{Diff}^{1+\mathrm{bv},\Delta}_+([0,1])$ be such that $f(x) > x$ for all $x \in (0,1)$, and let $p \in (0,1)$. Then 
\begin{equation}\label{as-limit}
\dist (f) = \lim_{N \to \infty} \mathrm{var} \big( \log Df^{2N} ; [f^{-N}(p), f^{-N + 1}(p)] \big).
\end{equation}
\end{prop}

This will be obtained as a corollary of the following technical lemma:

\begin{lem} \label{localization}
Let $f \in \mathrm{Diff}^{1+\mathrm{bv}}_+([0,1])$ be such that $f(x) > x$ for all $x \in (0,1)$, and let $p \in (0,1)$. There exists $N \geq 1$ such that, for all $n > 2N$, the value of the expression
$$\big| \mathrm{var} ( \log Df^n ) - (n-2N) \, \mathrm{var} \big( \log Df^{2N} ; [f^{-N}(p), f^{-N + 1}(p)] \big) \big| $$
is bounded from above by
\begin{small}$$n \big[ \mathrm{var} \big( \log Df; [0,f^{-N}(p)] \big) + \mathrm{var} \big( \log Df ; [f^{N}(p),1] \big) \big] 
+ (4N-1) \, \mathrm{var} \big( \log Df ; [f^{-N}(p), f^{N}(p) ] \big).$$\end{small}
\end{lem}

\vspace{0.1cm}

\begin{proof}[Proof of the implication] If we divide by $n$ each side of the inequality of the preceding lemma and let $n$ go to infinity, we obtain
\begin{small}$$\big| \dist (f) - \mathrm{var} \big( \! \log Df^{2N} ; [f^{-N}(p), f^{-N + 1}(p)] \big) \big| \leq 
\mathrm{var} \big( \! \log Df; [0,f^{-N}(p)] \big) + \mathrm{var} \big( \! \log Df ; [f^{N}(p),1] \big).$$
\end{small}Letting now $N$ go to infinity, each term of the right-side sum above converges to $0$, thus yielding the announced equality (\ref{as-limit}). 
\end{proof}

\vspace{0.1cm}

\begin{proof}[Proof of Lemma \ref{localization}] For each $n \geq 1$ we have 
\begin{equation}\label{suma-total}
\var (\log Df^n) = \sum_{\ell = -\infty}^{\infty} \var \big( \log Df^n ; f^{\ell} ([p, f(p)]) \big).
\end{equation}
There are 5 types of indices $\ell$ to analyze:  
\begin{enumerate}

\item If $\ell \leq -n - N$, then
$$ \var \big( \log Df^n; f^{\ell} ([p, f(p)]) \big) \leq \sum_{k=0}^{n-1} \mathrm{var} \big( \log Df; f^{k+\ell} ([p, f(p)]) \big),$$
and all the intervals involved in the right-hand side sum, namely,
$$[f^{\ell} (p), f^{\ell + 1} (p)], \ldots, [f^{\ell+n-1}(p), f^{\ell + n} (p)]$$
are contained in $[0, f^{-N}(p)]$, since $\, \ell + n   \leq -N $. 

\item If $-n-N \!<\! \ell \!\leq\! -n + N$, then $\, \var \big( \log Df^n; f^{\ell} ([p,f(p)]) \big) \,$ is smaller than or equal to   
$$\sum_{k=0}^{-\ell - N -1} \mathrm{var} \big( \log Df; f^{k+\ell} ([p, f(p)]) \big)
+ \sum_{k = -\ell - N}^{n-1} \mathrm{var} \big( \log Df; f^{k+\ell} ([p, f(p)]) \big).$$
The intervals involved by the first sum above are all contained in $[0,f^{-N}(p)]$. Moreover, since $\ell + n - 1 \leq N - 1$, the second sum is bounded from above by 
$$\sum_{k = 0}^{2N-1} \mathrm{var} \big( \log Df; f^{k} ([f^{-N}(p), f^{-N+1}(p)]) \big) 
= \mathrm{var} \big( \log Df; [f^{-N} (p), f^N (p)]\big).$$
Notice that this case arises for $2N$ possible values of $\ell$.

\item If $-n + N < \ell \leq -N$, then for $x \in [f^{\ell}(p), f^{\ell+1}(p)]$ we have $f^{-\ell-N}(x) \in [f^{-N}(p), f^{-N+1}(p)]$ and $f^{-\ell+N}(x) \in [f^N(p), f^{N+1}(p)]$. Hence, using the relation,  
$$Df^n (x) = Df^{n + \ell - N} ( f^{-\ell+N} (x)) \cdot  Df^{2N} (f^{-\ell-N}(x)) \cdot Df^{-\ell-N} (x),$$ 
we conclude that the value of 
$$\big| \mathrm{var} \big( \log Df^n; f^{\ell} ([p,f(p)]) \big) - \mathrm{var} \big( \log Df^{2N}; [f^{-N}(p), f^{-N+1}(p)]) \big) \big|$$
is bounded from above by 
$$\mathrm{var} \big( \log Df^{-\ell - N}; f^{\ell} ([p,f(p)]) \big) + \mathrm{var} \big( \log Df^{n+\ell - N}; f^{N}([p,f(p)]) \big).$$
In its turn, this is bounded from above by 
$$\sum_{k=0}^{-\ell - N - 1} \mathrm{var} \big( \log  Df ; f^{k + \ell} [p,f(p)] \big) + 
\sum_{k=0}^{ n + \ell - N - 1} \mathrm{var} \big( \log  Df ; f^{k + N} [p,f(p)] \big),$$
with the intervals involved in the first (resp. second) sum being contained in $[0,f^{-N}(p)]$ (resp. $[f^N(p),1]$). Notice that this case arises for $\, n \!-\! 2N \,$ values of $\ell$.

\item If $-N < \ell \leq N-1$, then $\, \mathrm{var} \big( \log Df^n ; f^{\ell} ([p,f(p)]) \big) \,$ is bounded from above by 
$$\sum_{k=0}^{N-\ell-1} \mathrm{var} \big( {\log Df; f^{k+\ell} ([p,f(p)])} \big) 
+ \sum_{k=N-\ell}^{n-1} \mathrm{var} \big( {\log Df; f^{k+\ell} ([p,f(p)])} \big),$$
which in its turn is smaller than or equal to 
$$ \mathrm{var} \big( \log Df; [f^{-N} (p), f^N (p)]\big) + 
\sum_{k=N-\ell}^{n-1} \mathrm{var} \big( {\log Df; f^{k+\ell} ([p,f(p)])} \big),$$
with the last sum involving only intervals contained in $[f^N(p),1]$. Notice that this case arises for $\, 2N \! - \! 1 \,$ different values of $\ell$.

\item If $N \leq \ell$, then 
$$\mathrm{var} \big( \log Df^n ; f^{\ell} ([p,f(p)]) \big) \leq \sum_{k=0}^{n-1} \mathrm{var} \big( {\log Df ; f^{k+\ell} ([p,f(p)])} \big),$$
and all the intervals involved in the last sum are contained in $[f^N (p), 1]$. 

\end{enumerate}

We come back to equality (\ref{suma-total}). Notice that the variation of $\log (Df)$ on each interval of the form $f^{\ell} ([p,f(p)])$ contained in either $[0,f^{-N}(p)]$ or $[f^N(p),1]$ appears precisely $n$ times along the preceding estimates. Putting all of this together, one easily obtains the desired estimate just using the triangle inequality. 
\end{proof}

\vspace{0.1cm}

\begin{proof}[Another proof of Theorem \ref{t:fund-ineq}] Let $f\in \Diff^{2,\Delta}_+([0,1])$. Assume without loss of generality that $f(x)>x$ for all $x\in(0,1)$, and denote by $X$ and $Y$ its left and right Szekeres vector fields. Fix $a\in(0,1)$, define $\psi_X:t\mapsto f_t(a)$ and $\psi_Y:t\mapsto f^t(a)$, and identify the Mather invariant $M_f$ of $f$ with its representative $M^{a,a}_f:=\psi_Y^{-1}\circ\psi_X$. Recall that the maps $\psi_X$ and $\psi_Y$ satisfy $\psi \circ T = f \circ \psi$ for $T := T_1$, the translation by $1$ on $\R$. Therefore, for each positive $m,n$ we have, letting $k:= m+n$: 
\begin{equation}\label{ren}
M_f = T_{-m} \circ ( \psi_Y )^{-1} \circ f^k \circ \psi_X \circ T_{-n}.
\end{equation}
This yields
$$DM_f (t) 
= \frac{D \psi_X (t-n)}{D \psi_Y \big( (\psi_Y)^{-1} f^k \, \psi_X (t-n) \big) } \cdot Df^k (\psi_X (t-n)) 
= \frac{D \psi_X (t-n)}{D \psi_Y ( M_f (t) + m ) } \cdot Df^k (\psi_X (t-n)),$$
hence
\begin{equation}\label{eq:der-M}
DM_f(t) = \frac{X (\psi_X (t-n)) }{Y (\psi_Y (M(t)+m))} \cdot Df^k (\psi_X (t-n)).
\end{equation}
This easily implies that 
$$\left| \var ( \log DM_f  ) - \var (\log Df^k ; [f^{-n}(a), f^{-n+1}(a)])\right|$$
is bounded from above by 
$$\var \big( \log(X); [f^{-n}(a), f^{-n+1}(a)] \big) + \var \big( \log(Y) ; [f^m(a),f^{m+1}(a)] \big).$$
By \eqref{e:logX2}, the latter expression is smaller than or equal to
$$\big| \log (Df(0)) \big| + \big| \log (Df (1)) \big| + \mathrm{var} (\log Df; [0,f^{-n}(a)] + \mathrm{var} (\log Df ; [f^m (a), 1]).$$
Letting $m=n = N \to \infty$, the last two terms above converge to $0$, and Proposition \ref{cor-dist-loc} yields 
\begin{equation}\label{ren-dist}
\var (\log Df^k; [f^{-n}(a), f^{-n+1}(a)]) = \var (\log Df^{2N} ; [f^{-N}(a), f^{-N+1}(a)] )\to \mathrm{dist}_{\infty} (f).
\end{equation}
Putting everything together, we finally obtain 
$$\big| \var ( \log  DM_f ) - \dist (f) \big| \leq \big| \log (Df(0)) \big| + \big| \log (Df (1)) \big|,$$
thus closing the proof.
\end{proof}

%%%%%%%%%%%%%%%%%%%%%%%%%%%%%%%%%%%%%%%%%%%%%%%%%%%%%%%%%%%%%%%%%%%%%%%

\section{(Dis)continuity of the asymptotic distortion: a proof of Theorem \ref{t:continu}}
\label{s:continu}

\noindent{\bf {\em Upper semicontinuity of $\dist$.}} We start this section by checking that $\dist$ is upper semicontinuous. To do this, remind that for a general subadditive sequence $a_n$, one has
\begin{equation}\label{inf}
\lim_n \frac{a_n}{n} = \inf_n \frac{a_n}{n}.
\end{equation}
Now, given $f \in \mathrm{Diff}_+^{1+\mathrm{bv}} ([0,1])$ and $\varepsilon > 0$, let $N$ be such that, for all $n \geq N$, 
$$\frac{\mathrm{var} (\log Df^n)}{n} < \dist (f) + \frac{\varepsilon}{2}.$$
If $g$ is close enough to $f$ in the $C^{1+\mathrm{bv}}$ topology, then 
$$\big| \mathrm{var} (\log Dg^N ) - \mathrm{var} (\log Df^N ) \big| < \frac{N \varepsilon}{2}.$$
Therefore, by the triangle inequality, 
$$\frac{\mathrm{var} (\log Dg^N)}{N} \leq \frac{\mathrm{var} (\log Df^N)}{N} + \frac{\varepsilon}{2} < \dist (f) + \varepsilon.$$
By (\ref{inf}), this yields
\, $\dist (g) < \dist (f) + \varepsilon,$ \, 
which shows the upper semicontinuity of $\dist$.\medskip

An alternative (and much shorter) argument that uses Proposition \ref{doble-criterio} proceeds as follows. According to (\ref{primera}), we have 
$$\dist (f) = \inf_{h \in \mathrm{Diff}^{1+\mathrm{bv}}_+([0,1])} \var (\log D (hfh^{-1})).$$
Now, for a fixed $h$, the function $f \mapsto \var (\log D (hfh^{-1}))$ is continuous. Therefore, as the infimum along a family of continuous functions, the asymptotic distortion is upper semicontinuous.

\vspace{0.1cm}

\begin{rem} In an analogous way, one can use the characterization (\ref{segunda}) to show the upper semicontinuity of $\dist$ in the space of $C^1$ diffeomorphisms with absolutely continuous derivative. We leave the details to the reader.
\end{rem}

\noindent{\bf {\em Continuity of $\dist$ on $\mathrm{Diff}_+^{1+\mathrm{bv},\Delta}([0,1])$.}} 
We now proceed to the proof of the continuity of $\dist$ at an arbitrary element 
$f \in \mathrm{Diff}_+^{1+vb,\Delta} ([0,1])$ ({\em i.e.} the first half of Theorem \ref{t:continu}). Since for every diffeomorphism $h$ one has $\mathrm{var} (\log D h) = \mathrm{var} (\log D h^{-1})$, we may assume that $f(x) \!>\! x$ for all $x \in (0,1)$. Fix $p \in (0,1)$. Lemma \ref{localization} yields that, for $n > 2N$, the value of 
\begin{equation}\label{a-estimar-dist}
\left| \frac{\mathrm{var} (\log Df^n)}{n} - \mathrm{var} \big( \log Df^{2N} ; [f^{-N}(p), f^{-N+1}(p)] \big) \right|
\end{equation}
is smaller than or equal to 
$$\mathrm{var} \big( \log Df; [0,f^{-N}(p)] \big) + \mathrm{var} \big( \log Df; [f^{N}(p),1] \big) + 
\frac{4N}{n} \mathrm{var} ( \log Df ) + \frac{2N}{n} \mathrm{var} ( \log Df^{2N}).$$
Given $\varepsilon > 0$, let $\delta > 0$ be such that 
$$\mathrm{var} \big( \log Df ; [0,\delta] \big) < \frac{\varepsilon}{8}
\qquad \mbox{and} \qquad
\mathrm{var} \big( \log Df ; [1-\delta,1] \big) < \frac{\varepsilon}{8}.$$
Fix a large-enough $N$ such that $f^{-N} (p) < \delta$ and $f^N (p) > 1-\delta$. Next consider any diffeomorphism $g \in \mathrm{Diff}_+^{1+vb,\Delta} ([0,1])$ that is close enough to $f$ in the $C^{1+\mathrm{bv}}$ topology so that the following conditions are satisfied:
\begin{itemize}

\item $g(x) > x$ for all $x \in (0,1)$,

\item $g^{-N} (p) < \delta$ and $g^N (p) > 1-\delta$,

\item $\mathrm{var} \big( \log Dg; [0,\delta] \big) < \frac{\varepsilon}{8} \,$ 
and $\, \mathrm{var} \big( \log Dg; [1-\delta,1] \big) < \frac{\varepsilon}{8}$,

\item $\mathrm{var} (\log Dg) < \mathrm{var} (\log Df) +1 \,$ and $\, \mathrm{var} (\log Dg^N ) < \mathrm{var} (\log Df^N ) +1$, 

\item $ \big| \mathrm{var} \big( \log Df^{2N} ; [f^{-N}(p), f^{-N+1}(p)] \big) - \mathrm{var} \big( \log Dg^{2N} ; [g^{-N}(p), g^{-N+1}(p)] \big) 
\big| < \frac{\varepsilon}{4}$.

\end{itemize}

\noindent Now consider any large-enough integer $n$ such that $n > 2N$ and 
$$\frac{4N}{n} \big( \mathrm{var} ( \log Df ) + 1 \big) + \frac{2N}{n} \big( \mathrm{var} ( \log Df^{2N}) + 1 \big) < \frac{\varepsilon}{8}.$$ 
Since the estimate given for expression (\ref{a-estimar-dist}) holds when replacing $f$ by $g$, we easily conclude from the previous conditions that 
$$\left| \frac{\mathrm{var} (\log Df^n )}{n} - \frac{\mathrm{var} (\log Dg^n )}{n}\right| < \varepsilon.$$
Since this holds for any large-enough $n$, passing to the limit we conclude that 
$$\big| \dist(f) - \dist(g) \big| \leq \varepsilon.$$
Since $\varepsilon > 0$ was arbitrary, this shows the continuity of $\dist$ at $f$.

\vspace{0.1cm}

\begin{rem} If $f$ is linear on both intervals $[0,f^{-N} (p)]$ and $[f^N (p), 1]$, then the estimate for expression (\ref{a-estimar-dist}) above becomes 
\begin{small}
$$\left| \frac{\mathrm{var} (\log Df^n)}{n} - \mathrm{var} \big( \log Df^{2N} ; [f^{-N}(p), f^{-N+1}(p)] \big) \right|
\leq \frac{4N}{n} \mathrm{var} ( \log Df ) + \frac{2N}{n} \mathrm{var} ( \log Df^{2N}).$$
\end{small}Passing to the limit in $n$, this yields
$$\dist (f) = \mathrm{var} \big( \log Df^{2N} ; [f^{-N}(p), f^{-N+1}(p)] \big).$$
Slightly more generally, it is not hard to check along the same lines that if $g \in \mathrm{Diff}_+^{1+\mathrm{bv}, \Delta}([0,1])$ is linear on intervals $[0,\delta]$ and $[1-\delta,1]$ and $k$ is an integer such that either $g^k (\delta) \geq 1- \delta$ or $g^k (1-\delta) \leq \delta$, then $\dist (g)$ equals either 
$$\mathrm{var} \big( \log Dg^k ; [\delta, g (\delta)] \big) \qquad \mbox{or} \qquad
\mathrm{var} \big( \log Dg^k ; [g (1-\delta), 1-\delta] \big),$$
respectively. If $g$ is of class $C^2$ and the endpoints are hyperbolic fixed points for it, this allows one to localize $\dist (g)$. Indeed, according to the Sternberg-Yoccoz' linearization theorem \cite{sternberg,yoccoz}, the map $g$ is $C^2$ conjugate to a diffeomorphism that is linear on neighborhoods of both endpoints.
\end{rem}

\begin{rem} \label{preuve-alt}
The previous remark may be used to give an alternative proof of Proposition~\ref{prop-unica}. Indeed, assume that $f(x)>x$ for all $x\in(0,1)$ for a given $f\in \Diff^{2,\Delta}_+([0,1])$ with hyperbolic fixed points, and denote by $X$ and $Y$ the left and right Szekeres vector fields. Since all terms in the inequality to be proved are invariant under $C^2$ conjugacy, we may assume that $f$ is linear on neighborhoods of $0$ and $1$. Now remind equality (\ref{eq:der-M}), namely (for $t \in [0,1]$)
$$DM_f(t) = \frac{X (\psi_X (t-n)) }{Y (\psi_Y (M_f(t)+m))} \cdot Df^k (\psi_X (t-n)).$$
We may choose $m=n$ (hence $k=2n$) large enough so that $f$ is linear on $[0,\psi_X (1-n)]$ and $[\psi_Y (n), 1]$. Then the equality
\begin{small}
\begin{equation}\label{eq:v-l}
\var (\log DM_f) 
= \int_0^1 \left| D \big[ \log X \circ \psi_X \circ T_{-n} - \log Y \circ \psi_Y \circ T_n \circ M_f + \log Df^{2n} \circ \psi_X \circ T_{-n} \big] \right| 
\end{equation}
\end{small}implies that $\var (\log DM_f)$ is smaller than or equal to the sum 
\begin{small}
$$\int_0^1 \left| D \big[ \log X \circ \psi_X \circ T_{-n} \big]  \right| 
+ \int_0^1 \left| D \big[ \log Y \circ \psi_Y \circ T_n \circ M_f \big] \right| 
+ \int_0^1 \left| D \big[ \log Df^{2n} \circ \psi_X \circ T_{-n } \big] \right| .$$ 
\end{small}This equals 
\begin{small}
$$\var \big( \! \log (X); [f^{-n}(a), f^{-n+1}(a)] \big) + \var \big( \! \log (Y); [f^n(a), f^{n+1}(a)] \big) 
+ \var \big( \! \log Df^{2n}; [f^{-n}(a),f^{-n+1} (a)] \big)$$
\end{small}which, in its turn, precisely coincides with 
\begin{equation}\label{eq:strict}
| \log Df (0)| + | \log Df (1) | + \dist (f).
\end{equation}
Indeed, the coincidence of the first two terms follows from that the expressions of $X$ and $Y$ are explicit on linearity intervals (namely, $X(x) = \lambda \, x$ and $Y(x) = \mu \, (1-x)$, where $\lambda := \log Df(0)$ and $\mu := \log Df (1)$), and the coincidence 
of the third term follows from the previous remark.

In order to see that expression (\ref{eq:strict}) is actually strictly larger than $\var (\log DM_f)$, we need to check that some cancellation occurs in (\ref{eq:v-l}). This follows from the next two facts:
\begin{itemize}
\item Since $[f^{-n}(a),f^{-n+1}(a)]$ and $[f^n(a),f^{n+1}(a)]$ lie in the domains of linearity of $f$ where the vector fields $X$ and $Y$ become explicit, one easily checks that $D \log X$ (resp. $-D \log Y$) is positive on the first (resp. second) interval, hence the part of the integral associated to $X$ and $Y$ in (\ref{eq:v-l}) is strictly positive;
\item One has
\begin{footnotesize}\begin{eqnarray*}
\int_0^1 \! D \big[ \log Df^{2n} \circ \psi_X \circ T_{-n } \big] 
\! &=& \! \log \left( \frac{Df^{2n} (f^{-n+1}(a))}{Df^{2n} (f^{-n}(a))} \right) 
\,= \, \log \left( \frac{Df (f^n (a)) \cdot Df^{2n-1} (f^{-n+1}(a))}{Df^{2n-1} (f^{-n+1}(a)) \cdot Df (f^{-n}(a))} \right) \\
\!&=&\! \log \left( \frac{Df (f^n (a))  }{  Df (f^{-n}(a))} \right) 
\,\, = \,\,\, \log \left( \frac{e^{\mu}}{e^{\lambda}} \right) \,\, = \,\,\, \mu - \lambda \,\,\, < \,\,\, 0,
\end{eqnarray*}
\end{footnotesize}hence the part of the integral associated to \, $\log Df^{2n}$ \, in (\ref{eq:v-l}) must be negative on a set of positive measure.
\end{itemize}
\end{rem}

\vspace{0.2cm}

\noindent{\bf {\em An example of discontinuity of $\dist$.}} We next proceed to build a sequence $f_n$ of pairwise conjugate $C^\infty$ diffeomorphisms of $[0,1]$ with zero asymptotic distortion that converges in the $C^\infty$ topology towards a $C^\infty$ diffeomorphism $f_{\infty}$ which, on the other hand, has nonzero asymptotic distortion, thus proving the discontinuity of $\dist$. (Notice that all these diffeomorphisms are necessarily parabolic at $0$ and $1$.) 

\begin{figure}[h!]
\centering
\includegraphics[width=9cm]{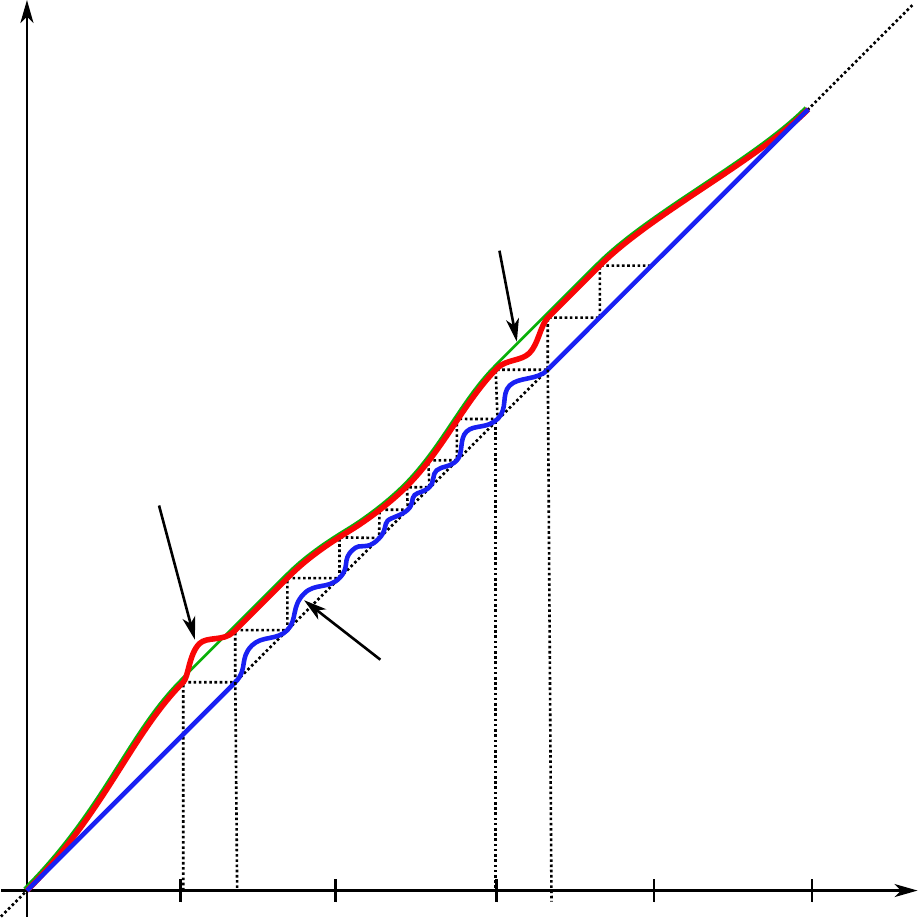}
\put(-150,64){${\phi_n}$}
\put(-220,118){${f_n}$}
\put(-120,190){${g_n}$}
\put(-209,-7){${\frac15}$}
\put(-166,-7){${\frac25}$}
\put(-120,-7){${\frac35}$}
\put(-77,-7){${\frac45}$}
\put(-32,-7){${1}$}
\caption{The diffeomorphisms $g_n$ (in green), $\phi_n$ (in blue), and $f_n$ (in red).}
\label{fig:f-phi}
\end{figure}
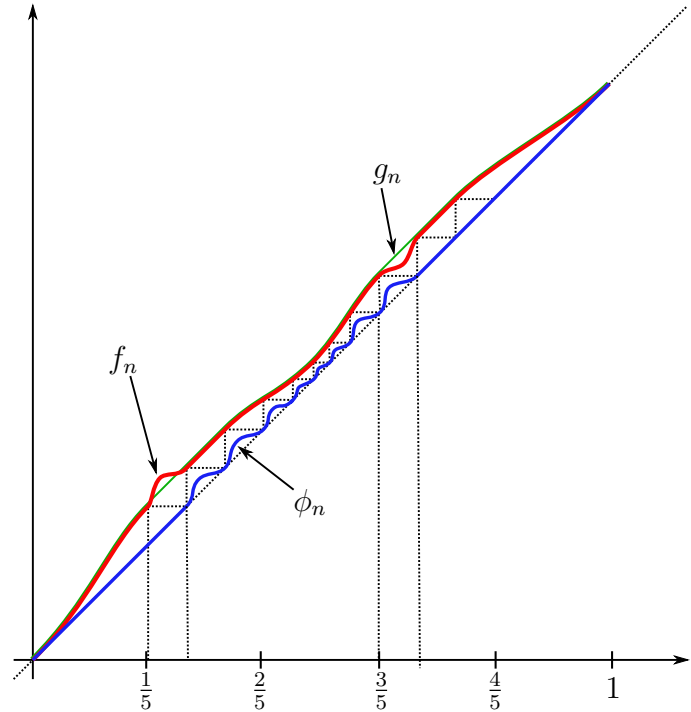

%\vspace{0.4cm}

\noindent{\em Preliminary definitions.} Start with a smooth vector field $Z_0$ on $[0,1]$ that satisfies the next two properties:
\begin{itemize}
\item It is non-vanishing on $(0,1)$ and $C^1$-flat at the endpoints;
\item It is constant equal to some positive number $\nu$ on $[\frac15,\frac45]$, small enough so that the forward orbit of $p=\frac15$ by the time-$1$ map $g$ of $Z_0$ has at least three points in each interval $[\frac15,\frac25]$ and $[\frac35,\frac45]$. 
\end{itemize}
Denote by $J$ the fundamental interval $[p,g(p)]$, thus equal to $[p,p+\nu]$, and by $K$ some fundamental interval $[g^k(p),g^{k+1}(p)]$ of the form $[q,q+\nu]$ contained in $[\frac35,\frac45]$. 

%\marginpar{\textcolor{red}{REFEREE: It would be helpful to label J, K, p, q and possibly 1/5, 4/5 on the figures. I had to draw the labels myself to keep track! Similarly for figure 2. on page 45.}}
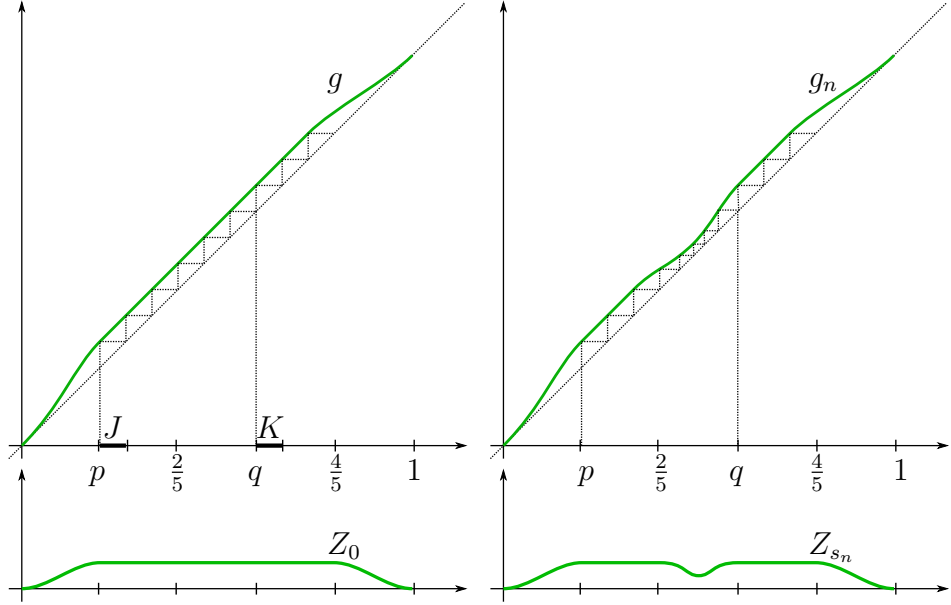
\begin{figure}[h!]
\centering
\includegraphics[width=12.5cm]{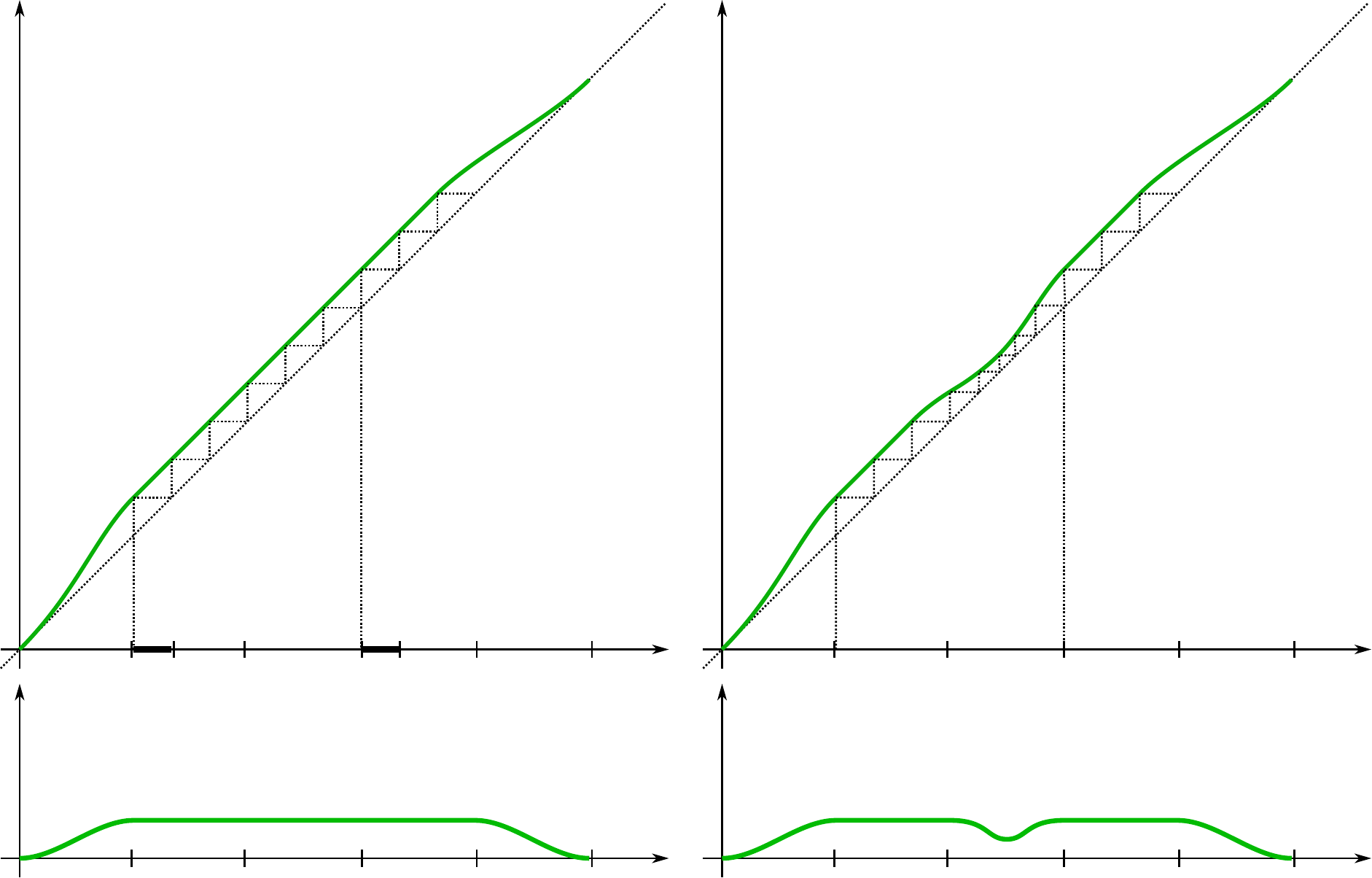}
\put(-320,62){${J}$}
\put(-325,46){${p}$}
\put(-262,62){${K}$}
\put(-295,46){${\frac25}$}
\put(-265,46){${q}$}
\put(-235,46){${\frac45}$}
\put(-205,46){${1}$}
\put(-235,18){${Z_0}$}
\put(-235,195){${g}$}
\put(-140,46){${p}$}
\put(-112,46){${\frac25}$}
\put(-83,46){${q}$}
\put(-53,46){${\frac45}$}
\put(-21,46){${1}$}
\put(-53,18){${Z_{s_n}}$}
\put(-53,195){${g_n}$}
\caption{On the left, the vector field $Z_0$ (below) and its time-1 map $g$ (above), with the fundamental intervals $J$ and $K$ in bold. On the right: a vector field $Z_{s_n}$ with $s_n \in (0,1)$ (below) and its time-1 map $g_{n}$ (above).}
\label{fig:Z-g}
\end{figure}

Now let $Z_1$ be another smooth vector field on $[0,1]$ vanishing only at $0$, $\frac12$ and $1$, and coinciding with $Z_0$ outside $[\frac25,\frac35]$. For every $s\in(0,1)$, let $Z_s:=(1-s)Z_0+sZ_1$. By continuity, there exists a sequence of parameters $s_n$ converging to $1$ such that, for every $n\in\N$, the point $q$ belongs to the forward orbit of $p$ under the time-$1$ map $g_{n}$ of $Z_{s_n}$. 

Denote by $g_\infty$ the time-$1$ map of $Z_1$. Notice that the sequence $g_n$ converges towards $g_\infty$ in the $C^\infty$ topology, and is constant equal to $g$ on $[0,g(p)]\cup[q,1]$. Finally, let $\phi$ be a smooth diffeomorphism of $[0,1]$ supported on $g(J)=[p+\nu,p+2\nu]$, and let $\psi$ be the diffeomorphism 
$g^{k-1}\circ\phi\circ g^{-(k-1)}$, which is supported on $K$.

\vspace{0.4cm}

\noindent{\em Definition of $f_n$.} For every $n\in\N\cup\{\infty\}$, let $f_n$ be the diffeomorphism coinciding with $g_n$ outside $J\cup K$, equal to $\phi\circ g_n$ on $J$, and equal to $g_n\circ \psi^{-1}$ on $K$. In particular, $f_n=f_0$ on $[0,g(p)]\cup[q,1]$. The sequence $f_n$ clearly converges towards $f_\infty$ in the $C^\infty$ topology. 

\vspace{0.4cm}

\noindent{\em Asymptotic distortions.}
It is straightforward to check that, for every $n \! \in \! \N$, there is a unique $C^{\infty}$ diffeomorphism $\phi_n$ that:
\begin{itemize}
\item is equal to the identity on $[0,g(p)]\cup [g(q),1]$, 
\item is equal to $\phi$ on $g(J)$, 
\item satisfies $\, \phi_n\circ g_n = g_n\circ \phi_n \,$ on $[g(p),q]$.
\end{itemize}
By construction, this diffeomorphism also satisfies $f_n = \phi_n\circ g_n\circ \phi_n^{-1}$. In particular, since $g_n$ is the time-$1$ map of a smooth vector field on $[0,1]$, so is $f_n$. Therefore, $f_n$ has a trivial Mather invariant, which by Theorem~A implies that $\dist(f_n)=0$. Since $p$ and $q$ are in the same orbit for each $f_n$, Lemma \ref{l:Cr-conj} implies that all $f_n$ are pairwise $C^{\infty}$ conjugate. 
 
Nevertheless, one can define a Mather invariant for $f_\infty$ on both $[0,\frac12]$ and $[\frac12,1]$, and we claim that both are nontrivial, so that $\dist(f_\infty)>0$ by Theorem~A and Lemma \ref{muchos}. Indeed, let us denote by $\hat{f}_\infty$ (resp. $\hat{g}_\infty$) the diffeomorphism of $[0,\frac12]$ without interior fixed points induced by $f_\infty$ (resp. $g_\infty$). As the time-$1$ map of a $C^1$ flow, $\hat{g}_\infty$ has a trivial Mather invariant. However, $\hat{f}_{\infty}$ is obtained from $\hat{g}_{\infty}$ by a perturbation supported on a single fundamental domain; hence, by Lemma \ref{break}, it has a nontrivial Mather invariant. 

%%%%%%%%%%%%%%%%%%%%%%%%%%%%%%%%%%%%%%%%%%%%%%%%%%%%%%%%%%%%%%%%%%%%%%%%%

\section{On the $C^1$ invariance of $\dist$: a proof of Theorem \ref{t:invariant}} 
\label{s:invariant}

We next give another application of Proposition \ref{cor-dist-loc}:

\begin{proof}[Proof of Theorem \ref{t:invariant}] Given two  $C^2$ diffeomorphisms $f,g$ of a compact 1-manifold 
such that $f = h g h^{-1}$ for a $C^1$ diffeomorphism $h$, our task is to prove that $\dist (f) = \dist (g)$. 
Let us first consider the case of the circle. On the one hand, if $f,g$ have irrational rotation number, then they both have zero asymptotic distortion \cite{mio2}. On the other hand, if they have rational rotation number, then there exists $k \geq 1$ such that both $f^k$ and $g^k$ have fixed points. We cut the circle at a fixed point of $f^k$, so that may view this map as a diffeomorphism of the unit interval. We independently cut the circle at the image under $h$ of this point, so that $g^k$ also becomes a diffeomorphism of the unit interval, as well as the conjugating map $h$. Since $\, \dist (f^k) = k \cdot \dist(f) \,$ and $\, \dist (g^k) = k \cdot \dist(g), \,$ if we show that $f^k$ and $g^k$ have the same asymptotic distortion, the same will be true for $f$ 
and $g$. Thus, we have reduced the general case to that of the interval.

We next deal with the case where $f,g$ both belong to $\mathrm{Diff}^{2,\Delta}_+ ([0,1])$: the general case for interval diffeomorphisms follows from this together with Lemma \ref{muchos}. By Corollary \ref{cor-dist-loc}, it suffices to show that 
\begin{equation} \label{est-final}
\big| \mathrm{var} \big( \log Df^{2N} ; [f^{-N}(p), f^{-N+1}(p)] \big) 
- \mathrm{var} \big( \log Dg^{2N} ; [g^{-N}(h(p)), g^{-N+1}(h(p))] \big) \big|
\end{equation}
converges to zero as $N$ goes to infinity. 

Now remind that the uniqueness of the left and right vector fields for diffeomorphisms and the fact that their flows coincide with the corresponding $C^1$ centralizers imply that $h$ sends the vector fields of $f$ to those of $g$. It then follows from the equality $Dh = (X_g \circ h) / X_f$ that $h$ is of class $C^2$ on $(0,1)$. Using that $g^{2N} = h f^{2N} h^{-1}$and the subadditivity property of $\mathrm{var} (\log D (\cdot) )$,  we can hence estimate the value of (\ref{est-final}) from above by 
$$\mathrm{var} \big( \log Dh^{-1}; [g^{-N}(p), g^{-N+1}(p)])\big) + \mathrm{var} \big( \log Dh; [f^{N}(p), f^{N+1}(p)] \big),$$
which coincides with
\begin{equation}\label{est-final-final}
\mathrm{var} \big( \log Dh; [f^{-N}(p), f^{-N+1}(p)])\big) + \mathrm{var} \big( \log Dh; [f^{N}(p), f^{N+1}(p)] \big).
\end{equation}
Finally, we claim that both terms of this sum converge to zero as $N$ goes to infinity. 
To show this, we will perform the explicit estimates for the first term, as those of the second one are analogous. We have
\begin{footnotesize}\begin{eqnarray*}
\mathrm{var} \big( \log Dh; [f^{-N}(p), f^{-N+1}(p)] \big) 
\!&=&\! 
\int_{f^{-N}(p)}^{f^{-N+1}(p)} \left| \frac{D^2 h}{ D h} \right| \\
\!&=&\!
\int_{f^{-N}(p)}^{f^{-N+1}(p)} \left|  \frac{DX_g}{X_g} \circ h \cdot Dh - \frac{DX_f}{X_f}\right| 
\,\,= \,\,
\int_{f^{-N}(p)}^{f^{-N+1}(p)} \left|  \frac{DX_g \circ h - DX_f}{X_f} \right|.
\end{eqnarray*}
\end{footnotesize}When $N$ goes to infinity, the supremum of the numerator above on the interval of integration converges to 
$$|DX_g (0) - DX_f (0)| = |Dg (0) - Df (0)|,$$
which equals $0$ since $f$ and $g$ are conjugated by a $C^1$ diffeomorphism. Finally, using the fact that 
$$\int_{f^{-N}(p)}^{f^{-N+1}(p)} \left| \frac{dx}{X_f (x)} \right|  =  1,$$ 
one easily deduces that $\, \mathrm{var} \big( \log Dh; [f^{-N}(p), f^{-N+1}(p)] \big) \,$ converges to $0$ as $N$ goes to infinity, as announced. 
\end{proof}
\vspace{0.1cm}

\begin{rem} \label{rem-invariance-dist}
In the proof of Theorem \ref{t:invariant} above, the key step was the case where $f,g$ belong to $\mathrm{Diff}_+^{2,\Delta} ([0,1])$. Actually, for this case, there are alternative (and quite clarifying) arguments in two different situations:
\begin{itemize}
\item If both endpoints are parabolic fixed points, then the invariance of $\dist$ under $C^1$ conjugacy follows from that of the Mather invariant together with the general formula~(\ref{parabolic}).

\item If both endpoints are hyperbolic fixed points, then a theorem of Ghys and Tsuboi 
\cite{Ghys-Tsuboi} (that strongly relies on the Sternberg-Yoccoz' linearization theorem \cite{sternberg,yoccoz}) establishes that the conjugating map is of class $C^2$, which makes the invariance under conjugacy quite obvious.
\end{itemize}

\noindent Unfortunately, we couldn't find such a direct argument for the case where 
one of the endpoints is parabolic and the other one is hyperbolic.
\end{rem}

\vspace{0.1cm}

We close this section with still another application of Lemma \ref{localization}. What follows is an analog of Lemma \ref{break} (see also Remark \ref{r:mather-local}). The issue here is that, since we are in very low differentiability, we cannot appeal to the Mather invariant, and the estimates of $\dist$ need to be done ``by hand''. Nevertheless, this strongly suggests that weak forms of both Szekeres vector fields and Mather invariant should exist for low-regularity (namely, $C^{1+\mathrm{bv}}$) diffeomorphisms of the interval. We will come back to this technical issue in the separate publication \cite{article2}. 

\vspace{0.1cm} 

\begin{prop} \label{extension-bump} 
Let $f \in \Diff^{1+\mathrm{bv},\Delta}_+([0,1])$ have vanishing asymptotic distortion. If $p \in (0,1)$, then any $g \in \Diff^{1+\mathrm{bv},\Delta}_+([0,1])$ 
that is equal to $f$ outside the interval of endpoints $p,f(p)$ and different from $f$ on this interval has a nonvanishing asymptotic distortion.
\end{prop}

\vspace{0.1cm}

For the proof, we will need a compactness argument that could be traced back to the work of Kopell \cite{kopell}, which we state and prove below.

\vspace{0.1cm}

\begin{lem} 
Given $f \in \mathrm{Diff}^{1+\mathrm{bv},\Delta}_+([0,1])$ and a nontrivial $C^1$ diffeomorphism $h$ supported on a fundamental domain of $f$, the sequence of conjugates $f^N h f^{-N}$ remains bounded away from the identity in the $C^1$ topology. 
\end{lem}

\begin{proof}
Denote $V := \var (\log Df)$, and let $k \geq 1$ be such that $Dh^k (q) > e^{2V}$ holds for some $q \!\in\! [p_0, p_1]$. We claim that, for each $N \geq 0$, there exists a point $x_N \in [p_0, p_1]$ such that $D (f^N h f^{-N}) (x_N) > e^{V/k}$. Assume otherwise. Then, for a certain $M \geq 0$, one would have $ \| D (f^M h f^{-M} ) \|_{\infty} \leq e^{V / k}$, which by the chain rule would yield 
$ \| D (f^M h^k f^{-M} ) \|_{\infty} \leq e^{V}$. However, at the point $q_N \!:=\! f^N (q)$, we have $D (f^N h^k f^{-N}) (q_N) > e^V$. Indeed,
\begin{small}\begin{eqnarray*}
\log (D (f^N h^k f^{-N}) (q_N)) 
&=& \log Df^N (h^k (q) ) + \log Dh^k (q) - \log Df^N (q) \\
&\geq& \log Dh^k (q)  - \big| \log Df^N (h^k (q) ) - \log Df^N (q) \big| \\
&> &  2V - \sum^{N-1}_{i=0} \big| \log Df (f^i (h^k (q))) - \log Df (f^i(q)) \big|. 
\end{eqnarray*}
\end{small}Since both $q$ and $h^k (q)$ lie in $[p_0, p_1]$ and this interval is a fundamental domain of $f$, 
$$\sum^{N-1}_{i=0} \big| \log Df (f^i (h^k (q))) - \log Df (f^i(q)) \big| \leq \var (\log Df; [p_0,1]) \leq V.$$
We thus conclude that \, $\log (D (f^N h^k f^{-N}) (q_N))  > V,$ \, as announced.
\end{proof}

\vspace{0.1cm}

\noindent{\em Proof of Proposition \ref{extension-bump}.}
We may write $g = fh$, where $h$ is a nontrivial diffeomorphism supported on the interval of endpoints $p$ and $f(p)$. For sake of concreteness, we assume $p < f(p)$, and for each $n \in \mathbb{Z}$ we denote $p_n := f^n (p) = g^n (p)$. Since equality (\ref{as-limit}) holds for both $f$ and $g$, we have
\begin{small}
$$\dist (f) = \lim_{N \to \infty} \mathrm{var} \big( \log D f^{2N} ; [p_{-N},  p_{-N+1}] \big),
\qquad 
\dist (g) = \lim_{N \to \infty} \mathrm{var} \big( \log D g^{2N} ; [p_{-N}, p_{-N+1}] \big) .$$
\end{small}Now, for each $x \in [p_{-N}, p_{-N+1}]$ we have
$$D f^{2N} (x) = Df^N (x) \cdot Df^N (f^N(x)), \qquad Dg^{2N} (x) = Df^N (x) \cdot Dh (f^N (x)) \cdot Df^N (hf^N(x)).$$
Thus, for all \, $p_{-N} \leq a_0 < a_1 < \ldots < a_k \leq p_{-N+1}$, \, we have that 
\begin{equation}\label{cal-dist}
\sum_{i} \big| \log Dg^{2N} (a_{i+1}) - \log Dg^{2N} (a_i) \big| 
\end{equation}equals\begin{small}
$$\sum_i \big| [ \log Df^N (a_{i+1}) - \log Df^N (a_i)] + [ \log Df^N (h(b_{i+1})) - \log Df^N (h(b_i))] + [\log Dh (b_{i+1}) - \log Dh (b_i)] \big|.$$
\end{small}where $b_j := f^N (a_j)$. By the triangle inequality, the difference between this expression (and hence of (\ref{cal-dist})) and 
\begin{small}
$$\sum_i \big|  [ \log Df^N (h(b_{i+1})) - \log Df^N (h(b_i))] - [ \log Df^N (b_{i+1}) - \log Df^N (b_i)] + [\log Dh (b_{i+1}) - \log Dh (b_i)] \big|$$
\end{small}is at most 
$$\sum_i \big| \log Df^{2N} (a_{i+1}) - \log Df^{2N} (a_i) \big| \leq \var \big( \log Df^{2N}; [p_{-N}, p_{-N+1}] \big).$$
Since the latter quantity converges to $\dist (f) = 0$ as $N$ goes to infinity, taking the supremum over all $a_i < a_{i+1}$ and passing to the limit in $N$, we conclude that 
\begin{small}\begin{eqnarray*}
\dist (g) 
&=& \lim_{N \to \infty} \var \big( \log Df^N \circ h - \log Df^N + \log Dh ; [p_0, p_1] \big) \\
&=& \lim_{N \to \infty} \var \big( \log D (f^N h f^{-N}) \circ f^{N} ; [p_{0}, p_{1}] \big) 
\,\,\, = \,\,\, \lim_{N \to \infty} \var \big( \log D (f^N h f^{-N})  ; [p_{N}, p_{N+1}] \big).
\end{eqnarray*}
\end{small}Therefore, the fact that $\dist (g) > 0$ directly follows from the preceding lemma.
$\hfill\square$ 

\vspace{0.1cm}

\begin{rem} It is not hard to see that any two diffeomorphisms $f,g$ in $\mathrm{Diff}_+^{2,\Delta} ([0,1])$ as in the previous lemma are conjugate by a bi-Lipschitz homeomorphism (see \cite[Theorem 3.6.14]{libro} for an idea of proof of this fact; again, this mostly relies on Kopell's work). Hence, Theorem \ref{t:invariant} has no extension in this direction.
\end{rem}

%%%%%%%%%%%%%%%%%%%%%%%%%%%%%%%%%%%%%%%%%%%%%%%%%%%%%%%%%%%%%%%%%%%%%%
%%%%%%%%%%%%%%%%%%%%%%%%%%%%%%%%%%%%%%%%%%%%%%%%%%%%%%%%%%%%%%%%%%%%%%

\section{Back to the ``conjugacy problem''}
\label{s:reste}

As we mentioned in the Introduction, the asymptotic distortion was introduced in \cite{mio2} to deal with the classical ``conjugacy problem'' in $C^{1+\mathrm{bv}}$ regularity, more precisely, to determine when the conjugacy class of a given diffeomorphism contains the identity in its closure. This problem is perhaps even more relevant for \emph{actions} rather than for single diffeomorphisms, notably in relation to {\em foliation theory}, where the role of the group is played by the holonomy pseudogroup. In this context, conjugacies become a natural tool to produce deformations of foliations. Let us recall in this direction that it is a longstanding open problem whether the space of codimension-1 foliations on a given manifold is locally path connected or not.

Concerning group actions on 1-dimensional manifolds, we can mention that the proof in 
\cite{mio1} of the path-connectedness of the space of $\Z^k$ actions by $C^1$ orientation-preserving diffeomorphisms of a compact 1-manifold involves conjugacies in a crucial way. In particular, it is proved therein that if \emph{finitely many} commuting diffeomorphisms $f_1, \ldots, f_k$ of $[0,1]$ are such that all their fixed points are parabolic, then there exists a sequence of $C^1$ diffeomorphisms $h_n$ of $[0,1]$ for which $h_n f_i h_n^{-1}$ converges to the identity in the $C^1$ topology for all $1 \leq i \leq k$ (and thus the same holds for $h_n g h_n^{-1}$ for every $g$ in the Abelian group generated by $f_1, \ldots, f_k$). 

The proof from \cite{mio1} crucially uses the finite generation hypothesis, and does not cover, for example, the case of 1-parameter families of commuting diffeomorphisms. Below we provide a first result in the latter direction, which reproves the previous statement in the case where one of the $f_i$'s has no fixed point in $(0,1)$. Notice that these results hold in $C^1$ regularity, as opposed to the rest of the paper, where the involved regularities are mainly $C^{1+\mathrm{bv}}$ and $C^2$. However, this will serve as a motivation for later considering the case of flows and vector fields. We will finally go back to the $C^2$ conjugacy problem for single diffeomorphisms, namely Question \ref{q:rapproch-C2} from the Introduction, and prove Theorem \ref{t:rapproch-C2}.

\vspace{0.2cm}

\begin{prop} \label{calculo-general}
If $f \in \mathrm{Diff}_+^{1,\Delta} ([0,1])$ is tangent to the identity at the endpoints, then there exists a sequence of $C^1$ diffeomorphisms $h_n$ of $[0,1]$ such that, for each $C^1$ diffeomorphism $g$ of $[0,1]$ that commutes with $f$, the sequence $h_n g h_n^{-1}$ converges to the identity in the $C^1$ topology.
\end{prop}

\begin{proof} 
Following \S \ref{s:dist}, as in (\ref{def-h_n}) we let $h_n$ be the $C^1$ diffeomorphism of $[0,1]$ whose derivative equals 
$$D h_n (x) := \frac{ \big[ \prod_{i=0}^{n-1} Df^i (x) \big]^{1/n} }{\int_{0}^1 \big[ \prod_{i=0}^{n-1} Df^i (t) \big]^{1/n} dt}.$$ 
As in the proof of Proposition \ref{doble-criterio}, letting $y := h_n (x)$, we compute
\begin{small}
\begin{eqnarray*}
D (h_n g h_n^{-1}) (y)
&=&
\frac{Dh_n (g (x))}{D h_n (x)}  \cdot D g (x) \\
&=&
\left[ \frac{ Df (g(x)) \cdot Df^2 (g(x)) \cdots Df^{n-1} (g(x)) }{Df (x) \cdot Df^2 (x) \cdots Df^{n-1} (x)} \right]^{1/n} Dg (x)\\
&=&
\left[ \frac{ Dg(x) \cdot Df (g(x)) Dg(x) \cdot Df^2 (g(x)) Dg(x) \cdots Df^{n-1} (g(x)) Dg (x) }{Df (x) \cdot Df^2 (x) \cdots Df^{n-1} (x)} \right]^{1/n} \\
&=&
\left[ \frac{ Dg (x) \cdot D (f g) (x) \cdot D (f^2 g) (x) \cdots D (f^{n-1} g) (x) }{Df (x) \cdot Df^2 (x) \cdots Df^{n-1} (x)} \right]^{1/n} \\
&=&
\left[ \frac{ Dg (x) \cdot D (g f) (x) \cdot D (g f^2) (x) \cdots D (g f^{n-1}) (x) }{Df (x) \cdot Df^2 (x) \cdots Df^{n-1} (x)} \right]^{1/n} \\
&=&
\left[ \frac{ Dg (x) \cdot Dg(f(x)) Df (x) \cdot D g (f^2 (x)) Df^2 (x) \cdots D g (f^{n-1} (x)) Df^{n-1}(x) }{Df (x) \cdot Df^2 (x) \cdots Df^{n-1} (x)} \right]^{1/n} \\
&=&
\left[ Dg (x) \cdot Dg (f(x)) \cdot Dg (f^2(x)) \cdots Dg (f^{n-1}(x)) \right]^{1/n}.
\end{eqnarray*}
\end{small}Thus,
$$\log \big( D (h_n g h_n^{-1}) (y) \big) = \frac{1}{n} \sum_{i=0}^{n-1} \log \big( Dg (f^i h_n^{-1}(y)) \big).$$

The invariant probability measures of $g$ are concentrated at the set of its fixed points. Therefore, if we show that all these points are parabolic, then this will ensure that expression 
$$\frac{1}{n} \sum_{i=0}^{n-1} \log (Dg (f^i (z)))$$
uniformly converges to zero as $n$ goes to infinity, which in view of the previous computation will imply the announced convergence to the identity of the conjugates. 

Thus, in order to close the proof, it suffices to show that every $C^1$ diffeomorphism $g$ that commutes with $f$ is $C^1$ tangent to the identity at each endpoint. To prove this, first assume that $g$ has a fixed point $a \in (0,1)$. Then, by commutativity, all the points $a_n := f^n(a)$, with $n \in \mathbb{Z}$, are fixed by $g$, with the same value for the derivative. Changing $f$ by $f^{-1}$ if necessary so that $f(x) > x$ for all $x \in (0,1)$, we obtain
$$Dg (0) = \lim_{n \to \infty} \frac{g (a_{-n}) - g (0)}{a_{-n} - 0} = 1 \qquad \mbox{and} \qquad Dg (1) = \lim_{n \to \infty} \frac{g (1)- g (a_n)}{1 - a_n} = 1.$$
This also implies that $Dg (a) = 1$, otherwise we would have $Dg (a_n) = Dg (a) \neq 1$ for all $n \in \mathbb{Z}$, and since the sequence $a_n$ accumulate at both endpoints, this would be in contradiction with $Dg(0)=Dg(1)=1.$

Suppose now for a contradiction that $Dg (0) \neq 1$, say $Dg (0) > 1$ after changing $g$ by $g^{-1}$ if necessary (the case where $Dg(1) \neq 1$ can be treated in a similar way). Given $n \geq 1$ we have $Dg (0) > 1 = Df^n (0)$, hence $g(x) > f^n (x)$ for all small-enough $x > 0$. However, if we fix $a \in (0,1)$, there is $N \in \mathbb{N}$ such that $f^N (a) > g (a)$. This implies that $f^{-N} g$ has a fixed point in $(0,a)$ for small-enough $a>0$. Since this map commutes with $f$, the case just settled establishes that $1 = D (f^{-N} g) (0) = D g (0)$, which is a contradiction.
\end{proof}

The previous proposition applies to the flow of a $C^1$ vector field that is nonvanishing at interior points and $C^1$ flat at the endpoints. Below we prove a related result for this case under a mild extra hypothesis.

\vspace{0.1cm}

\begin{prop} \label{champ-c-1-conj}
Let $X$ be a $C^1$ vector field on $[0,1]$ vanishing only at the endpoints and 
inducing a flow of diffeomorphisms whose fixed points are all parabolic. If the flow of $X$ contains a nonidentity $C^2$ diffeomorphism $f$, then there exists a sequence of $C^2$ diffeomorphisms $h_n \! : [0,1] \to [0,1]$ such that the ``conjugate'' vector fields
$$ X_n (x) := (D h_n \cdot X) \circ h_n^{-1} (x)$$
converge to zero in the $C^1$ topology.
\end{prop}

Some comments before passing to the proof. A direct consequence of the proposition is that, if $f^t$ denotes the flow of $X$, then the conjugate diffeomorphisms $h_n f^t h_n^{-1}$ converge to the identity in the $C^1$ topology. By the choice of $h_n$ below, 
the convergence of $h_n f h_n^{-1}$ will actually happen in the $C^{1+\mathrm{bv}}$ topology. However, we do not know whether the $C^2$ convergence holds; otherwise, we would have a positive answer for Question \ref{q:rapproch-C2}, and the $C^1$ convergence of $X_n$ to the zero vector field would directly follow from the continuity -- in a broad sense -- of the Szekeres vector field with respect to the associated diffeomorphism (cf.~\cite{yoccoz}). In proposition~\ref{t:champ-C2}, we obtain such a $C^2$ convergence under the assumption that the initial vector field is itself $C^2$ (this is shown via a completely different approach). 

\vspace{0.1cm}

\begin{proof} 
For all $x \in [0,1]$ we have
\begin{equation}\label{eq-varrho}
X (f(x)) = Df (x) \cdot X (x).
\end{equation}
Let again $h_n$ be the $C^2$ diffeomorphism defined by (\ref{def-h_n}). We claim that this sequence $h_n$ satisfies the desired property. Indeed, let us denote 
\begin{equation}\label{defn}
X_n (x) := D h_n (h_n^{-1} (x)) \cdot X (h_n^{-1} (x)).
\end{equation}
Since $X_n (0) = X_n (1) = 0$, we only need to check that $D X_n$ uniformly converges to zero. To do this, we compute:
\begin{eqnarray*}
D X_n (x) 
&=&
D^2 h_n (h_n^{-1} (x)) \cdot D h_n^{-1} (x) \cdot X (h_n^{-1}(x)) + Dh_n (h_n^{-1} (x)) \cdot D X (h_n^{-1}(x)) \cdot Dh_n^{-1} (x) \\
&=& 
\frac{D^2 h_n (h_n^{-1}(x))}{D h_n (h_n^{-1}(x))} \cdot X (h_n^{-1} (x)) + D X (h_n^{-1} (x)),
\end{eqnarray*}
hence
$$D X_n (h_n (x)) = \frac{D^2 h_n (x)}{D h_n (x)} \cdot X (x) + D X (x).$$
Now, a straightforward computation starting from (\ref{def-h_n}) yields 
$$\frac{D^2 h_n (x)}{D h_n (x)} = \frac{1}{n} \sum_{i=0}^{n-1} \frac{D^2 f^i (x)}{D f^i (x)},$$
thus
\begin{equation}\label{eq-inter}
D X_n (h_n (x)) = \frac{X (x)}{n} \sum_{i=0}^{n-1} \frac{D^2 f^i (x)}{D f^i (x)} + D X (x).
\end{equation}
Next, a repeated application of (\ref{eq-varrho}) yields $\, X (f^i (x)) = Df^i (x) \cdot X (x) \,$ for each $i \geq 0$. Taking derivatives, this gives
$$D X (f^i(x)) \cdot Df^i (x) = D^2 f^i (x) \cdot X (x) + D f^i (x) \cdot D X (x),$$
hence
$$D X (f^i (x)) - D X (x) = X (x) \cdot \frac{D^2 f^i (x)}{D f^i (x)}.$$
Introducing this equality into (\ref{eq-inter}), we obtain
$$D X_n (h_n (x)) 
= \frac{1}{n} \sum_{i=0}^{n-1} \big[ D X (f^i (x)) - D X (x) \big] + D X (x) 
= \frac{1}{n} \sum_{i=0}^{n-1} D X ( f^i (x) ).$$
The proof is concluded by noticing that, since $D X$ vanishes at each fixed point of $f$, 
the last expression uniformly converges to zero. (This follows by a classical argument of localization of the invariant probability 
measures at the set of fixed points for interval homeomorphisms, already used in the proof of the previous proposition; 
see also relation (1.2) in \cite{Polt-Sod}.)
\end{proof}

\vspace{0.01cm}

\begin{qsintro} 
Is the hypothesis of the existence of a nonidentity $C^2$ diffeomorphism $f$ in the flow of $X$ necessary for the validity of the previous proposition~?
\end{qsintro}

\vspace{0.15cm}

We now go back to the conjugacy problem for a single diffeomorphism in $C^2$ regularity, namely Question \ref{q:rapproch-C2}. The following result gives a first partial answer: if $f\in\Diff^{2,\Delta}_{+}([0,1])$ has parabolic fixed points and is the time-$1$ map of a $C^2$ vector field, then it does contain the identity in the closure of its conjugacy class. Combined with the subsequent Proposition~\ref{p:C1toC2}, 
this implies Theorem \ref{t:rapproch-C2}.

\vspace{0.1cm}

\begin{prop}
\label{t:champ-C2}
Assume that $f \in \Diff^{r,\Delta}_{+}([0,1])$ has parabolic fixed points and is the time-$1$ map of a $C^r$ vector field $X$, where $r \geq 2$. Then there exists a sequence of $C^r$ vector fields  $X_n$ that converge to the zero vector field in the $C^r$ topology and whose time-$1$ maps $f_n$ (which thus $C^r$-converge to the identity) are $C^r$-conjugate to $f$.
\end{prop}

\begin{proof} 
For every $n\in\N$, let $h_n$ be a $C^r$ diffeomorphism of $[0,1]$ that coincides with the homothety of ratio $n$ on $[0,\frac1{3n}]$ and $[1-\frac1{3n},1]$, and let $g_n := h_n\circ f \circ h_n^{-1}$. We will construct a $C^r$ vector field $X_n$ of small $C^r$ norm (going to $0$ as $n$ goes to infinity) that coincides with the Szekeres vector field $Y_n$ of $g_n$ near $0$ and $1$. By Lemma \ref{l:Cr-conj}, its time-$1$ map $\hat{g}_n$ will be $C^r$-conjugated to $g_n$, and thus to $f$, as required. 

On $[0,\frac14]\cup[\frac34,1]$, we define $X_n$ by letting $DX_n := \rho_1 \, DY_n$, where $\rho_1$ is a smooth function equal to $1$ near the endpoints (on $[0,\frac18]\cup[\frac78,1]$, for concreteness), to $0$ on $[\frac14,\frac34]$, and strictly monotonous on each of the two remaining intervals. 

We claim that $X_n$ does not vanish on $(0,\frac14]\cup[\frac34,1)$. Indeed, by symmetry, it is enough to verify this on $(0,\frac14]$; actually, it suffices to consider the interval $[\frac18,\frac14]$, since $X_n=Y_n$ on $(0,\frac18]$ and $Y_n$ does not vanish on $(0,1)$. Now, for $x\in[\frac18,\frac14]$,
\begin{small}\begin{equation}\label{x-ene}
X_n(x)=X_n(\tfrac18)+\int_{\tfrac18}^x \rho_1 \, DY_n 
= X_n(\tfrac18) + \left[\rho_1 Y_n\right]_{\tfrac18}^x -\int_{\tfrac18}^x Y_n \, D \rho_1 
= (\rho_1Y_n)(x)-\int_{\tfrac18}^x Y_n \, D \rho_1.
\end{equation}
\end{small}Since $\rho_1$ is decreasing on $(\frac18,\frac14)$, letting $\mu := \min_{[1/8, 1/4]} Y_n > 0$, we have on this interval $Y_n \, D \rho_1 \leq \mu \, D \rho_1$. This gives, for all $x \in [\frac18,\frac14]$, 
$$-\int_{\frac18}^x Y_n \, D\rho_1 \geq -\mu \, (\rho_1(x)-0),$$ 
with equality if and only if $Y_n$ is constant equal to $\mu$ on $(\frac18,x)$. If we introduce this into (\ref{x-ene}), we obtain 
$$X_n(x) \geq \rho_1(x)(Y_n(x)-\mu) \geq 0.$$ 
Moreover, the fist inequality above is strict, and hence $Y_n (x)$ is strictly positive, unless $Y_n$ is constant on $(\frac18, x)$, in which case $X_n$ is also constant on $(\frac{1}{8},x)$, hence $X_n(x)=X_n(\frac18)>0$ as well. 

We now let $u_n := X_n(\frac14)$ and $v_n := X_n(\frac34)$ (which, by the discussion above, are positive). On $[\frac14,\frac34]$, we interpolate between $u_n$ and $v_n$ by letting $X_n := \rho_2u_n+(1-\rho_2)v_n$, where $\rho_2$ is a smooth step function equal to $1$ on $[0,\frac14]$ and to $0$ on $[\frac34,1]$. 

We finally check the convergence of $X_n$ towards the zero vector field in the $C^r$ topology. To do this, since $X_n (0) = X_n (1) = 0$, it is enough to show that $D^{r} X_n$ uniformly converges to $0$ on $[0,1]$. On $J:=[0,\frac14]\cup [\frac34,1]$, we have 
$$D^{r} X_n = D^{r-1} (\rho_1 \, DY_n) = \sum_{k=0}^{r-1}\binom{r-1}{k} \, D^{r-1-k}\rho_1 \, D^{k+1}Y_n.$$ 
Now notice that, for $x \in J$, 
$$D^{k+1}Y_n(x) = \frac1{n^k} \, D^{k+1}X_f (y),$$  
where \, $y = x/n$ \, for \, $x \in [0,\frac14]$ \,  and \, $y = 1 - (1-x)/n$ \, for \, $x \in [\frac34 ,1]$. \,  Thus, 
$$\sup_{J}|D^r X_n|
\le \sum_{k=0}^{r-1}\binom{r-1}{k} \sup_J |D^{r-1-k}\rho_1|\cdot \frac1{n^k}\sup_{[0,\frac1{4n}]\cup [1-\frac1{4n},1]}|D^{k+1}X|,$$
and each term in this sum converges to zero as $n$ goes to infinity. (This is obvious for $k > 0$, whereas for $k= 0$ it follows from the fact that $DX(0) = DX (1) = 0$.)

Finally, notice that $\, u_n \!=\! \int_0^{\frac14}DX_n(t)dt \,$ goes to $0$ as $n$ goes to $\infty$ (this immediately follows from the previous discussion), and so does $v_n$. By construction, this implies that $\sup_{[\frac14,\frac34]} |D^r X_n| \xrightarrow[n\to+\infty]{}0,$ thus closing the proof.
\end{proof}

\vspace{0.1cm}
The next result is closely related to Proposition \ref{champ-c-1-conj}. In terms of vector fields, it might seem weaker since here the ``conjugate'' vector fields converge (in the  $C^1$ topology) towards {\em some} $C^2$ vector field rather than simply the zero vector field. However, there is a stronger part in the statement concerning the convergence of the conjugated flow maps: under the extra hypothesis that the times yielding to $C^2$ flow maps form a dense subgroup of $\mathbb{R}$ (and not just a nonempty subset as in Proposition \ref{champ-c-1-conj}), the conjugates of these flow 
maps converge \emph{in the $C^2$ topology}.    

\vspace{0.1cm}

\begin{prop}
\label{p:C1toC2} 
Let $X$ be a $C^1$ vector field on $[0,1]$ with flow $f^t$. Suppose that the set of times $t$ for which $f^t$ is a $C^2$ 
diffeomorphism is a dense subgroup of $\R$. Then there exists a sequence of $C^2$ conjugacies $h_n$ such that $(h_n)_*X$ converges 
in the $C^1$ sense towards a $C^2$ vector field $Y$ and that, for every $C^2$ diffeomorphism $f^\tau$ of the flow of $X$, the conjugate 
map $h_n\circ f^\tau\circ h_n^{-1}$ converges in the $C^2$ sense towards the time-$\tau$ map of~$Y$.
\end{prop}

\vspace{0.1cm}

The proof of this proposition is strongly inspired by a result from \cite{mio2}. The role played there by the Weil Equidistribution Theorem is played here by the following elementary lemma:

\vspace{0.1cm}

\begin{lem}Let $\phi \!: t \in [0,1]\mapsto \phi_t\in C^0([0,1],\R)$ be a (uniformly) continuous map, and let $\tau_n$ be a sequence of positive numbers converging to $0$. If we denote $k_n$ the integral part of $\frac1{\tau_n} + 1$, then the sequence 
$\frac1{k_n}\sum_{i=0}^{k_n-1}\phi_{i\tau_n}$ converges uniformly to $\int_0^1\phi_t(\cdot) \, dt$. 
\end{lem}

\begin{proof} 
Given $\epsilon>0$, let $\eta>0$ be such that $\|\phi_t-\phi_s\|<\epsilon$ for every $t,s\in[0,1]$ such that $|t-s|<\eta$. Let $n\in\N$ be such that $1/k_n<\eta$. If $0 \leq i \leq k_n-1$, then $i\tau_n$ belongs to $I_i := [i/k_n,(i+1)/k_n]$. Therefore, for every $x\in[0,1]$,
\begin{small}
$$\left|\frac1{k_n}\sum_{i=0}^{k_n-1}\phi_{i\tau_n}(x)-\int_0^1\phi_t(x)\,dt\right|
=\left| \sum_{i=0}^{k_n-1}\int_{I_i}\left(\phi_{i\tau_n}(x)-\phi_t(x)\right)\,dt\right|\\
\le \sum_{i=0}^{k_n-1} \int_{I_i}\left| \phi_{i\tau_n}(x)-\phi_t(x)\right|\,dt\le \epsilon,$$
\end{small}which gives the desired convergence.
\end{proof}

\vspace{0.1cm}

\begin{proof}[Proof of Proposition \ref{p:C1toC2}] 
Considering a multiple $\lambda X$ of $X$ (with $\lambda\in\R$) if necessary, we may 
assume that $f = f^1$ is a $C^2$ diffeomorphism. Consider a sequence $\tau_n$ of nonzero times converging to $0$ for which $f^{\tau_n}$ is a $C^2$ diffeomorphism. Letting $k_n-1$ be the integral part of $1/\tau_n$, we have $0<\tau_n<2\tau_n<\dots<(k_n-1)\tau_n<1$. For every $n\in\N$, we define $h_n$ by $h_n(0)=0$ and 
$$Dh_n(x) = \frac{\rho_n(x)}{\int_0^1 \rho_n(u)du}, \quad \text{where}\quad \rho_n(u) = \left(\prod_{i=0}^{k_n-1}Df^{i\tau_n}(u)\right)^{1/k_n}.$$
By definition, $\rho_n$ is $C^1$ and strictly positive; moreover, $h_n(1)=1$. Thus, $h_n$ is a $C^2$ diffeomorphism of $[0,1]$. 

\paragraph{On the $C^1$-convergence of $h_n$.} We first study the uniform convergence of 
$$\log Dh_n = \log \rho_n-\log (\textstyle{\int_0^1 \rho_n(u)du}) 
= \frac1{k_n}\sum_{i=0}^{k_n-1}\log Df^{i\tau_n} - \log \big( \int_0^1 \rho_n(u)du \big).$$
By Lemma 2 applied to $t\mapsto \log Df^t$, the first term on the right hand side converges uniformly towards 
$$x\mapsto \int_0^1 \log Df^t (x) dt=:\log \cP(x),$$
and thus the second term converges to 
$$\log\left(\int_0^1\exp\left(\int_0^1 \log Df^t(u)dt\right)du\right)= \log\left(\int_0^1\cP(u)du\right).$$

Now let us observe that the sequence $\log Dh_n$ is uniformly bounded. Indeed, since $X$ is $C^1$, there exists $C>0$ such that $\|\log Df^t\|_\infty\le C$ for every $t\in[0,1]$. In particular, for every $i\in\N$, we have $\|\log Df^{i\tau_n}\|_\infty\le C$, that is, $e^{-C}\le Df^{i\tau_n}\le e^C$. This easily implies $e^{-C}\le \rho_n\le e^C$, hence $e^{-2C}\le Dh_n\le e^{2C}$. 

By the previous remark, the uniform convergence of $\log Dh_n$ implies the convergence of $Dh_n$ towards the function $x\mapsto \frac{\cP(x)}{\int_0^1\cP(u)du}$.  Since this is positive and has integral 1 on $[0,1]$, its primitive $h$ vanishing at $0$ is a $C^1$ diffeomorphism of $[0,1]$. The same argument shows that $h_n^{-1}$ also converges in the $C^1$ sense. Necessarily, the limit is nothing but $h^{-1}$.\medskip

Notice that $h$ is $C^2$ on $(0,1)$. Indeed, on this interval, the equality $Df^t = X\circ f^t / X$ shows that $Df^t$ is $C^1$. Thus, $\log \cP$ is also $C^1$, and this yields the $C^2$ regularity of $h$.

\paragraph{On the $C^1$-convergence of $(h_n)_*X$ towards a $C^2$ vector field.} 
Let $X_n = (h_n)_*X$ (which, according to the above, converges in the $C^0$ topology towards $Y:=h_*X$). One has $X=h_n^*X_n = \frac{X_n\circ h_n}{Dh_n}$. Hence, on $(0,1)$, 
$$\log X = \log (X_n\circ h_n)-\log Dh_n.$$ 
Taking derivatives, this yields
$$\frac{DX}X=\left(\frac{DX_n}{X_n}\right)\circ h_n\cdot Dh_n - \frac{D^2 h_n}{D h_n}.$$ 
Multiplying by $X$ on both sides of this equality and using the equality $X \cdot Dh_n= X_n\circ h_n$, we obtain
\begin{equation}
DX_n\circ h_n = DX + X \cdot \frac{D^2 h_n}{D h_n}.
\end{equation}
Since $h_n^{-1}$ converges, it is enough to study the uniform convergence of $DX + X \cdot \frac{D^2 h_n}{D h_n}$. Now, since
$$\frac{D^2 h_n}{D h_n} = \frac1{k_n}\sum_{i=0}^{k_n-1} \frac{D^2 f^{i\tau_n}}{D f^{i\tau_n}} =  \frac1{k_n}\sum_{i=0}^{k_n-1}\frac{DX\circ f^{i\tau_n} -DX}X,$$ 
we have
\begin{align*}
DX + X \cdot \frac{D^2 h_n}{D h_n} &=\frac1{k_n}\sum_{i=0}^{k_n-1}DX\circ f^{i\tau_n}. 
\end{align*}
According to Lemma 2, the last expression uniformly converges towards 
$$x\mapsto \int_0^1DX(f^t(x))dt =\int_0^1\frac{DX}{X}(f^t(x)) \cdot X(f^t(x))dt,$$
and the change of variables $u=f^t(x)$ transforms this into
$$\int_x^{f^1(x)}\frac{DX}{X}(u)du =\log X(f(x))-\log X(x) = \log \Big( \frac{X(f(x)}{X(x)} \Big) = \log Df(x).$$
Therefore, $DX_n$ uniformly converges to $(\log Df)\circ h^{-1}$, which is a $C^1$ function, hence $X_n$ $C^1$-converges towards $x\mapsto \int_0^x\log Df\circ h^{-1}(u)du$, which is $C^2$. 

Let us now check that this limit is indeed $h_*X = (Dh\circ h^{-1})(X\circ h^{-1})$ or, equivalently, that, on $(0,1)$, 
$$(\log Df)\circ h^{-1} = D((Dh\cdot X)\circ h^{-1}).$$
To do this, notice that, since 
$$D((Dh\cdot X)\circ h^{-1})=\frac{D^2h\cdot X+Dh\cdot DX}{Dh}\circ h^{-1} = \left( X \cdot \frac{D^2 h}{D h} + DX \right)\circ h^{-1},$$ 
we simply need to check that \, $\log Df = X \cdot \frac{D^2 h}{D h} + DX$, \, and this follows from
\begin{small}
\begin{align*}
X(x)\cdot \frac{D^2 h}{D h} (x) 
&= X(x)\cdot D\log\cP(x) 
= \int_0^1X(x)\cdot \frac{D^2 f^t(x)}{D f^t (x)} \, dt \\
&= \! \int_0^1 \!\! \big( DX(f^t(x))-DX(x) \big)\, dt
 = \! \int_0^1 \!\! DX(f^t(x))\, dt -DX(x)
=\log Df(x)-DX(x),
\end{align*}
\end{small}where we have used a previous computation in the last step.

\paragraph{On the $C^2$-convergence of $f_n^\tau =h_n\circ f^\tau \circ h_n^{-1}$ if $f^\tau $ is $C^2$.} 
Let $\tau\in\R$ be such that $f^\tau $ is $C^2$. Given the $C^1$-convergence of $h_n$ and $h_n^{-1}$, we already have the $C^1$-convergence of $f_n^\tau $ towards the time-$t$ map $h\circ f^\tau \circ h^{-1}$ of $Y=h_*X$. Hence, it suffices to check the convergence of $\frac{D^2 f_n^\tau} {D f_n^{\tau}}  
= \frac{D^2 (h_n\circ f^\tau \circ h_n^{-1})}{D (h_n\circ f^\tau \circ h_n^{-1})}$. 
To do this, first notice that
$$Df_n^\tau  (x) 
= \big( Dh_n(f^\tau \circ h_n^{-1})\cdot Df^\tau (h_n^{-1})\cdot Dh_n^{-1} \big) (x) 
= \frac{Dh_n(f^\tau (y))}{Dh_n(y)}\cdot Df^\tau (y),$$
where $y=h_n^{-1}(x)$. Hence, by the chain rule, 
\begin{align*}
Df_n^\tau  (x) 
&=Df^\tau (y)\cdot \left(\frac{\prod_{i=0}^{k_n-1}Df^{i\tau_n}(f^\tau (y))}{\prod_{i=0}^{k_n-1}Df^{i\tau_n}(y)}\right)^{1/k_n}\\
&= \prod_{i=0}^{k_n-1}\left(\frac{Df^{i\tau_n+\textcolor{black}{\tau}}(y)}{Df^{i\tau_n}(y)}\right)^{1/k_n}
=\prod_{i=0}^{k_n-1} \big( Df^\tau (f^{i\tau_n}(y) \big)^{1/k_n}
=\prod_{i=0}^{k_n-1} \big( Df^\tau (f^{i\tau_n}\circ h_n^{-1}(x)) \big)^{1/k_n}.
\end{align*}
Thus,
\begin{footnotesize}
\begin{align*}
\frac{D^2 f_n^\tau}{D f_n^{\tau}}  (x)&= \frac1{k_n}\sum_{i=0}^{{k_n}-1}D(\log Df^\tau \circ f^{i\tau_n}\circ h_n^{-1})(x) 
= \frac1{k_n}\sum_{i=0}^{{k_n}-1} \left(\frac{D^2 f^\tau}{D f^{\tau}} (f^{i\tau_n}\circ h_n^{-1})\cdot Df^{i\tau_n}\circ h_n^{-1}\cdot Dh_n^{-1}\right)(x)\\
&=Dh_n^{-1}(x)\cdot \frac1{k_n}\sum_{i=0}^{{k_n}-1} \frac{D^2 f^\tau}{D f^{\tau}} (f^{i\tau_n}(h_n^{-1}(x)))\cdot Df^{i\tau_n}(h_n^{-1}(x)).
\end{align*}
\end{footnotesize}Again, thanks to Lemma 2, the latter expression uniformly converges towards 
$$Dh^{-1} (x) \cdot \int_0^1 \frac{D^2 f^\tau}{D f^{\tau}} (f^t(h^{-1}(x)))\cdot Df^t(h^{-1}(x)) \, dt,$$
Since this is a continuous function in $x$, this closes the proof.
\end{proof} 

\vspace{0.1cm}

\begin{proof}[Proof of Theorem \ref{t:rapproch-C2}] Let $f \in \Diff^{2,\Delta}_{+}([0,1])$ have a $C^2$ centralizer larger than infinite cyclic. Then its left and right Szekeres vector fields coincide, say they are equal to a certain $C^1$ vector field $X$ 
(\cf \cite{yoccoz}). Moreover, the $C^2$ centralizer of $f$ corresponds to the times of $X$ that yield to $C^2$ diffeomorphisms. By assumption, the corresponding subgroup of $\R$ is not cyclic, so it is dense. Thus, $X$ satisfies the hypothesis of Proposition \ref{p:C1toC2}, which ensures the existence of $\tilde f\in \Diff^{2,\Delta}([0,1])$ lying in the $C^2$ closure of the conjugacy class of $f$ that is the time-$1$ map of a $C^2$ vector field. If $f$ has parabolic fixed points, then the same holds for $\tilde{f}$. By Proposition \ref{t:champ-C2}, for each $\epsilon>0$ there exists $h\in \Diff^{2}([0,1])$ such that $\|h\tilde f h^{-1}-\id\|_2<\epsilon/2$. Obviously, this yields the existence of $\delta>0$ such that 
$$\|\bar f - \tilde f\|_2<\delta \quad \Rightarrow \quad \|h \bar f h^{-1}-\id\|_2 < \epsilon.$$
Since $\tilde{f}$ lies in the $C^2$ closure of the conjugacy class of $f$, there is $\bar h$ 
such that $\|\bar h f \bar h^{-1}-\tilde f\|_2 < \delta.$ Putting everything together, we obtain
$$\big\| (h\bar h)f(h\bar h)^{-1}-\id \big\|_2 < \epsilon.$$
Since $\varepsilon > 0$ was arbitrary, this shows that the identity lies in the $C^2$ closure of the conjugacy class of $f$, 
which finishes the proof.
\end{proof}

%%%%%%%%%%%%%%%%%%%%%%%%%%%%%%%%%%%%%%%%%%%%%%%%%%%%%%%%%%%%%%%%%%%%%

\vspace{0.1cm}

\noindent{\bf Acknowledgments.}  Andr\'es Navas was funded by the projects FONDECYT 1200114 (in Chile) and 
FORDECYT 265667 and the PREI of the DGAPA at UNAM (in M\'exico). Both authors were funded by the ANR project GROMEOV 
(in France), and would like to thank Casa Matem\'atica Oaxaca for the nice atmosphere and excellent working conditions during 
the conference ``Ordered groups and rigidity in dynamics and topology'', where part of this work was conceived.

%%%%%%%%%%%%%%%%%%%%%%%%%%%%%%%%%%%%%%%%%%%%%%%%%%%%%%%%%%%%%%%%%%%%%%

\begin{footnotesize}

\vspace{0.1cm}

\noindent H\'el\`ene Eynard-Bontemps \hfill{Andr\'es Navas}

\noindent  IMJ - PRG, Sorbonne Universit\'e (Paris) \hfill{Dpto de Matem\'atica y C.C., U. de Santiago de Chile}

\noindent Institut Fourier, Universit\'e Grenoble Alpes \hfill{Unidad Cuernavaca Inst. de Matem\'aticas, UNAM}

\noindent Email address: helene.eynard-bontemps@imj-prg.fr \hfill{Email address: andres.navas@usach.cl}

%\vspace{0.3cm}

%\noindent Andr\'es Navas\\ 

%\noindent Dpto de Matem\'atica y C.C., Univ. de Santiago de Chile\\ 

%\noindent Unidad Cuernavaca Instituto de Matem\'aticas, UNAM\\

%\noindent Email address: andres.navas@usach.cl \\ 

\end{footnotesize}

\end{document}